\documentclass[a4paper, 11pt, twoside, leqno]{article}

\usepackage[T1]{fontenc}
\usepackage[utf8x]{inputenc}
\usepackage[english]{babel}
\usepackage{lmodern}               
\usepackage{amsmath,amsthm,amssymb}
\usepackage{mathrsfs,esint,bbm,stmaryrd}
\usepackage{enumitem,graphicx,xcolor}
\usepackage[textwidth=16cm,centering,headheight=24pt]{geometry}
\usepackage{authblk}

\newtheorem{theorem}{Theorem}           
\newtheorem{corollary}[theorem]{Corollary}
\newtheorem{lemma}[theorem]{Lemma}
\newtheorem{prop}[theorem]{Proposition}

\theoremstyle{definition}              

\theoremstyle{remark}                  
\newtheorem{step}{Step}

\newtheorem{remark}{Remark}

\DeclareMathOperator{\Id}{Id}                                       
\DeclareMathOperator{\dist}{dist}                                   
\DeclareMathOperator{\sign}{sign}                                   
\DeclareMathOperator{\spt}{spt}                                     
\DeclareMathOperator{\curl}{\curl}                                  

\DeclareMathOperator{\BV}{BV}
\DeclareMathOperator{\SBV}{SBV}

\newcommand{\abs}[1]{\left| #1 \right|}                             
\newcommand{\norm}[1]{\left\| #1 \right\|}                          
\newcommand{\mres}
{\mathbin{\vrule height 1.6ex depth 0pt width 0.13ex\vrule height 0.13ex depth 0pt width 1.3ex}}
\newcommand{\csubset}{\subset\!\subset}                             

\DeclareMathAlphabet{\mathpzc}{OT1}{pzc}{m}{it}

\newcommand{\D}{\mathrm{D}}       
\newcommand{\T}{\mathrm{T}}
\renewcommand{\d}{\mathrm{d}}
\newcommand{\J}{\mathrm{J}}
\newcommand{\N}{\mathbb{N}}       
\newcommand{\R}{\mathbb{R}}
\newcommand{\Z}{\mathbb{Z}}

\newcommand{\M}{\mathbb{M}}
\newcommand{\F}{\mathbb{F}}
\newcommand{\I}{\mathbb{I}}
\renewcommand{\SS}{\mathbb{S}}
\newcommand{\G}{\mathbf{G}}       
\renewcommand{\P}{\mathbb{P}}

\newcommand{\NN}{\mathscr{N}}     

\newcommand{\EE}{\mathscr{E}}

\renewcommand{\H}{\mathscr{H}}
\renewcommand{\L}{\mathscr{L}}

\SetSymbolFont{stmry}{bold}{U}{stmry}{m}{n}  
\newcommand{\nnu}{{\boldsymbol{\nu}}}

\newcommand{\PR}{\mathbb{R}\mathrm{P}^2}  

\newcommand{\RR}{\varrho}
\newcommand{\X}{\mathscr{X}}

\renewcommand{\S}{\mathbf{S}}

\definecolor{lightblue}{rgb}{0.22,0.45,0.70}   
\definecolor{darkgray}{gray}{0.4}    
\definecolor{lightgray}{gray}{0.8}
\newcommand{\BBB}{}

\title
{Topological singular set of vector-valued maps, I: \\
Applications to manifold-constrained Sobolev and BV spaces}

\author{Giacomo Canevari\thanks{BCAM --- Basque Center for Applied Mathematics,
Alameda de Mazarredo 14, 48009 Bilbao, Spain. \\
\emph{E-mail address}: \texttt{gcanevari@bcamath.org}} \mbox{ }and 
Giandomenico Orlandi\thanks{Dipartimento di Informatica --- Universit\`a di Verona,
Strada le Grazie 15, 37134 Verona, Italy. \\
\emph{E-mail address}: \texttt{giandomenico.orlandi@univr.it}}}

\date{\today}

\begin{document}

\maketitle

\begin{abstract}
 We introduce an operator~$\S$ on vector-valued maps~$u$ which has the ability to capture
 the relevant topological information carried by~$u$.
 In particular, this operator is defined on maps that take values in a closed submanifold~$\NN$ 
 of the Euclidean space~$\R^m$, and coincides with the distributional Jacobian 
 in case~$\NN$ is a sphere. {\BBB More precisely, the range of~$\S$ is a set of maps
 whose values are} flat chains with coefficients in
 a suitable normed abelian group. In this paper, we use~$\S$ to characterise
 strong limits of smooth, $\NN$-valued maps with respect to Sobolev norms, 
 extending a result by Pakzad and Rivière. We also discuss applications to the study
 of manifold-valued maps of bounded variation. In a companion paper, we will consider 
 applications to the asymptotic behaviour of minimisers
 of Ginzburg-Landau type functionals,  with $\NN$-well potentials.

 
\end{abstract}

\section{Introduction}


Let~$\NN$ be a smooth, closed Riemannian manifold, isometrically embedded in 
a Euclidean space~$\R^m$, and let $\Omega\subseteq\R^d$ be a bounded,
smooth domain of dimension $d\geq 2$. Functional spaces of maps $u\colon\Omega\to\NN$
(e.g., Sobolev or BV) have been extensively studied in the literature, in connection
with manifold-constrained variational problems, in order to detect 
the topological information encoded by~$u$.

In this paper, instead of dealing directly with $\NN$-valued maps,
we consider \emph{vector-valued} maps $u\colon\Omega\to\R^m$, which we think of as 
approximations of a map $v\colon\Omega\to\NN$. 
This point of view also arises quite naturally from
variational problems, such as the penalised harmonic map problem, the Ginzburg-Landau model for 
superconductivity or other models from material science that share a common structure,
e.g. the Landau-de Gennes model for nematic liquid crystals. Moreover,
working with vector-valued, instead of manifold-valued, maps allows for
more flexibility. On the other hand, if~$u\colon\Omega\to\R^m$ does
not take values uniformly close to~$\NN$ but only close in, say, an integral sense
(e.g. $\int_\Omega\dist(u, \, \NN)$ is small) then it might not be
obvious to extract the topological information carried by~$u$.
For instance, in the Ginzburg-Landau theory, this task is accomplished by 
means of the distributional Jacobian. However, this tool is only available
when the distinguished manifold~$\NN$ has a special structure
--- typically, when $\NN$ is a sphere --- and cannot be applied to some
cases that are relevant to applications, for instance, 
when~$\NN$ is a real projective plane~$\R\mathrm{P}^2$, as is the case
in many models for liquid crystals.

The goal of this paper is to define an operator, $\S$,
such that~$\S(u)$ corresponds to the set of topological singularities of~$u$ and
plays the r\^ole
of a ``generalised Jacobian'', which can be applied to more general 
target manifolds~$\NN$. The properties of~$\S$
are stated in our main result, Theorem~\ref{th:Stop} 
in Section~\ref{subsect:theorem} below.
As the distributional Jacobian,
this operator captures topological information and enjoys
compactness properties, and in fact it reduces to the distributional 
Jacobian in the special case~$\NN\simeq\SS^{n}$.
The construction of~$\S$ is carried out in the setting of
flat chains with coefficients in a normed abelian group.
This approach has been proposed by Pakzad and Rivi\`ere~\cite{PakzadRiviere},
in the context of manifold-valued maps, in order to characterise
strong limits of smooth $\NN$-valued maps in $W^{1, p}(B^d, \, \NN)$.
Because we are interested in vector-valued maps, our construction 
is different from theirs, and relies on the ``projection trick''
devised by Hardt, Kinderlehrer and Lin~\cite{HKL} 
(see also~\cite{Hajlasz, BousquetPonceVanSchaftingen}).
Eventually, we  generalise Pakzad and Rivi\`ere's main result 
to a broader range of values for the exponent~$p$,
see Theorem~\ref{th:PR} in Section~\ref{subsect:intro-applications}.

In this paper, we discuss some applications of the operator~$\S$ to the study
of manifold-valued functional spaces. 
In addition to the aforementioned generalisation of 
the result by Pakzad and Rivi\`ere (Theorem~\ref{th:PR}), 
we study manifold-valued spaces of functions of bounded variation. We 
show weak density of smooth maps in~$\BV(\Omega, \, \NN)$,
see Theorem~\ref{th:BV_density} in Section~\ref{subsect:intro-applications}, 
thus generalising a result by Giaquinta and Mucci~\cite{GiaquintaMucci}.
We also discuss the lifting problem in~BV (see, for instance, \cite{DavilaIgnat})
for a larger class of manifolds~$\NN$, see Theorem~\ref{th:BV_lifting}
in Section~\ref{subsect:intro-applications}.
{\BBB Further applications 
to variational problems, including the asymptotic behaviour
of the Landau-de Gennes model for liquid crystals,
will be investigated in forthcoming work~\cite{CO2}.}
As is the case for the distributional Jacobian
in the Ginzburg-Landau theory, we expect that~$\S$ 
might be used to identify the set where the energy concentrates
and characterise the limiting energy densities.

The plan of the paper is the following. After recalling some background 
in Section~\ref{subsect:intro-background}, we sketch our construction in 
Section~\ref{subsect:intro-construction}, and we present the statements
of Theorems~\ref{th:PR}, \ref{th:BV_density}, \ref{th:BV_lifting} 
in Section~\ref{subsect:intro-applications}.
In Section~\ref{sect:preliminaries}, we review some preliminary material
about flat chains 
(Section~\ref{subsect:flat_chains}), topology 
(Sections~\ref{subsect:group_norm}--\ref{subsect:RR}), and manifold-valued Sobolev spaces 
(Section~\ref{subsect:manifold-Sobolev}). The main technical result
of this paper, Theorem~\ref{th:Stop}, which gives the existence of the operator~$\S$,
is stated in Section~\ref{subsect:theorem}. The rest of Section~\ref{sect:Stop}
is devoted to the proof of Theorem~\ref{th:Stop} and of Theorem~\ref{th:PR},
which we recover as a corollary of Theorem~\ref{th:Stop}.
Finally, Section~\ref{sect:BV} contains the applications to 
manifold-valued BV spaces, with the proofs of Theorem~\ref{th:BV_density}
and~\ref{th:BV_lifting}.

\subsection{Background and motivation}
\label{subsect:intro-background}

For the sake of motivation, consider the Ginzburg-Landau functional:
\begin{equation} \label{GLenergy}
 u\in W^{1,2}(\Omega,\, \R^2) \mapsto E^{\mathrm{GL}}_\varepsilon(u):= \int_{\Omega}
 \left\{ \frac{1}{2}\abs{\nabla u}^2 + \frac{1}{4\varepsilon^2} (1 - |u|^2)^2 \right\} \! ,
\end{equation}
where~
$\varepsilon >0$ is a small parameter. Functionals of this form arise as variational models 
for the study of type-II superconductivity. In this context, $u(x)$ represents the 
magnetisation vector at a point $x\in\Omega$ and the energy favours configurations 
with $|u(x)| = 1$, which have a well-defined direction of magnetisation 
as opposed to the non-superconducting phase $u=0$.
Let $\SS^1$ denote the unit circle in the plane $\R^2$.
As is well known, minimisers~$u_\varepsilon$ subject to a
($\varepsilon$-independent) boundary condition 
$u_{\varepsilon|\partial\Omega} = u_{\mathrm{bd}}\in W^{1/2, 2}(\partial\Omega, \, \SS^1)$
satisfy the sharp energy bound $E_\varepsilon(u_\varepsilon)\leq C|\log\varepsilon|$
for some $\varepsilon$-independent constant~$C$ (see e.g.~\cite[Proposition~2.1]{Riviere-DenseSubsets}).
In particular, $u_\varepsilon$ takes values ``close'' to~$\SS^1$ when
$\varepsilon$ is small, in the sense that
$\int_\Omega (1 - |u_\varepsilon|^2)^2\leq C\varepsilon^2|\log\varepsilon|$.
Despite the lack of uniform energy bounds, under suitable conditions on~$u_{\mathrm{bd}}$,
minimisers $u_\varepsilon$ converge to a limit map~$u_0\colon\Omega\to\SS^1$, which is
smooth except for a singular set of codimension two
(see e.g.~\cite{BBH, LinRiviere, BethuelBrezisOrlandi, JerrardSoner-GL, ABO2, 
SandierSerfaty, BethuelOrlandiSmets-Annals}). 
Moreover, the singular set of~$u_0$ is itself a minimiser --- in a suitable sense --- 
of some ``weighted area'' functional~\cite{ABO2}. 
The emergence of singularities in the limit map~$u_0$
is related to topological obstructions, which may prevent the existence of a map
in $W^{1, 2}(\Omega, \,  \SS^1)$ that satisfies the boundary conditions. 

There are other functionals, arising as variational models for material science, which share
a common structure with~\eqref{energy}, i.e. they can be written in the form
\begin{equation} \label{energy}
 u\in W^{1,k}(\Omega,\, \R^m) \mapsto E_\varepsilon(u):= \int_{\Omega}
 \left\{ \frac{1}{k}\abs{\nabla u}^k + \frac{1}{\varepsilon^2} f(u) \right\} \! .
\end{equation}
Here $f\colon\R^m\to\R$ is a non-negative, smooth potential that satisfies suitable 
coercivity and non-degeneracy conditions, and $\NN := f^{-1}(0)$ is assumed to be a non-empty,
smoothly embedded, compact, connected submanifold of~$\R^m$ without boundary.
The elements of~$\NN$ correspond to the ground states for the material, i.e.
the local configurations that are most energetically convenient. 
An important example is the Landau-de Gennes model for nematic liquid crystals
(in the so-called one-constant approximation, see e.g.~\cite{deGennes}). In this case, $k=2$ and
the distinguished manifold is a real projective plane~$\NN = \PR$, whose elements
describe the locally preferred direction of alignment of the constituent molecules (which 
might be schematically described as un-oriented rods).

As in the Ginzburg-Landau case, topological obstructions may imply the
\emph{lack} of an extension operator 
$W^{1-1/k, k}(\partial\Omega, \, \NN)\to W^{1, k}(\Omega, \, \NN)$
(see for instance \cite{Bethuel-Extension}). As a consequence, 
minimisers~$u_\varepsilon$ subject to a Dirichlet boundary condition
$u_\varepsilon = u_{\mathrm{bd}}\in W^{1-1/k, k}(\partial\Omega, \, \NN)$ may not satisfy uniform 
energy bounds with respect to~$\varepsilon$.
Compactness results in the spirit of the Ginzburg-Landau theory 
have been shown for minimisers of the Landau-de Gennes
functional~\cite{MajumdarZarnescu, pirla, GolovatyMontero, pirla3}.
However, some points that are understood in the Ginzburg-Landau theory 
--- for instance, a variational characterisation of the singular set of the limit
or a description of the problem in terms of~$\Gamma$-convergence,
as in~\cite{JerrardSoner-GL, ABO2, AlicandroPonsiglione} ---
are still missing, even for the Landau-de Gennes functional.

A key tool in the analysis of the Ginzburg-Landau functional is the distributional Jacobian.
In case~$d=m=2$, the distributional Jacobian~$\mathrm{J}u$ 
of a map $u\in (L^\infty\cap W^{1,1})(\R^2, \, \R^2)$ is defined as the
distributional curl of the field
$\frac{1}{2}(u^1\partial_1u^2-u^2\partial_1 u^1, \, u^1\partial_2 u^2-u^1\partial_2 u^1)$.
Equivalently, in the language of differential forms,
$\mathrm{J}u := \star\d u^*\omega_{\SS^1}$, 
where~$\star$ denotes the Hodge duality operator 
and~$\omega_{\SS^1}(y) := \frac{1}{2}(y^1\d y^2 - y^2\d y^1)$
is the $1$-homogeneous extension of the renormalised volume form on~$\SS^1$.
The purpose of the distibutional Jacobian is two-fold:
on one hand, it captures topological information associated with~$u$, as is 
demonstrated by several formulas relating the Jacobian with the topological degree 
(see e.g.~\cite[Theorem~0.8]{BourgainBrezisMironescu2005});
on the other hand, it enjoys compactness properties
--- for instance, despite being a quadratic operator, it is stable under
weak $W^{1,2}$-convergence.
Unfortunately, an adequate notion of Jacobian may be missing for general manifolds~$\NN$.
Consider the following simple example: let~$S$ be a $(d-k)$-plane in~$\R^d$, and
let~$u\colon\Omega\setminus S\to\NN$ be a material
configuration that is smooth everywhere, except at~$S$.
Then~$S$ can be encircled by a $(k-1)$-dimensional
sphere~$\Sigma\subseteq\Omega\setminus S$, and the (based) homotopy class of
$u_{|\Sigma}\colon\Sigma\to\NN$ defines an element of~$\pi_{k-1}(\NN)$
which, roughly speaking, characterises the behaviour of the material around the defect.
(This is the basic idea of the topological classification of defects
in ordered materials; see e.g.~\cite{Mermin} for more details.)
If $\pi_{k-1}(\NN)$ contains elements of finite order, these cannot be realised 
via integration of a differential form, so no notion of Jacobian that can be 
expressed as a differential form is able to capture such homotopy classes of defects. 
An example is provided by the Landau-de Gennes model for nematic liquid crystals, where
$k=2$, $\NN\simeq\PR$ and~$\pi_1(\PR)\simeq\Z/2\Z$.

The aim of this paper is to construct an object that
(i) brings topological information and (ii) enjoys compactness properties even 
when the distributional Jacobian is not defined, in particular when~$\pi_{k-1}(\NN)$
contains elements of finite order. A notion of ``set of topological singularities'' 
for a manifold-valued Sobolev map was already introduced by
Pakzad and Rivi\`ere~\cite{PakzadRiviere}, using the language of flat chains. 
Roughly speaking, a flat chain of dimension~$n$
with coefficients in an abelian group~$\G$ is described by a collection of
$n$-dimensional sets, carrying multiplicities that are elements of~$\G$
(see~\cite{FedererFleming, Fleming}). The group of flat $n$-chains with coefficients
in~$\G$ can be given a norm, called the flat norm, which satisfies useful compactness
properties. Given integers numbers $2 \leq k \leq d$ 
and~$u\in W^{1, k-1}(B^d, \, \NN)$, the topological singular set of~$u$  
`\`a la Pakzad-Rivi\`ere' is a flat chain $\S^{\mathrm{PR}}(u)$ of dimension~$(d-k)$
with coefficients in~$\pi_{k - 1}(\NN)$, and has the following property: 
$u$ can be $W^{1,k-1}$-strongly approximated by smooth maps~$\Omega\to\NN$ if and only
if~$\S^{\mathrm{PR}}(u)= 0$ \cite[Theorem~II]{PakzadRiviere}.
The construction we carry out here is different (and relies on ideas from \cite{HKL}),
as we want to deal with vector-valued maps~$u\colon\Omega\to\R^m$
instead of manifold-valued ones.
However, following Pakzad and Rivi\`ere, we work in the formalism of flat chains.
We discuss the link between Pakzad and Rivi\`ere's construction and the one
presented here in Section~\ref{subsect:Sobolev-NN}.

\subsection{Sketch of the construction}
\label{subsect:intro-construction}

Throughout the paper, $d$, $m$, $k$ will be integer numbers with $\min\{d, \, m\}\geq k\geq 2$,
$\Omega$ will be a smooth, bounded domain in~$\R^d$, and $\NN$ will denote
a smooth submanifold of~$\R^m$ without boundary. We make the following assumption
on~$\NN$ and~$k$:

\begin{enumerate}[label=(H), ref=\textrm{H}]
 \item \label{H} $\NN$ is compact and $(k-2)$-connected,
 that is $\pi_0(\NN) = \pi_{1}(\NN) = \ldots = \pi_{k-2}(\NN) = 0$. In case~$k=2$,
 we also assume that~$\pi_1(\NN)$ is abelian.
\end{enumerate}

The integer~$k$ is thus related to the topology of~$\NN$, and represents the
codimension of the (highest-dimensional) topological singularities for $\NN$-valued maps.
{\BBB 
The dimension of~$\NN$ plays no explicit r\^ole in our construction, hence
it is not specified.}
Under the assumption~\eqref{H}, the group~$\pi_{k-1}(\NN)$ is abelian, and will be the 
coefficient group for our flat chains. As noted above, $\pi_{k-1}(\NN)$ classifies 
the topological defects of $\NN$-valued maps. 
We will endow~$\pi_{k-1}(\NN)$ with a norm, see Section~\ref{subsect:group_norm}.

The construction we carry out has been introduced by Hardt, Kinderlehrer and 
Lin~\cite{HKL} as a method to produce manifold-valued comparison maps with suitable 
properties.
{\BBB This approach has been used by Haj\l{}asz~\cite{Hajlasz}, 
to prove strong (resp., sequential weak)
density of smooth maps in $W^{1,p}(\Omega, \, \NN)$ in case~$\NN$ is 
$\lfloor p \rfloor$-connected (resp., $(\lfloor p \rfloor-1)$-connected).
It has also been used by Bousquet, Ponce 
and Van Schaftingen~\cite{BousquetPonceVanSchaftingen}, who extended 
Haj\l{}asz's results to the setting of fractional Sobolev spaces~$W^{s,p}$ with~$s\geq 1$.
We sketch now the main ideas of our construction.}

It is impossible to construct a smooth projection 
of~$\R^n$ onto a closed manifold~$\NN$. However, as noted by Hardt and
Lin~\cite[Lemma~6.1]{HardtLin-Minimizing}, under the assumption~\eqref{H}
it is possible to construct a smooth projection 
$\RR\colon\R^m\setminus\X\to\NN$, where~$\X$ is a union of~$(m-k)$-manifolds.
Given a smooth map~$u\colon\R^d\to\R^m$, one could identify the 
set of topological singularities of~$u$ with $u^{-1}(\X)$, which is exactly the set
where the reprojection $\RR\circ u$ fails to be well-defined, but $u^{-1}(\X)$ may be 
very irregular even if~$u$ is smooth. However, Thom transversality theorem implies that,
for a.e. $y\in\R^m$, the set $(u-y)^{-1}(\X)$ is indeed a union of
$(d-k)$-dimensional manifolds. This set can be equipped, in a natural way,
with multiplicities in~$\pi_{k-1}(\NN)$,
so to define a flat chain~$\S_y(u)$ of dimension~$d-k$.
Thus, we define the set of topological singularities of~$u$ as a
map~$y\in\R^m\mapsto\S_y(u)$ with values in the group of flat chains.

By integrating over~$y\in\R^m$ according to the strategy devised in~\cite{HKL},
and applying the coarea formula, one obtains estimates on~$\S_y(u)$ depending on the Sobolev
norms of~$u$. Then, by density, one can define~$\S_y(u)$ in case~$u$ is a Sobolev map,
thus obtaining an operator
\[
 \S\colon (L^\infty\cap W^{1,k-1})(\Omega,\, \R^m) \to
 L^1(\R^m; \, \F_{d-k}(\overline\Omega; \, \pi_{k-1}(\NN)) )
\]
Here $\F_{d-k}(\Omega; \, \pi_{k-1}(\NN))$ denotes the normed $\pi_{k-1}(\NN)$-module
of~$(d-k)$-dimensional flat chains in~$\Omega$
with coefficients in~$\pi_{k-1}(\NN)$ (see Section~\ref{subsect:relative_chains}), and
$L^1(\R^m; \, \F_{d-k}(\Omega; \, \pi_{k-1}(\NN))) =: Y$
is the set of Lebesgue-measurable maps
$S\colon\R^m \to \F_{d-k}(\Omega; \, \pi_{k-1}(\NN))$ such that
\begin{equation}\label{integral_flat}
 \norm{S}_Y := \int_{\R^m} \F_\Omega(S_y) \, \d y < +\infty
\end{equation}
($\F_\Omega$ being the natural norm on $\F_{d-k}(\Omega; \, \pi_{k-1}(\NN))$, see 
Section~\ref{subsect:relative_chains}). In general, $Y$ is not a vector space but
it is a $\pi_{k-1}(\NN)$-module, and the left hand side of~\eqref{integral_flat} 
defines a norm on~$Y$. The operator~$\S$ is continuous in the following 
sense: if $(u_j)_{j\in\N}$ is a sequence of maps such that~$u_j\to u$ strongly 
in~$W^{1,k-1}$ and~$\sup_j\|u_j\|_{L^\infty} < +\infty$, then
$\|\S(u_j) - \S(u)\|_Y \to 0$. The same remains true if the sequence~$(u_j)_{j\in\N}$
is assumed to converge only \emph{weakly} in~$W^{1,k}$ and to be uniformly
bounded in~$L^\infty$; therefore, some of the compensation compactness properties 
that are typical of the Jacobian are retained by~$\S$. Moreover, $\S$ carries
topological information on the map~$u$. Indeed, the intersection (in a suitable sense:
see Section~\ref{subsect:intersection}) between~$\S_y(u)$ and, say,
a $k$-disk~$R$ completely determines the homotopy class of~$\RR\circ(u-y)$ on~$\partial R$.
A precise statement of these properties, which requires some notation, is given in 
Theorem~\ref{th:Stop}.

In the special case~$\NN=\SS^{k-1}$ (the unit sphere in $\R^k$), $\X = \{0\}\subseteq\R^k$ 
and~$\RR\colon\R^k\setminus\{0\}\to\SS^{k-1}$ is the radial projection given 
by~$\RR(y) = y/|y|$,
we have~$\pi_{k-1}(\SS^{k-1})\simeq\Z$ and so elements of
$\F_{d-k}(\Omega; \, \pi_{k-1}(\SS^{k-1}))$ have an alternative description as
integer currents. Moreover, $\S_y(u)$ is related to the distributional Jacobian, as 
for any~$u\in (L^\infty\cap W^{1,k-1})(\Omega, \, \R^k)$ there holds
\begin{equation} \label{Jacobian_integral}
 \mathrm{J} u = \frac{1}{\omega_k}\int_{\R^m} \S_y(u) \, \d y,
\end{equation}
where~$\omega_k$ is the volume of the unit $k$-disk and the integral in the right-hand side
is intended in the sense of distributions (see e.g. \cite[Theorem~1.2]{JerrardSoner-Jacobians}).
However, if $\pi_{k-1}(\NN)$ is a finite group (or, more generally, if it only contains 
elements of finite order), then there is no meaningful way to define the integral of~$\S_y(u)$
with respect to the 
Lebesgue measure~$\d y$, as $\pi_{k-1}(\NN)\otimes\R = 0$.

It is worth noticing that the proof of our main result, Theorem~\ref{th:Stop},
does not strictly rely upon the manifold structure of~$\NN$. What is needed, 
is the existence and regularity 
of the exceptional set~$\X$ and the retraction~$\RR$, in order to be 
able to apply Thom transversality theorem. 
This suggests a possible extension to more general targets~$\NN\subseteq\R^m$
such as, for instance, finite simplicial complexes.

\subsection{Applications}
\label{subsect:intro-applications}

We have chosen to work with vector-valued maps, instead of manifold-valued ones, as 
we were motivated by the applications to variational problems, such as~\eqref{energy}.
We expect that the results presented in this paper could be used as tools to obtain
energy lower bounds for~\eqref{energy} in the spirit of~\cite{Sandier, Jerrard}, or even
$\Gamma$-convergence results along the lines of~\cite{ABO2}.
These questions will be addressed in a forthcoming work~\cite{CO2}. Instead, 
we discuss here a few applications of this approach to classical questions
in the theory of manifold-valued function spaces. 

The first application concerns density of smooth maps.
We define $W^{1,p}(B^d, \,\NN)$ as the set of maps~$u\in W^{1,p}(B^d, \,\R^m)$
such that~$u(x)\in\NN$ for a.e.~$x\in\Omega$, and endow it with the distance induced
by~$W^{1,p}(B^d, \,\R^m)$. Bethuel~\cite{Bethuel-Density} showed that smooth maps
are dense in~$W^{1,p}(B^d, \,\NN)$ if and only if~$\pi_{\lfloor p\rfloor}(\NN)=0$
or~$p\geq d$. Maps that belong to the strong-$W^{1,p}$ closure of
$C^\infty(\overline{B}^{k}, \, \SS^{k-1})$
have been characterised in~\cite{Bethuel1990}, in case~$p=k-1$, 
and in~\cite{BethuelCoronDemengelHelein}, in case $k - 1 < p < k$, using 
the distributional Jacobian. Pakzad and Rivi\`ere~\cite[Theorem~II]{PakzadRiviere}
generalised this result to other target manifolds, working in the setting of flat chains.
As a corollary of our construction, we recover Pakzad and Rivi\`ere's result.

\begin{theorem} \label{th:PR}
 Let~$d\geq 2$ be an integer, let~$1 \leq p < d$, and let $\NN$ be a compact, smooth, 
 $(\lfloor p\rfloor-1)$-connected manifold without boundary. In case $1\leq p < 2$, 
 we also suppose that $\pi_1(\NN)$ is abelian. Then, there exists a continuous map
 \[
  \S^{\mathrm{PR}}\colon W^{1, p}(B^d, \, \NN)\to
  \F_{d - \lfloor p \rfloor - 1}(\overline{B}^d; \, \pi_{\lfloor p\rfloor}(\NN))
 \]
 such that~$\S^{\mathrm{PR}}(u) =0$ if and only if~$u$ is a 
 strong~$W^{1,p}$-limit of smooth maps $\overline{B}^d\to\NN$.
\end{theorem}

In contrast with Pakzad and Rivi\`ere, we do not need to impose the technical restriction
$\lfloor p\rfloor\in\{1, \, d-1\}$. The arguments in~\cite{PakzadRiviere} rely on fine
results in Geometric Measure Theory~\cite{GiaquintaModicaSoucek-I} 
(which require $\lfloor p\rfloor\in\{1, \, d-1\}$); instead, the proof
of Theorem~\ref{th:PR} follows directly from our main construction, which is based
essentially on the coarea formula, combined with the 
``removal of the singularities'' results in~\cite{PakzadRiviere}.
It is worth mentioning that the theorem may fail if the domain is not a disk
(see the counterexamples in~\cite{HangLin} and the discussion in~\cite{PakzadRiviere}).

We next drive our attention to manifold-valued $\BV$-maps.
Recall that the space~$\BV(\Omega, \, \R^m)$, by definition, consists of those
functions~$u\in L^1(\Omega, \, \R^m)$ whose distributional derivative~$\D u$ is a finite 
Radon measure. The BV-norm is defined by $\|u\|_{\BV(\Omega)} := \|u\|_{L^1(\Omega)} 
+ |\D u|(\Omega)$, where $|\cdot|$ denotes the total variation measure.
We say that $u\in\SBV(\Omega, \, \R^m)$ if there exist
Borel functions $\psi_0$, $\psi_1\colon\Omega\to\R^{m\times d}$
such that $\psi_j$ is $\H^{d-j}$-integrable, for $j\in\{0, \, 1\}$, and 
$\D u = \psi_0\H^d + \psi_1\H^{d-1}$.
We say that a sequence~$u_j$ of~$\BV$-functions 
converges weakly to~$u$ if and only if~$u_j\to u$ strongly in~$L^1$
and~$\D u_j \rightharpoonup^* \D u$ weakly$^*$ as elements of the
dual~$C_0(\Omega, \, \R^m)^\prime$.
We define $\BV(\Omega, \, \NN)$ (resp., $\SBV(\Omega, \, \NN)$)
as the set of maps~$u\in\BV(\Omega, \, \R^m)$ (resp., $u\in\SBV(\Omega, \, \R^m)$)
such that~$u(x)\in\NN$ for a.e.~$x\in\Omega$.

\begin{theorem} \label{th:BV_density}
 Let~$\NN$ be a smooth, compact, connected manifold without boundary, with abe\-lian~$\pi_1(\NN)$.
 Then, $C^\infty(\overline{B}^d, \, \NN)$ is sequentially 
 weakly dense in~$\BV(B^d, \, \NN)$.
\end{theorem}

A similar result has been obtained by Giaquinta and Mucci~\cite[Theorem~2.13]{GiaquintaMucci},
who worked in the framework of currents (more precisely, in the class of cartesian currents,
see~\cite{GiaquintaModicaSoucek-I}).
Giaquinta and Mucci need the additional assumption that $\pi_1(\NN)$ 
contains no element of finite order, in order to apply the formalism of currents.
By working in the setting of flat chains, instead of currents, this assumption is not
required any more, although we still need that $\pi_1(\NN)$ be abelian.
In contrast with the scalar case, it may \emph{not} be possible
to construct approximating maps~$u_j\in C^\infty(\overline{B}^d, \, \NN)$ 
in such a way that $|\D u_j|(B^d)\to |\D u|(B^d)$ 
(see~\cite{GiaquintaMucci}).

The proofs of Theorems~\ref{th:PR} and~\ref{th:BV_density}
follow a strategy that was adopted by 
Bethuel, Brezis and Coron in~\cite{BethuelBrezisCoron-Relaxed}:
first we control the flat norm of the topological singular set, 
by means of the results in Section~\ref{sect:Stop},
then we ``remove the singularities'' using the results of~\cite{PakzadRiviere}.
The flat norm of the topological singular
set coincides with what Bethuel, Brezis and Coron referred to as ``minimal connection''.

Finally, we consider the \emph{lifting} problem in~$\BV$.
Let~$\pi\colon\EE\to\NN$ be the universal covering of~$\NN$. We choose a metric on~$\EE$
and an isometric embedding $\EE\hookrightarrow\R^\ell$ in such a way that
$\pi$ is a local isometry. We say that~$v\in\BV(\Omega, \, \EE)$ is a \emph{lifting}
for $u\in\BV(\Omega,\, \NN)$ if $u = \pi\circ v$ a.e. on~$\Omega$.

\begin{theorem}\label{th:BV_lifting}
 Let~$\Omega\subseteq\R^d$ be a smooth, bounded domain with~$d\geq 2$,
 and let~$\NN$ be a smooth, compact, connected manifold without boundary, with abelian~$\pi_1(\NN)$.
 There exists a constant~$C$ such that any~$u\in\BV(\Omega, \, \NN)$ admits
 a lifting~$v\in\BV(\Omega, \, \EE)$ satisfying $|\D v|(\Omega)\leq C|\D u|(\Omega)$.
 Moreover, if~$u\in\SBV(\Omega, \, \NN)$ then any lifting~$v$ of~$u$ belongs
 to~$\SBV(\Omega, \, \EE)$. 
\end{theorem}

The lifting problem in manifold-valued Sobolev spaces was studied 
by Bethuel and Chiron~\cite{BethuelChiron}, who proved that any map
$v\in W^{1, p}(\Omega, \, \NN)$ with $\Omega$ simply connected and~$p\geq 2$
has a lifting~$v\in W^{1, p}(\Omega,  \, \EE)$. (The particular case~$\NN\simeq\R\mathrm{P}^2$,
with applications to liquid crystals, was also studied by Ball and Zarnescu~\cite{BallZarnescu}).
As conjectured by Bethuel and Chiron~\cite[Remark~1]{BethuelChiron}, Theorem~\ref{th:BV_lifting}
implies that any map~$u\in W^{1,p}(\Omega, \, \NN)$, with~$p\geq 1$,
has a lifting~$v\in\BV(\Omega, \, \EE)$ 
(which may not belong to~$W^{1, p}$, see~\cite[Lemma~1]{BethuelChiron}).
The lifting problem in the space~$W^{s, p}(\Omega, \, \SS^1)$ has been extensively studied by 
Bourgain, Brezis, and Mironescu, see e.g.~\cite{BourgainBrezisMironescu, BourgainBrezisMironescu2005}. 
In the setting of BV-spaces, the lifting problem
has been previously studied by Davila and Ignat~\cite{DavilaIgnat},
Ignat~\cite{Ignat-Lifting} in case $\NN=\SS^1$, 
and recently by Ignat and Lamy~\cite{IgnatLamy}, in case $\NN=\R\mathrm{P}^{n}$. 
In contrast with Theorem~\ref{th:BV_lifting}, the results in~\cite{Ignat-Lifting, IgnatLamy}
are sharp, in the sense that they provide the optimal constant~$C$
such that $|\D v|(\Omega)\leq C|\D u|(\Omega)$; however, Theorem~\ref{th:BV_lifting} 
is robust, in that it applies to more general manifolds.
The proof of this theorem combines properties of the singular set~$\S_y(u)$ 
with a classical argument in topology, which gives the existence of the lifting 
for smooth functions~$u$, and which we revisit here in case the function~$u$ has jumps.

\goodbreak

\subsection{{\BBB Concluding remarks}}
{\BBB As remarked above, the techniques presented in this paper apply
to quite general target manifolds. 
While these methods capture effectively the 
topological singularities of the highest expected dimension (i.e., those of dimension~$d-k$),
a severe limitation is that they do not seem suitable to study
lower-dimensional topological singularities. For example, in case~$\NN=\PR$, $k=2$ and~$d=3$, 
a map $u\in W^{1,1}(B^3, \, \PR)$ may have both non-orientable line singularities 
(associated with $\pi_1(\PR)\simeq\Z/2\Z$) and point singularities,
associated with $\pi_2(\PR)$; these methods only provide information about the former ones.
Another example is the case~$\NN=\SS^2$,~$k=3$, $d=4$ and~$u\in W^{1,3}(B^4, \, \SS^2)$. Such a map~$u$
cannot have singularities of dimension~$d-k=1$, but it may have point singularities
which are not seen by the operator~$\S$ \cite{HardtRiviere}. These point singularities 
lead to the failure of sequential weak density of smooth maps in~$W^{1,3}(B^4, \, \SS^2)$,
as proved in a striking recent paper by Bethuel~\cite{Bethuel-Inventiones}. 

As shown in~\cite{PakzadRiviere} and in Theorem~\ref{th:PR}, the operator~$\S$
detects the local obstruction to approximability by smooth maps in manifold-valued
Sobolev spaces. At the current stage, it is not clear whether the ``projection approach''
could be used to detect the \emph{global} obstruction, introduced
in~\cite{HangLin}, as this would probably require a complete control 
of low-dimensional topological singularities.

Finally, another restriction lies in the choice of the target manifold.}
Closed manifolds~$\NN$ (or simplicial complexes) with non-abelian $\pi_1(\NN)$ are
excluded, because the theory of flat chains with coefficients
in a group~$\G$ requires~$\G$ to be abelian. 
However, in the topological obstruction theory,
this kind of restriction can be removed by using suitable technical tools 
(homology with local coefficients systems). This leaves a hope to extend, at least partially,
some of the results in this paper to the case of non-abelian $\pi_1(\NN)$. Density
(in the sense of biting convergence) of smooth maps in $W^{1, 1}(\Omega, \, \NN)$ with
non-abelian $\pi_1(\NN)$ has been proven by Pakzad~\cite{Pakzad-NonAbelian}.

\numberwithin{equation}{section}
\numberwithin{definition}{section}
\numberwithin{theorem}{section}
\numberwithin{remark}{section}

\section{Notation and preliminaries}
\label{sect:preliminaries}

\subsection{Flat chains over an abelian coefficient group} 
\label{subsect:flat_chains}

Let~$(\G, \, |\cdot|)$ be a normed abelian group, that is, an abelian group 
(we will use additive notation for the operation on~$\G$) together with
a non-negative function $|\cdot|\colon\G\to [0, \, +\infty)$ that satisfies
\begin{enumerate}[label=(\roman*)]
 \item $|g| = 0$ if and only if~$g=0$
 \item $|-g| = |g|$ for any~$g\in\G$
 \item $|g + h|\leq |g|+|h|$ for any~$g, \, h\in\G$.
\end{enumerate}
Throughout the following, we will assume that the norm~$|\cdot|$ satisfies
\begin{equation} \label{discrete_norm}
 |g| \geq 1 \qquad \textrm{for any } \G\setminus\{0\}.
\end{equation}
In order to fix some notation, and for the convenience of the reader, we 
recall some basic definitions and facts about flat chains with multiplicities in~$\G$.
We follow the approach in~\cite{Whitney-GIT, Fleming, White-Rectifiability},
to which we refer the reader for further details.

For~$n\in\Z$, $1\leq n\leq d$, consider the free $\G$-module generated by
compact, convex, oriented polyhedra of dimension~$n$ {\BBB in the ambient space~$\R^d$.} 
(In other words, we consider 
the set of all formal sums of polyhedra as above, with coefficients in~$\G$;
there is a natural notion of sum which makes this set an abelian group.)
We quotient this module by the equivalent relation~$\sim$, requiring
$-\sigma\sim\sigma^\prime$ if $\sigma^\prime$ and~$\sigma$ only differ for 
the orientation, and $\sigma\sim\sigma_1+\sigma_2$ if $\sigma$ is obtained by gluing
$\sigma_1$, $\sigma_2$ along a common face (with the correct orientation). 
The quotient group is called the group of polyhedral $n$-chain with
coefficients in~$\G$, and is denoted $\P_n(\R^d; \, \G)$.
Every element~$S\in\P_n(\R^d; \, \G)$ can be represented as a finite sum
\begin{equation} \label{polyhedral}
 S = \sum_{i = 1}^p \alpha_i\llbracket \sigma_i\rrbracket,
\end{equation}
where $\alpha_i\in\G$, the $\sigma_i$'s are compact, convex, non-overlapping 
$n$-dimensional polyhedra, and $\llbracket\cdot\rrbracket$
denotes the equivalence class modulo the relation~$\sim$ defined above.

The mass of a polyhedral chain~$S\in\P_n(\R^d; \, \G)$,
presented in the form~\eqref{polyhedral}, is defined by 
$\M(S) := \sum_i|\alpha_i|\H^n(\sigma_i)$.
A linear operator $\partial\colon\P_n(\R^d; \, \G)\to\P_{n-1}(\R^d; \, \G)$,
called the boundary operator,
is defined in such a way that, for a single polyhedron~$\sigma$,
$\partial\llbracket\sigma\rrbracket$ is the sum of the boundary faces of~$\sigma$,
with the orientation induced by $\sigma$ and multiplicity~$1$. 
The boundary operator satisfies $\partial\circ\partial =0$. 
The flat norm of a polyhedral $n$-dimensional chain~$S$ is defined by
\[
 \F(S) :=\inf\left\{\M(P)+\M(Q)\colon P\in\P_{n+1}(\R^d; \, \G), \ 
 Q\in\P_n(\R^d; \, \G), \ S = \partial P + Q\right\} \! .
\]
It can be shown (see e.g. \cite[Section~2]{Fleming})
that $\F$ indeed defines a norm on~$\P_n(\R^d; \, \G)$,
in such a way that the group operation on $\P_n(\R^d; \, \G)$ is Lipschitz continuous.
The completion of $(\P_n(\R^d; \, \G), \, \F)$, as a metric space, will be denoted
$\F_n(\R^d; \, \G)$. It can be given the structure of a $\G$-module,
and it is called the group of flat $n$-chain with coefficients in~$\G$. Moreover,
the mass~$\M$ extends to a $\F$-lower semi-continuous functional
$\F_n(\R^d; \, \G)\to[0, \, +\infty]$, still denoted~$\M$, and it remains true that
\begin{equation} \label{flat-mass}
 \F(S) :=\inf\left\{\M(P)+\M(Q)\colon P\in\F_{n+1}(\R^d; \, \G), \ 
 Q\in\F_n(\R^d; \, \G), \ S = \partial P + Q \right\}
\end{equation}
for any $S\in\F_n(\R^d; \, \G)$ \cite[Theorem~3.1]{Fleming}.
We let~$\M_n(\R^d; \, \G)$ be the set of
flat $n$-chains~$S$ with $\M(S) < +\infty$, and we let
\[
 \N_n(\R^d;\, \G) := \left\{S\in\M_n(\R^d;\, \G)\colon 
 \M(S) + \M(\partial S) <+\infty\right\} \! .
\]
In fact, $\M$ is a norm on $\M_n(\R^d; \, \G)$.

\paragraph{Operations with flat chains.}

Any Lipschitz map~$f\colon\R^d\to\R^N$ induces group homeomorphisms
$f_*\colon\F_n(\R^d; \, \G)\to\F_n(\R^N; \, \G)$, for $0 \leq n\leq d$,
called the push-forward via~$f$. One first defines the push-forward of a
a single polyhedron, $f_*\llbracket\sigma\rrbracket$, by approximating~$f$ with
piecewise-affine maps (see~\cite[p.~297]{Whitney-GIT}). Then, $f_*$ extends to
polyhedral chains by linearity, and to arbitrary chains by approximation
with polyhedral chains. The push-forward commutes with the boundary,
that is $\partial(f_*S) = f_*(\partial S)$. If $S$ is a flat $n$-chain 
and $\lambda$ is a Lipschitz constant for~$f$, then
\begin{equation} \label{pushforward}
 \M(f_*S) \leq\lambda^n \M(S), \qquad 
 \F(f_*S) \leq\max\{\lambda^n, \, \lambda^{n+1}\}\F(S) 
\end{equation}
(see e.g. \cite[Section~5]{Fleming}). A chain of the form~$f_*S$,
where $S$ is polyhedral and~$f$ is Lipschitz (resp., smooth), will be called a 
Lipschitz (resp., smooth) chain. By a remarkable result by Fleming~\cite{Fleming},
later improved by
White~\cite{White-Rectifiability}, if~$\G$ satisfies~\eqref{discrete_norm} then
Lipschitz chains are dense in $\M_n(\R^d; \, \G)$ with respect to the~$\M$-norm,
and in particular $\spt S$ is a rectifiable set for any~$S\in\M_n(\R^d; \, \G)$. 
(However, we will not need this result in our arguments.)

Given a chain~$S$ of finite mass and a Borel set~$A\subseteq\R^d$, one can define
the restriction of~$S$ to~$A$, denoted~$S\mres A$, which roughly speaking 
represents the portion of~$S$ contained in~$A$. Again, this is obtained via
approximation with polyhedral chains (see \cite[Section~4]{Fleming}). Then,
for fixed~$A$, the functional~$S\mapsto\M(S\mres A)$ is $\F$-lower continuous,
while for fixed~$S$, $A\mapsto\M(S\mres A)$ is a Radon measure.

A flat chain~$S$ is said to be supported in a closed set~$K\subseteq\R^d$
if, for any open neighbourhood~$U$ of~$K$,
there exists a sequence of polyhedral chains~$(P_j)_{j\in\N}$ that lie in~$U$
(i.e., every cell of~$P_i$ is contained in~$U$) and $\F$-converges to~$S$.
If~$S$, $R$ are supported in a closed set~$K$, 
then $\partial S$, $S+R$ are also supported in~$K$.
{\BBB The support of a chain~$S$, noted~$\spt S$, is defined 
by Fleming~\cite[Sections~3 and~4]{Fleming} as the smallest closed set~$K$
such that~$S$ is supported in~$K$. Fleming shows that the support of~$S$
exists if either (i) $S$ is supported in a compact set or~(ii)
if $S$ has finite mass. In the latter case, $\spt S$ 
coincides with the support of the measure $A\mapsto\M(S\mres A)$. However,
a more general definition of support can be given~\cite[Section~5]{Adams}, 
\cite[Section~4]{White-Notes} so to show that~$\spt S$ exists for any flat chain~$S$.}

For~$K$ closed set in~$\R^d$, we denote by~$\F_n(K; \, \G)$ 
(resp., $\M_n(K; \, \G)$, $\N_n(K; \, \G)$) the set of chains $S\in\F_n(\R^d; \, \G)$
(resp., $S\in\M_n(\R^d; \, \G)$, $S\in\N_n(\R^d; \, \G)$) that are supported in~$K$.
It follows from the definition of $\spt S$, and from the lower semi-continuity
of the mass, that the sets $\F_n(K; \, \G)$, $\M_n(K; \, \G)$, $\N_n(K; \, \G)$
are closed under $\F$-convergence.

Finally, we recall the following property of~$0$-dimensional flat chains.

\begin{lemma}[{\cite[Theorem~2.1]{White-Rectifiability}}] \label{lemma:augmentation}
 There exists a unique group homomorphism~$\chi\colon\F_0(\R^d; \, \G)\to\G$
 that satisfies the following properties:
 \begin{enumerate}[label=(\roman*)]
  \item $\chi(\sum_{j=1}^q g_j\llbracket x_j\rrbracket) = \sum_{j=1}^q g_j$ for~$g_j\in\G$ and~$x_j\in\R^d$, $j\in\{1, \, \ldots, \, q\}$.
  \item $\chi(\partial R) = 0$ for any~$R\in\F_1(\R^d; \, \G)$.
  \item  $|\chi(S)|\leq\F(S)$ for any~$S\in\F_0(\R^d; \, \G)$.
 \end{enumerate}
\end{lemma}

The map~$\chi$ is sometimes called the augmentation homomorphism.

\begin{remark} \label{remark:augmentation_stability}
 Lemma~\ref{lemma:augmentation}.(iii) and our assumption~\eqref{discrete_norm}
 imply that $\chi(S_0) = \chi(S_1)$ if the chains~$S_0$, $S_1\in\F_0(\R^d;\,\G)$
 are such that~$\F(S_0 - S_1) < 1$.
\end{remark}

\paragraph{Relative flat chains on an open set.}
\label{subsect:relative_chains}

In view of our applications, we will need to consider flat chains 
defined in an open set~$U\subseteq\R^d$.
A definition of the space~$\F_n(U; \, \G)$ is given in
several places in the literature (see, e.g., \cite{Federer, GiaquintaModicaSoucek-I,
PakzadRiviere}\ldots) but, to the best of the authors' knowledge, 
it is usually required that the elements of~$\F_n(U; \, \G)$
{\BBB are} compactly supported in~$U$, which is not convenient for our purposes.
We discuss here an alternative definition and present some basic results
for the sake of completeness, being aware that these facts might be
well-known by the experts of the field.

Let~$U\subsetneq\R^d$ be a non-empty open set, and let~$K$ be a closed set
that contains~$U$. Recall that we have defined $\F_n(K; \, \G)$ as the set of chains 
in $\F_n(\R^d; \, \G)$ that are supported in~$K$. We now define
\[
 \F_n(U; \, \G) := \F_n(K; \, \G)/\F_n(K\setminus U; \, \G).
\]
$\F_n(K\setminus U; \, \G)$ is a $\G$-submodule of $\F_n(K; \, \G)$
and is closed with respect to the $\F$-norm because~$K\setminus U$ is closed,
therefore $\F_n(U; \, \G)$ is a complete normed $\G$-module, with respect to the quotient norm:
\begin{equation} \label{quotient_norm}
 \F_U(S) := \inf\left\{\F(R)\colon R\in\F_n(\R^d; \, \G), \ 
 \spt(R)\subseteq K, \ \spt(R - S)\subseteq K\setminus U \right\}\! ,
\end{equation}
for $S\in\F_n(K; \, \G)$ --- by abuse of notation, we 
denote by the same symbol the chain~$S$ and its equivalence class
in~$\F_n(U; \, \G)$. The boundary operator~$\partial$
induces a well-defined, continuous operator $\F_n(U; \, \G)\to \F_{n-1}(U; \, \G)$,
still denoted~$\partial$.
We now give an alternative characterisation of the norm~$\F_U$.

\begin{lemma} \label{lemma:relative_flat}
 For any $S\in\F_n(K; \, \G)$, there holds
 \[
   \begin{split}
    \F_U(S) &= \inf\big\{\M(P\mres U) + \M(Q\mres U)\colon 
    P\in\M_{n+1}(\R^d; \, \G), \  Q\in\M_{n}(\R^d; \; \G), \\ 
    &\qquad\qquad\qquad \spt(S - \partial P - Q)\subseteq\R^d\setminus U  \big\}.
   \end{split}
 \]
\end{lemma}
\begin{proof}
 Denote by~$\tilde{\F}_U(S)$ the right-hand side.
 For any~$\varepsilon > 0$, using the definition~\eqref{quotient_norm} of~$\F_U$ and the
 characterisation~\eqref{flat-mass} of the flat norm, we find 
 $P\in\M_{n+1}(\R^d; \, \G)$, $Q\in\M_{n}(\R^d; \, \G)$
 such that $\spt(\partial P + Q)\subseteq K$,
 $\spt(\partial P + Q - S)\subseteq K\setminus U$ and
 $\M(P)+\M(Q)\leq\F_U(S)+\varepsilon$. This shows the inequality
 $\tilde{\F}_U(S)\leq\F_U(S)$.
 
 Before checking the opposite inequality, we remark that,
 for any chain~$T$ of finite mass and any open set~$W\subseteq\R^d$, there holds
 \begin{equation} \label{support}
  \spt(T - T\mres W) \subseteq \R^d\setminus W, \qquad 
  \spt(\partial T - \partial (T\mres W)) \subseteq \R^d\setminus W. 
 \end{equation}
 (The first inclusion holds true because $(T - T\mres W)\mres W = T\mres W - T\mres W = 0$;
 the second one follows from the first, because the boundary of chain
 supported in~$\R^d\setminus W$ is also supported in~$\R^d\setminus W$.)
 Now, we fix~$S\in\F_n(K; \, \G)$,
 $P\in\M_{n+1}(\R^d; \, \G)$ and $Q\in\M_{n}(\R^d; \; \G)$
 such that $\spt(S - \partial P - Q)\subseteq \R^d\setminus U$.
 Let~$K_0$ be the interior of~$K$, and let~$P^\prime := P\mres K_0$,
 $Q^\prime := Q\mres K_0$. Then~$P^\prime$, $Q^\prime$ are supported in~$K$,
 and so is $S - \partial P^\prime - Q^\prime$. Moreover, there holds
 \[
  S - \partial P^\prime - Q^\prime = \underbrace{S - \partial P - Q}
  \ + \ \underbrace{\partial P^\prime - \partial P} \ + \ \underbrace{Q^\prime - Q}
 \]
 and the three terms that are indicated by underbraces are all supported out of~$U$
 (the first one is supported in~$\R^d\setminus U$ by assumption, the second and
 the third ones are supported in $\R^d\setminus K_0\subseteq \R^d\setminus U$
 by~\eqref{support}). Therefore, 
 $\spt(S - \partial P^\prime - Q^\prime)\subseteq K\setminus U$.
 Finally, we have
 \[
  S = \underbrace{\partial (P^\prime\mres U) + Q^\prime\mres U}_{=:R} \ + \
  \underbrace{\partial P^\prime  - \partial(P^\prime\mres U)} \ + \
  \underbrace{Q^\prime - Q^\prime\mres U}\ + \ 
  \underbrace{S - \partial P^\prime - Q^\prime}
 \]
 The chain $R$ is supported in $\overline{U}\subseteq K$, and
 all the terms in the right-hand side but~$R$ are supported in~$K\subseteq U$,
 thanks to~\eqref{support}. Therefore, by the definition~\eqref{quotient_norm}
 of~$\F_U$ and~\eqref{flat-mass}, we deduce that
 $\F_U(S)\leq\F(R)\leq\M(P\mres U) + \M(Q\mres U)$ and hence,
 by arbitrarity of~$P$, $Q$, that $\F_U(S)\leq\tilde{\F}_U(S)$.
\end{proof}

The right-hand side of Lemma~\ref{lemma:relative_flat} do not depend on~$K$.
Therefore, the space~$\F_n(U; \, \G)$
is indeed independent of the choice of~$K$, in the following sense:
for any closed sets $K_1$, $K_2$ with $K_1\supseteq K_2\supseteq U$,
there exists an isometric isomorphism
\begin{equation*} 
 \F_n(K_1; \, \G)/\F_n(K_1\setminus U; \, \G) \to
 \F_n(K_2; \, \G)/\F_n(K_2\setminus U; \, \G).
\end{equation*}
The isomorphism is obtained by considering the map $\M_n(K_1; \, \G) \to
\F_n(K_2; \, \G)/\F_n(K_2\setminus U; \, \G)$ induced by the restriction 
operator $S\mapsto S\mres K_2$, extending it by density
to a map $\F_n(K_1; \, \G)\to\F_n(K_2; \, \G)/\F_n(K_2\setminus U; \, \G)$,
with the help of Lemma~\ref{lemma:relative_flat}, and passing to the quotient. 
The inverse is induced by the inclusion $K_2\hookrightarrow K_1$.
Because of the existence of an isomorphism, it makes sense to omit~$K$ 
in the notation. In the rest of this section, we assume 
that $K = \overline{U}$, but other choices of~$K$ might be convenient.

In a similar fashion, from Lemma~\ref{lemma:relative_flat} we can derive 
the following compatibility property 
with respect to restrictions.
For notational convenience, we set $\F_{\R^d} := \F$.

\begin{lemma} \label{lemma:flat_restriction}
 Let~$U_1$, $U_2$ be non-empty, open sets in~$\R^d$ with~$U_1\subseteq U_2$.
 Then, there exists a continuous map
 \[
  \Psi\colon\F_n(U_2; \, \G)\to\F_n(U_1; \, \G)
 \]
 such that $\Psi(S) = S\mres U_1$ for any $S\in\M_n(\overline{U}_2; \, \G)$.
\end{lemma}

We omit the proof of this lemma.
An analougous compatibility property with respect to restrictions
does \emph{not} hold, in general, for the $\F$-norm. 
(For instance, let $R_j\in\M_2(\R^2; \, \Z)$ be 
the chain carried by the rectangle $[-1/j, \, 1/j]\times[-1, \, 1]$ 
with standard orientation {\BBB and multiplicity~$1$;} then $\partial R_j$ $\F$-converges 
to zero but $\partial R_j\mres (0, \, +\infty)\times\R$ does not.) 
Moreover, if $S$ has infinite mass, the restriction $S\mres U$ might not be
well defined in~$\F_0(\R^d; \, \G)$:
for example, consider the $0$-chain with coeffients in~$\Z/2\Z$
carried by the set $\cup_{j\geq 1} \{(-2^{-j}, \, j), \, 
(2^{-j}, \, j)\}\subseteq\R^2$ and $U = (0, \, +\infty)\times\R$. 
In this case, $S\mres U$ is not well-defined
in $\F_0(\R^2; \, \Z/2\Z)$, even though $\Psi(S)$ is a well-defined
element of~$\F_0(U; \, \Z/2\Z)$ (i.e., there exists a chain
$R\in \F_0(\overline{U}; \, \G)$ such that $(S - R)\mres U = 0$).
However, the following statement holds: if
$S_j$ is a sequence of chians that $\F$-converges to~$S$, and if
$f\colon\R^d\to\R$ is a Lipschitz function, then for a.e.~$t\in\R$
the restrictions $S_j\mres f^{-1}(-\infty, \, t)$, 
$S\mres f^{-1}(-\infty, \, t)$ are well-defined and
$\F(S_j\mres f^{-1}(-\infty, \, t) - S\mres f^{-1}(-\infty, \, t))\to 0$
(see e.g. \cite[Theorem~5.2.3.(2)]{DePauwHardt};
we recall the proof in the following lemma).

\begin{lemma} \label{lemma:restriction}
 Let~$U\subsetneq\R^d$ be a non-empty open set,
 and let~$\rho_0$ be a positive number.
 For~$\rho\in(0, \, \rho_0]$, set~$:= \{x\in U\colon\dist(x, \, \partial U)>\rho\}$.
 Let $(S_j)_{j\in\N}$ be a sequence in~$\F_n(\overline{U}; \, \G)$
 and let~$S\in\F_n(\overline{U}; \, \G)$ be such that 
 $\F_U(S_j - S)\to 0$ as~$j\to+\infty$. Then, for a.e. $\rho\in (0, \, \rho_0]$
 the restrictions $S_j\mres U_\rho$, $S\mres U_\rho$ are well-defined and
 $\F(S_j\mres U_\rho - S\mres U_\rho)\to 0$ as~$j\to+\infty$.
\end{lemma}
\begin{proof}
 We first claim that, for any $T\in\F_n(\overline{U}; \, \G)$ and a.e. 
 $\rho\in (0, \, \rho_0]$, the restriction $T\mres U_\rho$ is well-defined
 and there holds
 \begin{equation} \label{restr0}
  \int_0^{\rho_0}\F(T\mres U_\rho) \, \d\rho \leq (1 + \rho_0) \F(T). 
 \end{equation}
 (As mentioned above, this fact is well-known and we include a proof here 
 only for the sake of completeness.) Suppose first that $T$ has finite mass.
 Let~$P\in\M_{n+1}(\R^d; \, \G)$, $Q\in\M_{n}(\R^d; \, \G)$ be such that
 $T = \partial P + Q$. Having assumed that~$T$ has finite mass, it 
 follows that~$\partial P$ has finite mass. We also remark that, for any~$\rho$,
 $U_\rho$ is a sublevel set for the signed distance function from~$\partial U$ 
 (i.e., the function~$f$ defined by $f(x) := -\dist(x, \, \partial U)$ 
 if~$x\in U$, $f(x) := \dist(x, \, \partial U)$ if~$x\notin U$),
 which is $1$-Lipschitz continuous. 
 Then, we can apply~\cite[Theorem~5.7]{Fleming}, 
 and deduce that $B_\rho := \partial(P\mres U_\rho) - (\partial P)\mres U_\rho$
 is well-defined, 
 and there holds
 \begin{equation} \label{restr1}
  \int_0^{+\infty} \M(B_\rho) \, \d\rho \leq \M(P\mres U).
 \end{equation}
 Moreover, {\BBB for a.e.~$\rho\in (0, \, \rho_0]$} there holds
 \[
  T\mres U_\rho =  \partial (P\mres U_\rho) + Q\mres U_\rho - B_\rho,
 \]
 which yields
 \[
  \F(T\mres U_\rho) \leq \M(P) + \M(Q) + \M(B_\rho).
 \]
 By integrating this inequality with respect to~$\rho\in(0, \, \rho_0]$,
 using~\eqref{restr1}, and taking the infimum with respect to all possible
 choices of~$P$ and~$Q$, we deduce that~\eqref{restr0} holds, in case~$T$
 has finite mass. If $T$ has infinite mass, we recover the same result
 using that finite-mass chains are dense in~$\F_n(\overline{U}; \, \G)$.
 
 Now, let $(S_j)_{j\in\N}$ be a sequence in~$\F_n(\overline{U}; \, \G)$
 that $\F_U$-converges to~$S$. By possibly modifying the $S_j$'s out of~$U$,
 we can assume that there exists a sequence $(R_j)_{j\in\N}$ 
 in~$\F_n(\overline{U}; \, \G)$ such that $\F(S_j - R_j)\to 0$ as $j\to+\infty$
 and $R_j - S$ is supported out of~$U$, for each~$j$. By~\eqref{restr0},
 $\F((S_j - R_j)\mres U_\rho)\to 0$ as $j\to +\infty$, for a.e.~$\rho$.
 But $R_j\mres U_\rho = S\mres U_\rho$,
 so the lemma is proved.
\end{proof}

\begin{remark}\label{remark:compact_support}
 Assume that~$H\subseteq U$ is a Borel set such that 
 $\dist(H, \, \partial U) >0$ and let $T\in\M_n(\overline{U}; \, \G)$
 be a finite-mass chain. By taking~$\rho_0 := \dist(H, \, \partial U)$,
 noting that $T\mres U_\rho = (T\mres H) + (T \mres U_\rho\setminus H)$
 for any $\rho\in (0, \, \rho_0]$, and selecting a suitable $\rho$,
 from~\eqref{restr0} we deduce that
 \[
  \F(T\mres H) \leq \left(1 + \dist^{-1}(H, \, \partial U) \right)
  \F(T) + \M(T\mres (U\setminus H)).
 \]
 By taking the infimum over all finite-mass $T$'s in a given
 equivalence class of~$\F_n(U; \, \G)$, we also deduce
 \begin{equation} \label{flat-average}
  \F(T\mres H) \leq \left(1 + \dist^{-1}(H, \, \partial U) \right)
  \F_U(T) + \M(T\mres (U\setminus H)).
 \end{equation}
 {\BBB In particular, if~$T$ is supported in~$H$ then
 \begin{equation} \label{CO2}
  \F(T) \leq \left(1 + \dist^{-1}(H, \, \partial U) \right) \F_U(T).
 \end{equation}}
\end{remark}

\begin{lemma} \label{lemma:relative_lsc_mass}
 Let $(S_j)_{j\in\N}$ be a sequence in~$\M_{n}(\overline{U}; \, \G)$, and
 let~$S\in\F_n(\overline{U}, \, \G)$ be such that $\F_U(S_j - S)\to 0$
 as~$j\to+\infty$. Then, $S\mres U$ is well-defined and 
 there holds $\M(S\mres U)\leq\liminf_{j\to+\infty}\M(S_j\mres U)$.
\end{lemma}
\begin{proof}
 Let~$U_\rho$ be as in Lemma~\ref{lemma:restriction}.
 By Lemma~\ref{lemma:restriction} and the $\F$-lower semi-continuity
 of the mass, we deduce that $\M(S\mres U_\rho)\leq
 \liminf_{j\to+\infty}\M(S_j\mres U_\rho)$ for a.e. $\rho > 0$.
 The lemma follows by letting $\rho\to 0$.
\end{proof}

Finally, we establish a compactness result with respect to the norm~$\F_U$.
We first remark that, as a consequence of our assumption~\eqref{discrete_norm},
the following property holds:
\begin{equation} \label{G_compact}
 \textrm{for any } \Lambda > 0, \textrm{ the set } 
 \left\{g\in\G\colon |g|\leq\Lambda \right\} \textrm{ is compact.} 
\end{equation}

\begin{lemma} \label{lemma:compactness}
 Assume that the coefficient group~$\G$ satisfies~\eqref{G_compact}.
 Let~$U\subseteq\R^d$ be a non-empty, bounded, open set, and let
 $(S_j)_{j\in\N}$ be a sequence in~$\M_{n}(\overline{U}; \, \G)$ such that
 \begin{equation} \label{hp-compactness}
 \sup_{j\in\N} \big(\M(S_j\mres U) + \M(\partial S_j\mres U)\big) <+\infty.
 \end{equation}
 Then, there exists a subsequence (still denoted~$S_j$) and a
 chain~$S\in\M_n(\overline{U}, \, \G)$ such that $\F_U(S_j - S)\to 0$ as~$j\to+\infty$.
\end{lemma}
\begin{proof}
 As in Lemma~\ref{lemma:restriction}, let $U_\rho := \{x\in U\colon \dist(x, \, U)> \rho\}$, for $\rho\in (0, \, \rho_0]$ and~$\rho_0>0$ fixed.
 Consider the sequence of measures in~$C_0(U)^\prime$
 defined by $A\mapsto \M(S_j\mres A)$, for $A\subseteq U$ Borel set.
 This sequence is bounded due to~\eqref{hp-compactness}, 
 and therefore it converges weakly$^\star$ (up to a subsequence)
 to a limit measure $\mu\in C_0(U)^\prime$. The boundedness of~$\mu$ implies
 that $\mu(\partial U_\rho) = 0$ for a.e.~$\rho$.
 
 Setting $B_{\rho, j} := \partial(S_j\mres U_\rho) - (\partial S_j)\mres U_\rho$,
 by \cite[Theorem 5.7]{Fleming} and Fatou lemma we deduce that
 \[
   \int_0^{\rho_0} \liminf_{j\to +\infty} \M(B_{\rho, j}) \, \d\rho \leq 
   \liminf_{j\to+\infty} \M(S_j\mres U) 
   \stackrel{\eqref{hp-compactness}}{<} + \infty.
 \]
 Therefore, for a.e.~$\rho\in(0, \, \rho_0]$ there exists a subsequence 
 (still denoted~$S_j$) such that
 \[
  \sup_{j\in\N} \left(\M(S_j\mres U_\rho) + \M(\partial(S_j\mres U_\rho))\right)
  \leq \sup_{j\in\N} \left(\M(S_j\mres U_\rho) +
     \M((\partial S_j)\mres U_\rho) + \M(B_{\rho, j}) \right) <+\infty.
 \]
 Due to the assumption~\eqref{G_compact}, and the boundedness of~$U$,
 for a.e.~$\rho\in(0, \, \rho_0]$ we can apply the compactness
 result~\cite[Corollary~7.5]{Fleming}. With the help of a diagonal argument, 
 we find a sequence $\rho_k\searrow 0$
 and a subsequence of~$j$ such that, for any~$k\in\N$,
 the following properties hold:
 \begin{gather}
  \mu(\partial U_{\rho_k}) = 0, \label{mu-bd} \\
  (S_j\mres U_{\rho_k})_{j\in\N}  \ \F\textrm{-converges to a limit, say }
     R_k\in\M_n(\overline{U}_{\rho_{k+1}}; \, \G). \label{R_k}
 \end{gather}
 The uniqueness of the limit implies that $R_k\mres U_{\rho_{h}} = R_h$, 
 for any $h < k$. The sequence $(R_k)_{k\in\N}$ is $\M$-convergent,
 because~\eqref{R_k} and the $\F$-lower semi-continuity of the mass imply
 \[
  \sum_{k\in\N} \M(R_{k + 1} - R_k) \leq \liminf_{j\to +\infty}
  \sum_{k\in\N} \M(S_j\mres (U_{\rho_{k+1}}\setminus U_{\rho_k})) 
  = \liminf_{j\to +\infty} \M(S_j\mres U) 
  \stackrel{\eqref{hp-compactness}}{<} +\infty.
 \]
 Let~$S\in\M_n(\overline{U}; \, \G)$ be the $\M$-limit of the $R_k$'s;
 it remains to check that $\F_U(S_j - S) \to 0$. For a fixed $\varepsilon > 0$,
 let $k_*\in\N$ be such that
 $\mu(U \setminus U_{\rho_{k_*}}) \leq \varepsilon/2$. Due to~\eqref{mu-bd},
 we have $\M(S_j \mres (U\setminus U_{\rho_{k_*}}))\to\mu(U\setminus U_{\rho_*})$
 and hence
 \[
  \begin{split}
  \limsup_{j\to+\infty}\F_U(S_j - S) &\leq 
  \limsup_{j\to+\infty} \left\{ \F_U((S_j - S)\mres U_{\rho_{k_*}}) +
     \M((S_j - S) \mres (U\setminus U_{\rho_{k_*}})) \right\} \\
  &\stackrel{\eqref{R_k}}{\leq} 2\mu(U\setminus U_{\rho_{k_*}}) = \varepsilon,
  \end{split}
 \]
 where we have used again the $\F$-lower semi-continuity of the mass
 and the fact that $\F_U \leq \F$. Since $\varepsilon > 0$ is arbitrary, 
 the lemma follows.
\end{proof}

\paragraph{Intersection index for flat chains.}
\label{subsect:intersection}

For~$y\in\R^d$, we denote by $\tau_y\colon x\in\R^d\mapsto x+y$ the translation map 
associated with~$y$. Given chains $S\in\F_n(\R^d; \, \G)$ and~$R\in\N_{m}(\R^d; \, \Z)$, 
with~$n+m\geq d$, for a.e.~$y\in\R^d$ we would like to define the 
intersection~$S\cap \tau_{y,*}R$ as an element of~$\F_{n+m-d}(\R^d;\, \G)$.
This construction has been described in \cite[Section~5]{White-Notes}
but, for the convenience of the reader, we briefly recall it here. 

Suppose first that~$S$, $R$ are single polyhedra. By Thom transversality theorem, 
for a.e.~$y$ the polyhedra $S$ and~$\tau_{y,*}R$ intersect transversely, so
the set~$\sigma:= \llbracket S\rrbracket\cap\llbracket\tau_{y,*}R\rrbracket$
is a finite union of polyhedra of dimension~$n+m-d$.
We orient~$\llbracket S\rrbracket\cap\llbracket\tau_{y,*}R\rrbracket$
according to the convention of~\cite[Section~3]{White-Rectifiability},
i.e., the orientation is chosen in such a way that the following holds:
if $(u_1, \, \ldots,\, u_{n+m-d})$ is an oriented basis for the $(n+m-d)$-plane spanned
by~$\llbracket S\rrbracket\cap\llbracket\tau_{y,*}R\rrbracket$,
$(u_1, \, \ldots,\, u_{n+m-d}, \, v_1, \, \ldots, \, v_{d-m})$ is an 
oriented basis for the $n$-plane spanned by~$\llbracket S\rrbracket$, and
$(u_1, \, \ldots,\, u_{n+m-d}, \, w_1, \, \ldots, \, w_{d-m})$ is an 
oriented basis for the $m$-plane spanned by~$\llbracket \tau_{y,*}R\rrbracket$, then
$(u_1, \, \ldots,\, u_{n+m-d}, \, v_1, \, \ldots, \, v_{d-m}, \,  w_1, \, \ldots, \, w_{d-n})$
is a positively oriented basis for~$\R^d$. Having chosen the orientation, we can
regard the intersection $S\cap\tau_{y,*}R$ as a polyhedral $(n+m-d)$ chain,
in the obvious way. This definition now extend by linearity to the case~$S$,
$R$ are polyhedral. Now, it can be shown
\cite[Theorem~5.3]{White-Notes} that
\begin{equation} \label{cap_chain}
 \int_{\R^d} \F(S\cap \tau_{y,*}R) \,\d y \leq \F(S)(\M(R)+\M(\partial R)).
\end{equation}
As a consequence, we can extend~$\cap$ by continuity so that, for
any~$S\in\F_n(\R^d; \, \G)$ and~$R\in\N_{m}(\R^d; \, \Z)$, $S\cap\tau_{y,*}R$
{\BBB is a well-defined element of~$\F_{n+m-d}(\R^d; \, \G)$ for a.e.~$y\in\R^m$.}
Moreover, for a sequence~$(S_j)_{j\in\N}$ that converges to~$S$ 
in the flat norm and a.e.~$y\in\R^d$,
the chain $S_j\cap\tau_{y,*}R$ flat-converges to~$S\cap\tau_{y,*}R$.

For the convenience of the reader, we sketch the proof of~\eqref{cap_chain}.

\begin{proof}[Proof of~\eqref{cap_chain}]
 Suppose that~$S$, $R$ are single polyhedra. Then, by applying the coarea formula,
 we deduce that
 \begin{equation*}
  \int_{\R^d} \M(S\cap \tau_{y,*}R) \,\d y \leq \M(S)\M(R).
 \end{equation*}
 This inequality can be extended by linearity to the case~$S$, $R$ are polyhedral chains.
 Now, it can be checked that, when $A$, $B$ are polyhedral chains that intersect
 transversely, and with the orientation convention described above, there holds
 \begin{equation} \label{cap_boundary}
  \partial(A\cap B) = (-1)^{d-m}\partial A\cap B + A\cap\partial B. 
 \end{equation}
 Therefore, writing $S = P + \partial Q$, we have
 \[
  S\cap\tau_{y,*}R = P\cap\tau_{y,*}R + (-1)^{d-m} Q\cap\tau_{y,*}(\partial R)
  + (-1)^{d-m+1}\partial (Q\cap\tau_{y,*}R).
 \]
 By taking the flat norm and integrating with respect to~$y\in\R^d$, we see that
 the left-hand side of~\eqref{cap_chain} is bounded by 
 $\M(P)\M(R)+ \M(Q)\M(\partial R) + \M(Q)\M(R)$, and hence \eqref{cap_chain} follows.
\end{proof}

In the rest of the paper, we will be interested in the case~$S$, $R$ are of 
complementary dimensions, that is, $\dim(S) + \dim(R) = d$. In this case, $S\cap\tau_{y,*}R$
is a $0$-chain, and we can consider the quantity~$\chi(S\cap\tau_{y,*}R)\in\G$,
where~$\chi$ is the augmentation homomorphism given by Lemma~\ref{lemma:augmentation}.

\begin{lemma} \label{lemma:stability_chi}
 Suppose that~\eqref{discrete_norm} is satisfied. Let $S\in\F_n(\R^d; \, \G)$,
 $R\in\N_{d-n}(\R^d; \, \Z)$ be chains such that
 \begin{equation} \label{hp_intersection}
  \spt(\partial S)\cap\spt(R) = \spt(S)\cap\spt(\partial R)  = \emptyset.
 \end{equation}
 Then, there exists~$\delta = \delta(S, \, R) > 0$ such that,
 for a.e.~$y_1$, $y_2\in B^d_{\delta}$, there holds
 \[
  \chi(S\cap \tau_{y_1,*}R) = \chi(S\cap \tau_{y_2,*}R).
 \]
\end{lemma}
\begin{proof}
 Suppose first that~$S$, $R$ are polyhedra of complementary
 dimensions that satisfy~\eqref{hp_intersection}. Take~$y_1$,
 $y_2$ such that $S$ intersects transversely~$\tau_{y_1,*}R$
 and~$\tau_{y_2,*}R$ and $y_2 - y_1$ does not belong to the linear
 subspace spanned by~$R$. If $|y_1|$, $|y_2|$ are small enough, then the
 $0$-chain $S\cap(\tau_{y_2,*}R - \tau_{y_1,*}R)$ is (either~$0$ or)
 the boundary of a segment whose length tends to zero as~$|y_2 - y_1|\to 0$,
 for fixed~$S$, $R$. (The assumption~\eqref{hp_intersection} is essential here.) 
 Thus, $\F(S\cap\tau_{y_2,*}R - S\cap\tau_{y_1,*}R)\to 0$ as
 $|y_1|$ and~$|y_2|$ simoultaneously tend to~$0$, and hence
 \eqref{discrete_norm}, together with Lemma~\ref{lemma:augmentation}, implies that
 $\chi(S\cap\tau_{y_1,*}R) = \chi(S\cap\tau_{y_2, *}R)$ for~$|y_1|$, $|y_2|$
 small enough. By linearity and a density argument, 
 using the stability of~$\cap$ and~$\chi$ with respect to the flat convergence 
 (Equation~\eqref{cap_chain} and Remark~\ref{remark:augmentation_stability} respectively),
 the lemma follows.
\end{proof}

By Lemma~\ref{lemma:stability_chi}, the function~$y\in\R^d\mapsto 
\chi(S\cap\tau_{y,*}R)\in \G$ is equal a.e. to a constant, in a neighbourhood of~$0$.
We call such constant \emph{the intersection product} of~$S$ and~$R$,
and we denote it by~$\I(S, \, R)$. Note that~$\I(S, \, R)$ is not well-defined
if the condition~\eqref{hp_intersection} does not hold.

Let~$U\subseteq\R^d$ is a non-empty, open set. Let $S$, 
$R$ satisfy~\eqref{hp_intersection} with~$\spt R\subseteq U$, and
let~$S^\prime\in\F_{n}(\R^d; \, \G)$ be such that
$\spt(S - S^\prime)\subseteq\R^d\setminus U$. By approximating $S$, $S^\prime$, $R$
with polyhedral chains $S_j$, $S^\prime_j$, $R_j$ such that 
$\spt R_j\csubset U\setminus\spt(S_j - S^\prime_j)$, it can be checked that
$\I(S, \, R) = \I(S^\prime, \, R)$. Therefore, the intersection index 
$\I(S, \, R)$ is well-defined when~$S\in\F_n(U; \, \G)$, provided that~$R$ satisfies 
$\spt(R)\subseteq U$ in addition to~\eqref{hp_intersection}.

\begin{lemma} \label{lemma:intersection_product}
 The intersection product satisfies the following properties.
 \begin{enumerate}[label=(\roman*)]
  \item \label{intersection:support} $\I(S, \, R) = 0$ if~$\spt(S)\cap\spt(R) = \emptyset$.
  
  \item \label{intersection:bilinear} $\I$ is bilinear: $\I(S_1 + S_2, \, \, R) = \I(S_1, \, R) + \I(S_2, \, R)$ 
  and $\I(S, \, \, R_1 + R_2) = \I(S, \, R_1) + \I(S, \, R_2)$, as soon as all the
  terms are well-defined.
  
  \item \label{intersection:stability} $\I$ is stable with respect to $\F_U$-convergence: 
  if~$U\subseteq\R^d$ is a non-empty open set, 
  $(S_j)_{j\in\N}$ is a sequence in~$\F_n(U; \, \G)$ 
  that $\F_U$-converges to~$S$, and if~$R\in\N_{d-n}(\R^d; \, \Z)$ satisfies
  \[
   \spt(R)\subseteq U, \qquad \spt(\partial S_j)\cap\spt(R) 
   = \spt(S_j)\cap\spt(\partial R) = \emptyset \quad \textrm{for any } j\in\N,
  \]
  then~$\I(S, \, R) = \I(S_j, \, R)$ for any~$j$ large enough.
  
  \item \label{intersection:homology} $\I$ is stable with respect to homology:
  for any~$S\in\F_n(\R^d; \, \G)$ and~$R\in\N_{d-n+1}(\R^d; \, \Z)$ such that
  $\spt(\partial S)\cap\spt(\partial R) = \emptyset$, there holds
  $\I(S, \, \partial R) = (-1)^n\I(\partial S, \, R)$. In particular, if
  $\spt(\partial S)\cap\spt(R)= \emptyset$ then~$\I(S, \, \partial R) =0$.
 \end{enumerate}
\end{lemma}
\begin{proof}
 Properties (i), (ii) and~(iii) follow in a straighforward way from~\eqref{cap_chain} and
 Lemma~\ref{lemma:augmentation}, \ref{lemma:stability_chi}.
 For~\ref{intersection:homology} we remark that, due to~\eqref{cap_boundary}, there holds
 \[
  (-1)^{n-1}\partial S\cap\tau_{y,*}R + S\cap\tau_{y,*}(\partial R) =
  \partial(S\cap\tau_{y,*} R).
 \]
 By taking~$\chi$ on both sides, and applying Lemma~\ref{lemma:augmentation}.(ii) 
 and~\ref{intersection:bilinear}, we obtain~\ref{intersection:homology}.
\end{proof}

\subsection{The group~$\pi_{k-1}(\NN)$}
\label{subsect:group_norm}

If the condition~\eqref{H} is satisfied, in particular if~$\pi_j(\NN) = 0$ for any
integer~$0 \leq j\leq k-2$ (and~$\pi_1(\NN)$ is abelian in case~$k=2$), then the action 
of~$\pi_1(\NN)$ over~$\pi_{k-1}(\NN)$ is trivial. 
Therefore, we can and we shall identify the \emph{free} homotopy classes of
continuous maps~$\SS^{k -1}\to\NN$ 
with the elements of the homotopy group~$\pi_{k-1}(\NN)$. Moreover, we have an isomorphism
\begin{equation*} 
 \pi_{k-1}(\NN) \simeq H_{k-1}(\NN),
\end{equation*}
due to Hurewicz theorem (see e.g.~\cite[Theorem~4.37 p.~371]{Hatcher}).
The group~$H_{k-1}(\NN)$ is finitely generated because~$\NN$ can be given the structure of a
\emph{finite} CW complex, hence $\pi_{k-1}(\NN)$ is finitely generated. The choice of a finite
generating set~$\{\gamma_j\}_{j=1}^q$ for~$\pi_{k-1}(\NN)$
induces a group norm on~$\pi_{k-1}(\NN)$, in the following way:
for~$a\in\pi_{k-1}(\NN)$, $|a|$ is the smallest length
of a sum of~$\gamma_j$'s representing~$a$, that is
\begin{equation} \label{group_norm}
 |a| :=\inf \left\{\sum_{j = 1}^q |d_j|\colon (d_j)_{j=1}^q\in\Z^q, \ 
 a = \sum_{j=1}^q {d_j}\gamma_j \right\}\! .
\end{equation}
It can be easily checked that the right-hand side does define a group norm.
This norm is integer-valued, so~$\pi_{k-1}(\NN)$ is a discrete topological space,
and the condition~\eqref{group_norm} is satisfied. Throughout the paper, 
we will consider flat chains with multiplicities in~$\G=\pi_{k-1}(\NN)$.

\subsection{Smooth complexes and the retraction over~$\NN$}
\label{subsect:RR}

A compact set~$\X\subseteq\R^m$ will be called a $n$-dimensional smooth (resp., Lipschitz) complex 
if and only if there exists a diffeomorphism (resp., a bilipschitz map),
defined on a neighbourhood of~$\X$, that takes~$\X$ onto a \emph{finite} $n$-dimensional 
simplicial complex~$\widehat{\X}\subseteq\R^{m}$.
For~$j\in\N$ with~$j\leq n$, we define the~$j$-skeleton~$\X^j$ of~$\X$ 
as the union of all the cells of dimension~$\leq j$.

We recall an important topological fact, upon which our construction is based.

\begin{lemma}[{\cite[Lemma~6.1]{HardtLin-Minimizing}}] \label{lemma:X}
 Suppose that~\eqref{H} is satisfied. Then, there
 exist a compact~$(m-k)$-dimensional smooth complex~$\X\subseteq\R^m$ and a locally smooth
 retraction $\RR\colon\R^m\setminus\X\to\NN$ such that
 \[
  |\nabla\RR(y)| \leq \frac{C}{\dist(y, \, \X)}
 \]
 for any~$y\in\R^m\setminus\X$ and some constant~$C=C(\NN, \, m, \, \X)>0$,
 and~$\nabla\RR$ has full rank on a neighbourhood of~$\NN$.
\end{lemma}
\begin{proof}
 This is exactly the statement of~\cite[Lemma~6.1]{HardtLin-Minimizing}, except that in~\cite{HardtLin-Minimizing},
 the set~$\X$ is required to be a Lipschitz complex and~$\RR$ is required to be a Lipschitz map.
 However, the same argument can be used to produce a smooth pair~$(\X, \, \RR)$ with the same properties 
 (one starts with a smooth triangulation of~$\SS^m = \R^m\cup\{\infty\}$, in place of a Lipschitz triangulation;
 the smoothness of~$\RR$ can be achieved by a standard regularisation argument).
 Another, self-contained, proof of this fact is given in \cite[Proposition~2.1]{BousquetPonceVanSchaftingen}.
\end{proof}

Notice that $\RR_{y}\colon z\in\NN\mapsto\RR\circ(y-z)$, for $|y|$ small enough,
defines a smooth family of maps $\NN\to\NN$ such that~$\RR_0 =\Id_{\NN}$. Therefore,
the implicit function theorem implies that $\RR_y$ has a smooth
inverse~$\RR_y^{-1}\colon\NN\to\NN$ for $|y|$ sufficiently small. 

\subsection{Manifold-valued Sobolev maps}
\label{subsect:manifold-Sobolev}

Given a bounded, smooth open set~$U\subseteq\R^d$ and a number~$1 \leq p < +\infty$,
we let $H^{1,p}(U, \, \NN)$ denote the \emph{strong}~$W^{1,p}$-closure
of~$C^\infty(\overline{U}, \, \NN)$. We denote by $H^{1,p}_{\mathrm{loc}}(U, \, \NN)$ 
the set of maps~$u\in W^{1, p}(U, \, \NN)$ such that, for any point~$x\in U$, 
there exists a ball~$B_r(x)\csubset U$ such that $u_{|B_r(x)}\in H^{1,p}(B_r(x), \, \NN)$.
Clearly, we have the chain of inclusions
\[
 H^{1,p}(U, \, \NN) \subseteq 
 H^{1,p}_{\mathrm{loc}}(U, \, \NN) \subseteq W^{1,p}(U, \, \NN)
\]
and a well-known result by Bethuel~\cite[Theorem~1]{Bethuel-Density} implies that the equality 
$H^{1,p}_{\mathrm{loc}} = W^{1,p}$ holds
if and only if~$p\geq d$ or~$\pi_{\lfloor p\rfloor}(\NN) = 0$. The equality~$H^{1,p} = W^{1,p}$
has been characterised by Hang and Lin~\cite[Theorem~1.3]{HangLin}
in terms of topological properties of~$U$ and~$\NN$. In particular, we have

\begin{lemma} \label{lemma:H1,p_loc}
 Suppose that~$U\subseteq\R^d$ is a smooth, bounded domain that has the same homotopy type
 of a smooth $(k-1)$-complex, and let~$p \geq k - 1$.
 Then, $H^{1,p}(U, \, \NN) = H^{1, p}_{\mathrm{loc}}(U, \, \NN)$.
\end{lemma}
\begin{proof}
 If~$p\geq d$ then, arguing as in \cite[Proposition~p.~267]{SchoenUhlenbeck2}, 
 one sees that $H^{1,p}(U, \, \NN) = H^{1, p}_{\mathrm{loc}}(U, \, \NN) = W^{1,p}(U, \, \NN)$,
 so we can assume that~$k - 1\leq p < d$. Let~$u\in H^{1,p}_{\mathrm{loc}}(U, \, \NN)$ 
 and~$\varepsilon>0$ be given.
 By reflection across~$\partial\Omega$, we can extend~$u$ 
 to a map in~$W^{1,p}(U^\prime, \, \NN)$, where~$U^\prime\supset\!\supset U$ 
 is a slightly larger domain that retracts onto~$\overline{U}$ 
 (see e.g. \cite[Lemma~8.1 and~Remark 8.2]{ABO2});
 thus, we can apply the methods of~\cite{HangLin}, even though~$U$ has a boundary.
 
 Thanks to~\cite[Theorem~6.1]{HangLin} (or \cite[Theorem~2]{Bethuel-Density}),
 we find a smooth cell decomposition~$M$ of~$U^\prime$, a dual 
 $d-\lfloor p \rfloor - 1$-skeleton~$L^{d-\lfloor p \rfloor - 1}$
 and a map~$\tilde{u}\in W^{1,p}(U^\prime, \, \NN)$ that is continuous 
 on~$U^\prime\setminus L^{d-\lfloor p \rfloor - 1}$ and satisfies
 $\|u - \tilde{u}\|_{W^{1,p}}\leq\varepsilon$. It suffices to show
 that~$\tilde{u}_{|M^{\lfloor p\rfloor}}$ can be extended to a continuous map~$U^\prime\to\NN$,
 as~\cite[Theorem~6.2]{HangLin} yields then~$\tilde{u}\in H^{1,p}(U, \, \NN)$ and,
 by arbitarity of~$\varepsilon$, $u\in H^{1,p}(U, \, \NN)$.
 
 For each cell~$Q\in M^{\lfloor p \rfloor+1}$, denote 
 by~$Q^\prime\in L^{d-\lfloor p \rfloor - 1}$ the dual cell and let~$x$
 such that~$Q\cap Q^\prime = \{x\}$. Since~$u\in H^{1,p}_{\mathrm{loc}}(U, \, \NN)$,
 there exist~$0 < \rho < \dist(x, \, \partial Q)$ a sequence of smooth maps
 $u_j\colon B_\rho(x)\to\NN$ that converges to~$u$ in~$W^{1, p}$.
 In case~$p\notin\mathbb{Z}$, Fubini theorem and Sobolev embeddings imply
 that, for a.e.~$r\in(0, \, \rho)$, $u_j\to u$ uniformly on~$Q\cap\partial B_r(x)$,
 therefore the homotopy class of~$u_{|Q\cap\partial B_r(x)}$ is trivial.
 In case~$p\in\mathbb{Z}$, one can approximate~$u_j$ with the continuous functions
 \[
  u_j^{\delta}(y) := \fint_{Q \cap B_r(x)\cap B_\delta(y)} u_j(z)\,\d z
  \qquad \textrm{for } y\in Q\cap B_r(x).
 \]
 By the Poincar\'e inequality, we deduce that $\sup_{j, y}\dist(u^\delta_j(y), \, \NN)\to 0$ 
 as~$\delta\to 0$ and~$u^\delta_j \to u^\delta$ uniformly on~$Q\cap B_r(x)$
 as~$j\to+\infty$, so the same conclusion follows. (Details of the argument can be found 
 in~\cite{BN1} and~\cite[Lemma~4.4]{HangLin}.) 
 By similar arguments we also obtain that, if~$\varepsilon$ is small enough,
 then the homotopy class of~$\tilde{u}_{|Q\cap\partial B_r(x)}$ is trivial,
 hence the homotopy class of~$\tilde{u}_{|\partial Q}$ is trivial 
 and~$\tilde{u}_{M^{\lfloor p \rfloor}}$ has a continuous extension
 $M^{\lfloor p \rfloor + 1}\to\NN$. Finally, by applying~\cite[Lemma~2.2]{HangLin}
 and reminding that~$U$ is homotopy equivalent to a $(k-1)$-complex and that~$p\geq k-1$,
 we conclude that~$\tilde{u}_{|M^{\lfloor p \rfloor}}$ has a continuous extension~$U\to\NN$.
\end{proof}

\paragraph*{Push-forward of a chain by a Sobolev map and homology classes.}

Let~$S\in\M_{k-1}(U; \, \Z)$ be an integral chain
with~$\partial S = 0$, and let~$u\in H^{1,k-1}(U, \, \NN)$.
We aim at defining the homology class of the push-forward chain~$u_{*}(S)$.
To this end, we pick a sequence~$(u_n)_{n\in\N}$ in~$(C^\infty\cap W^{1,k-1})(U, \, \NN)$ 
that converges to~$u$ in~$W^{1,k-1}$,
and a sequence~$(S_j)_{j\in\N}$ of polyhedral chains supported in an open set~$U^\prime\csubset U$,
with~$\partial S_j = 0$ for any~$j\in\N$, that converges to~$S$ in the flat-norm.
(Such a sequence~$S_j$ exists as a consequence of the deformation theorem; see e.g.~\cite[Theorem~5.6 and remark at p.~175]{Fleming}.)
We claim that, for any~$n$,~$m$,~$i$,~$j$ large enough, 
\begin{equation} \label{same_class}
 [u_{n,*}(S_i)]_{H_{k-1}(\NN)} = [u_{m,*}(S_j)]_{H_{k-1}(\NN)}.
\end{equation}
This homology class does not depend on the choice of the sequences~$(u_n)$
and~$(S_j)$, for any two such pairs of
sequences~$(u_n, \, S_j)$ and~$(u^\prime_n, \, S^\prime_j)$
can be restructured into a single converging one.
We denote this homology class by~$[u_{*}(S)]$. By the Hurewicz
isomorphism~\cite[Theorem~4.37 p.~371]{Hatcher},
$[u_{*}(S)]$ defines a unique homotopy class in~$\pi_{k-1}(\NN)$,
which we denote by the same symbol. By an approximation argument, 
once Claim~\eqref{same_class} is proved we can deduce

\begin{lemma} \label{lemma:homotopy}
 If~$(u_j)_{j\in\N}$ is a sequence in~$H^{1,k-1}(U, \, N)$
 that converges $W^{1,k-1}$-strongly to~$u$, and~$(S_j)_{j\in\N}$ is a sequence of cycles
 in~$\N_{k-1}(U; \, \Z)$ that converges to~$S$ in the flat-norm, then
 \[
  [u_{*}(S)] = [u_{j, *}(S_j)]
 \]
 for any~$j$ large enough.
\end{lemma}

\begin{proof}[Proof of Claim~\eqref{same_class}]
 By applying~\cite[Lemma~7.7]{Fleming}, for~$i$ and~$j$ large enough
 we find a polyhedral $k$-chain~$R_{ij}$, supported in~$U$,
 such that $S_i - S_j = \partial R_{ij}$. Then, for any~$n$ we have
 \begin{equation*} 
  u_{n,*}(S_i) - u_{n,*}(S_j) = \partial \left(u_{n,*}(R_{ij})\right),
 \end{equation*}
 so~$u_{n,*}(S_i)$ and~$u_{n,*}(S_j)$ belong to the same (smooth) homology class.
 Now, it follows from~\cite[Lemma~4.5]{HangLin} and the fact that~$(u_n)_{n\in\N}$ 
 is a Cauchy sequence in~$W^{1, k-1}(U)$ that, for any~$n$, $m$ large enough
 and any $(k-1)$-polyhedral complex~$K\subseteq U$, ${u_n}_{|K}$ and~${u_m}_{|K}$ 
 belong to the same homotopy class of continuous maps~$K\to\NN$. Since homotopic
 maps induce the same push-forward in homology, it follows
 that~$[u_{n,*}(S_i)] = [u_{m,*}(S_i)]$ for any~$n$, $m$ large enough and any~$i$
 and, hence, Claim~\eqref{same_class} is proved.
\end{proof}

\section{The construction of the sets of topological singularities}
\label{sect:Stop}

\subsection{Statement of the main results}
\label{subsect:theorem}

Let~$\Omega\subseteq\R^d$ be a {\BBB smooth, bounded, connected open set}, $d\geq k$.
We consider the set $X(\Omega) := (L^\infty\cap W^{1, k-1})(\Omega, \, \R^m)$
with the direct limit topology induced by the inceasing family of subspaces
\[
 X_\Lambda(\Omega) := \left\{ u\in W^{1,k-1}(\Omega, \, \R^m)\colon 
 \norm{u}_{L^\infty(\Omega)} \leq \Lambda \right\} \qquad \textrm{for } \Lambda>0
\]
(each~$X_\Lambda(\Omega)$ is given the strong~$W^{1,k-1}$-topology).
This defines a metrisable topology on~$X(\Omega)$, and 
a sequence~$(u_j)_{j\in\N}$ converges to~$u$ in~$X(\Omega)$ 
if and only if $u_j \to u$
strongly in $W^{1, k-1}$ and $\sup_{j\in\N} \|u_j\|_{L^\infty} < +\infty$.
We also consider the set $Y(\Omega) := L^1(\R^m, \, \F_{d-k}(\Omega; \, \pi_{k-1}(\NN)))$,
whose elements are Lebesgue-measurable maps~$S\colon\R^m\to
\F_{d-k}(\Omega; \, \pi_{k-1}(\NN))$ such that
\[
 \norm{S}_Y := \int_{\R^m} \F_\Omega(S(y)) \, \d y < +\infty.
\]
The set~$Y(\Omega)$ is a complete normed $\pi_{k-1}(\NN)$-modulus,
with respect to the norm~$\|\cdot\|_Y$. When no ambiguity arises, 
we will write $X$, $Y$ instead of $X(\Omega)$, $Y(\Omega)$.
{\BBB Given an operator~$S\colon X\to Y$, throughout the paper
we will write $S_y(u) := (S(u))(y)$ for~$y\in\R^m$.}
Recall that, {\BBB even when not explicitely stated,} 
the assumption~\eqref{H} is in force, see Section~\ref{subsect:intro-construction}.

\begin{theorem} \label{th:Stop}
 Suppose that~\eqref{H} is satisfied. Then, there exists a unique continuous map
 $\S\colon X\to Y$ that satifies the following property:
 \begin{enumerate}[label=\emph{(P\textsubscript{\arabic*})}, ref=P\textsubscript{\arabic*}]
  \item \label{S_intersection} For any~$u\in X$, 
  {\BBB any $R\in\N_k(\overline{\Omega}; \, \Z)$ such that $\spt R\subseteq\Omega$,
  and a.e.~$y\in\R^m$ such that
  $\spt(\partial R)\cap\spt(\S_y(u))=\emptyset$,} there holds
  \[
   \I(\S_y(u), \, R) = \left[\RR\circ(u - y)_{*}(\partial R)\right] \! .
  \]
  {\BBB Here~$\RR\colon\R^m\setminus\X\to\NN$
 denotes the retraction given by Lemma~\ref{lemma:X}.}
 \end{enumerate}
 Moreover, for any~$\Lambda > 0$ there exists $C_{\Lambda}>0$
 such that, for any~$u\in X$ with~$\|u\|_{L^\infty(\Omega)}\leq\Lambda$ 
 and a.e.~$y\in\R^m$, the following properties are satisfied.
 \begin{enumerate}[label=\emph{(P\textsubscript{\arabic*})}, ref=P\textsubscript{\arabic*}, resume]
  \item \label{S_spt} $\RR\circ(u - y)\in W^{1, k-1}(\Omega, \, \NN)\cap 
  H^{1, k-1}_{\mathrm{loc}}(\Omega\setminus\spt\S_y(u), \, \NN)$.
  
  \item \label{S_boundary} $\S_y(u)$ is a relative boundary in~$\Omega$: 
  there exists~$R_y\in\M_{d-k+1}(\overline{\Omega}; \, \pi_{k-1}(\NN))$ such that
  $\spt(\S_y(u) - \partial R_y)\subseteq\R^d\setminus\Omega$ and 
  \[
   \int_{\R^m} \M(R_y) \, \d y \leq C_{\Lambda}
   \norm{\nabla u}_{L^{k-1}(\Omega)}^{k-1} \! .
  \]
  
  \item \label{S_mass} If, in addition,~$u\in W^{1, k}(\Omega, \, \R^m)$ then, 
  for a.e.~$y\in\R^m$, the chain $\S_y(u)$ has finite mass and there holds
  \[
   \int_{\R^m} \M(\S_y(u)) \,\d y \leq C_{\Lambda}
   \norm{\nabla u}^k_{L^k(\Omega)} \! .
  \]
   
  \item \label{S_flat} If~$u_0$, $u_1\in X$ satisfy
  $\|u_0\|_{L^\infty(\Omega)}\leq\Lambda$,
  $\|u_1\|_{L^\infty(\Omega)}\leq\Lambda$, then
  \[
   \int_{\R^m} \F_\Omega(\S_y(u_1) - \S_y(u_0)) \,\d y \leq 
   C_{\Lambda}\int_\Omega \abs{u_1 - u_0} \left(\abs{\nabla u_1}^{k-1} + \abs{\nabla u_0}^{k-1}\right) \! .
  \]
 \end{enumerate}
\end{theorem}

Property~\eqref{S_intersection} 
implies that $\S_y(u)$ does capture topological information on~$u$,
and motivates the name ``set of topological singularities''. 
Notice that both sides of \eqref{S_intersection} are well-defined, thanks to 
\eqref{S_spt} and Lemmas~\ref{lemma:H1,p_loc}, \ref{lemma:homotopy}.
\eqref{S_boundary} and~\eqref{S_mass} provide an integral control on the
$\F_\Omega$-norm and the mass norm of~$\S_y(u)$, respectively.
Property~\eqref{S_boundary} will be crucially exploited in the
applications we present in Section~\ref{sect:BV}, 
while \eqref{S_mass} is important in applications to variational problems, 
along the lines of~\cite{ABO2}.
{\BBB From~\eqref{S_boundary} and White's rectifiability
theorem for flat chains~\cite{White-Rectifiability}, one can deduce that~$\S_y(u)$
is the boundary of a rectifiable chain for a.e.~$y\in\R^m$ 
(compare with~\cite[Theorem~3.8]{ABO1}).
Finally,~\eqref{S_flat} is a continuity estimate in the spirit of
\cite[Theorem~1]{BrezisNguyen}, Statement~(i). 
By applying the H\"older inequality to~\eqref{S_flat},
for any~$\Lambda>0$ and any vector-valued maps $u_0$,
$u_1\in (L^\infty\cap W^{1,k})(\Omega, \, \R^m)$ 
satisfying~$\|u_0\|_{L^\infty(\Omega)}\leq\Lambda$ and
$\|u_1\|_{L^\infty(\Omega)}\leq\Lambda$ we obtain
\begin{equation} \label{Stop_weak_cont}
  \int_{\R^k} \F_\Omega\left(\S_y(u_1) - \S_y(u_0)\right)\d y 
  \leq C_\Lambda \norm{u_0 - u_1}_{L^k(\Omega)}
  \left( \norm{\nabla u_0}^{k-1}_{L^k(\Omega)} 
  + \norm{\nabla u_1}^{k-1}_{L^k(\Omega)} \right) \! .
\end{equation}
A natural question is whether the distance~$\norm{\S(u_1) - \S(u_0)}_Y$
can be bounded in terms of the Sobolev norms of~$\nabla u_1 - \nabla u_0$.
Notice that such an estimate is available for the Jacobian 
(see \cite[Theorem~1]{BrezisNguyen}, Statement~(ii)).
However, the answer is not known in general.}

%
%

The uniqueness of the operator~$\S$, together with
Lemma~\ref{lemma:flat_restriction}, implies the following property.

\begin{corollary} \label{cor:locality}
 Let $\Omega_1$, $\Omega_2$ be bounded, smooth domains in~$\R^d$
 with $\Omega_1\csubset\Omega_2$, and let $\S^{\Omega_1}$, $\S^{\Omega_2}$
 be the corresponding operators, given by Theorem~\ref{th:Stop}.
 Let~$\Psi\colon \F_{d-k}(\Omega_2; \, \pi_{k-1}(\NN))
 \to\F_{d-k}(\Omega_1; \, \pi_{k-1}(\NN))$ be the restriction map, given by
 Lemma~\ref{lemma:flat_restriction}. Then, for any~$u\in X(\Omega_2)$ 
 and a.e.~$y\in\R^d$, there holds
 \[
  \S^{\Omega_1}_y(u_{|\Omega_1}) = \Psi(\S^{\Omega_2}_y(u)).
 \]
\end{corollary}

Corollary~\ref{cor:locality} implies that the operator $\S$ is local:
if two maps $u_1$, $u_2\in X(\Omega)$ coincide a.e. on a 
(not necessarily smooth) open subset $\omega\subseteq\Omega$, then
$\spt(\S_y(u_2) - \S_y(u_1))\subseteq\overline{\Omega}\setminus\omega$ 
for a.e.~$y\in\R^m$. If we had constructed $\S$ as an operator
with values in {\BBB $L^1(\R^m, \, \F_n(\overline{\Omega}; \, \pi_{k-1}(\NN)))$}, then
Corollary~\ref{cor:locality} would not hold, because the restriction
{\BBB $\S^{\Omega_2}_y(u)\mres\overline{\Omega}$} need not be well-defined
(see the discussion in Section~\ref{subsect:relative_chains}).

We can, in the suitable sense, define ``the trace of $\S$'' on the boundary of~$\Omega$.
More precisely, suppose that $d\geq k+1$, consider the space
$X^{\mathrm{bd}} := (L^\infty\cap W^{1-1/k, k})(\partial\Omega, \, \R^m)$ 
and define a direct limit topology on it, in such a way that a sequence
$(g_j)_{j\in\N}$ converges to~$g$ in~$X^{\mathrm{bd}}$
if and only if $g_j\rightharpoonup g$ \emph{weakly} in $W^{1-1/k, k}$ and 
$\sup_{j\in\N}\|g_j\|_{L^\infty(\partial\Omega)}<+\infty$. Let 
$Y^{\mathrm{bd}} := L^1(\R^m, \, \F_{d-k-1}(\partial\Omega; \,\, \pi_{k-1}(\NN)))$
be endowed with the norm
\[
 \|S\|_{Y^{\mathrm{bd}}} := \int_{\R^m} \F(S_y) \, \d y.
\]

\begin{prop} \label{prop:boundary}
 There exists a
 sequentially continuous operator $\S^{\mathrm{bd}}\colon X^{\mathrm{bd}} \to Y^{\mathrm{bd}}$
 with the following property:
 for any~$g\in X^{\mathrm{bd}}$, any open set~$\Omega^\prime\supset\!\supset\Omega$,
 any~$u\in(L^\infty\cap W^{1, k})(\Omega^\prime, \, \R^m)$
 such that~$u = g$ on~$\partial\Omega$ (in the sense of traces), and a.e.~$y\in\R^m$,
 there holds $\S^{\mathrm{bd}}_y(g) = \partial(\S_y(u)\mres\Omega) 
 = \partial(\S_y(u)\mres\overline{\Omega})$.
\end{prop}

Note that, for a.e.~$y$, the restrictions~$\S_y(u)\mres\Omega$,
$\S_y(u)\mres\overline{\Omega}$ are well-defined
because $\S_y(u)$ has finite mass, due to 
\eqref{S_mass}. The space $X^{\mathrm{bd}}$ does \emph{not} coincide with the image of~$X$
under the trace operator, that is
$\mathrm{tr} \, (X) = (L^\infty\cap W^{1 - 1/(k-1), k-1})
(\partial\Omega, \, \R^m)\supseteq X^{\mathrm{bd}}$. In general,
it is not possible to extend~$\S^{\mathrm{bd}}$ to an operator
$\mathrm{tr} \, (X)\to Y^{\mathrm{bd}}$
that is continuous with respect to the strong topology on~$\mathrm{tr} \, (X)$.
In case~$\NN=\SS^1$, $k=m=2$, $\Omega$ 
is the unit ball in~$\R^3$, if such an extension existed then 
$\S^{\mathrm{bd}}_y(g)$ would be defined for merely measurable maps
$g\colon\SS^2\to\SS^1$, and continuous with respect to strong $L^1$-convergence.
But $C^\infty(\SS^2, \, \SS^1)$ is dense in $L^1(\SS^2, \, \SS^1)$
and $\S^{\mathrm{bd}}_y(g) = 0$ for any $g\in C^\infty(\SS^2, \, \SS^1)$ and 
a.e. $y\in\R^m$, so $\S^{\mathrm{bd}} = 0$. This is a contradiction,
in view of~\eqref{Jacobian_integral}, 
as there are maps in~$W^{1,1}(\SS^2, \,\SS^1)\subseteq W^{1/2,2}(\SS^2, \,\SS^1)$
whose distributional Jacobian is non zero.

Recall that two chains are said to be homologous (or cobordant) if they differ by a boundary.
In case~$u\in(L^\infty\cap W^{1,k})(\Omega, \, \R^m)$, the homology class of~$\S_y(u)$
is determined by the boundary conditions only. More precisely, we have the following

\begin{prop} \label{prop:cobordism}
 For any $g\in X^{\mathrm{bd}}$, any open
 set~$\Omega^\prime\supset\!\supset\Omega$, any two maps 
 $u_1$, $u_2\in (L^\infty\cap W^{1, k})(\Omega^\prime, \, \R^m)$ with
 $u_1 = u_2 = g$ on~$\partial\Omega$ (in the sense of traces)
 and a.e. $y_1$, $y_2\in\R^m$ 
 there exists a chain~$R\in\M_{d-k+1}(\overline{\Omega}; \, \pi_{k-1}(\NN))$ such that
 \begin{equation} \label{cobordism1}
  \S_{y}(u_2)\mres\overline{\Omega} - \S_{y}(u_1)\mres\overline{\Omega} = \partial R.
 \end{equation}
 Let~$\delta_0 := \dist(\NN, \, \X)$.
 If, in addition, $g$ takes values in~$\NN$ then for a.e.~$y_1$, $y_2\in\R^m$ with
 $|y_1|<\delta_0$, $|y_2|<\delta_0$ there exists a 
 chain~$R\in\M_{d-k+1}(\overline{\Omega}; \, \pi_{k-1}(\NN))$ such that
 \begin{equation} \label{cobordism2}
  \S_{y_2}(u_2)\mres\overline{\Omega} - \S_{y_1}(u_1)\mres\overline{\Omega} = \partial R.
 \end{equation}
\end{prop}

As above, the previous result need not be true for~$u_1\in X$, $u_2\in X$, because
the restrictions~$\S_{y}(u_1)\mres\overline{\Omega}$, $\S_{y}(u_2)\mres\overline{\Omega}$
may not be well-defined. However, when~$u_1$, $u_2$ are merely in~$X$
and have the same trace at the boundary it is possible to show that,
for a.e.~$y$, there exists a chain~$R$ of finite mass such that 
$\spt(\S_{y}(u_2) - \S_{y}(u_1) - \partial R)\subseteq\R^d\setminus\Omega$
(this follows by Proposition~\ref{prop:Stop_continuity} below).

Finally, let us mention an additional property of $\S_y(u)$, in case $u$ is
an $\NN$-valued map.

\begin{prop} \label{prop:N-valued}
 As above, let~$\delta_0 := \dist(\NN, \, \X)$. {\BBB Then, we have:
 \begin{enumerate}[label=(\roman*)]
  \item for any $u\in W^{1, k-1}(\Omega, \, \NN)$ and a.e. 
  $y_1$, $y_2\in\R^m$ with $|y_1|<\delta_0$, $|y_2|<\delta_0$ 
  there holds $\S_{y_1}(u) = \S_{y_2}(u)$.
  
  \item If $g\in W^{1-1/k, k}(\partial\Omega, \, \NN)$,
  then $\S^{\mathrm{bd}}_{y_1}(g) = \S^{\mathrm{bd}}_{y_2}(g)$ for a.e.~$y_1$, $y_2$ with
  $|y_1|<\delta_0$, $|y_2|<\delta_0$.
  
  \item If $u\in W^{1, k}(\Omega, \, \NN)$,
  then $\S_y(u) = 0$ for a.e. $y\in\R^m$ with~$|y|<\delta_0$.
 \end{enumerate} }
\end{prop}
In case $u$ is $\NN$-valued and $|y|<\delta_0$, the chain~$\S_y(u)$
actually agrees with the topological singular set as defined by Pakzad
and Rivi\`ere in~\cite{PakzadRiviere} (see Section~\ref{subsect:Sobolev}).

\subsection{The case of smooth maps}
\label{subsect:smooth}

We first carry out the construction of~$\S_y(u)$ for a smooth map~$u$. In order to control 
the behaviour of~$u$ at the boundary, we assume that~$\Omega$ is compactly contained
in a domain~$\Omega^\prime\subseteq\R^d$, and we assume that $u$ is smoothly defined
on~$\Omega^\prime$, with~$\|u\|_{L^\infty(\Omega^\prime)}\leq\Lambda$.
Throughout this section, we also tacitly assume that the condition~\eqref{H} is satisfied.

\paragraph*{Construction of~$\S_y(u)$.}

Recall from Lemma~\ref{lemma:X} that, as a consequence of~\eqref{H}, there exists a smooth 
rectraction $\RR\colon\R^m\setminus \X\to\NN$, where~$\X$ is a smooth $(m-k)$-complex.
Let~$K$ be a $(m-k)$-dimensional cell of~$\X$, and let~$\boldsymbol{\sigma}$ be a smooth~$(m-k)$-vector field that defines an orientation on~$K$.
For any~$x\in K\setminus\partial K$ and~$r>0$, we consider the~$k$-dimensional disk~$D^k_r(x) := B^m_r(x)\cap (\T_x K)^\perp$, 
with the orientation induced by~$\boldsymbol{\sigma}$. We suppose that~$r$ is so small that~$D^k_r(x)\cap \X = D^k_r(x)\cap K = \{x\}$.
Then, we denote by~$\gamma(K, \, \boldsymbol{\sigma}) := [\RR, \, \partial D^k_r(x)]\in\pi_{k-1}(\NN)$ 
the homotopy class of~$\RR$ restricted to~$\partial D^k_r(x)\simeq\SS^{k-1}$.
(One easily checks that~$\gamma(K, \, \boldsymbol{\sigma})$ does not depend
on the choice of~$x$ and~$r$, but only on~$K$ and~$\boldsymbol{\sigma}$.)

By applying Thom parametric transversality theorem 
(see e.g.~\cite[Theorem~2.7 p.~79]{Hirsch}) to the map
$(x, \, y)\in\Omega^\prime\times\R^m\mapsto u(x) - y$, 
which is smooth and has surjective differential at every point,
we deduce that, for a.e.~$y\in\R^m$, the map~$u - y$ is transverse to all the cells of~$\X$.
Therefore, for any~$j$-cell $K$ of~$\X$ with~$m-d\leq j \leq m-k$,
the set {\BBB $(u-y)^{-1}(K)$} is a smooth $(d-m+j)$-submanifold 
of~$\Omega^\prime$ with~$\partial((u-y)^{-1}(K))\subseteq (u-y)^{-1}(\X^{j-1})$, 
while $(u-y)^{-1}(K)=\emptyset$ if~$j < m-d$.
We subdivide each~$(u-y)^{-1}(K)$ into $(d-m+j)$-cells, in such a way to make $(u-y)^{-1}(\X)$ a smooth, finite complex.
Using Thom transversality theorem again, we see that, for a.e.~$y\in\R^m$, the intersection of any cell of~$(u-y)^{-1}(\X)$ 
with~$\partial\Omega$ is a smooth manifold. Therefore, up to further
subdivision, we can assume that each cell of~$(u-y)^{-1}(\X)$
is {\BBB contained either in~$\overline{\Omega}$ or
in~$\Omega^\prime\setminus\Omega$.}

Let~$H$ be a~$(d-k)$-cell of~$(u-y)^{-1}(\X)$, oriented by a smooth $(d-k)$-unit vector field~$\boldsymbol{\tau}$.
By construction, the image $(u-y)(H)$ is contained in a $(m-k)$-cell~$K$ of~$\X$.
Let~$\boldsymbol{\sigma}$ be a smooth $(m-k)$-vector field associated with the orientation of~$K$.
We define
\[
 \epsilon(H, \boldsymbol{\tau}, \, \, K, \, \boldsymbol{\sigma}) 
                := \begin{cases}
                 1  & \textrm{if } (u-y)_*((\star\boldsymbol{\tau})^{\#})\wedge\boldsymbol{\sigma} 
                      \textrm{ is a positive } m\textrm{-vector field in } \R^m\\
                 -1 & \textrm{otherwise.}
                \end{cases}
\]
Here~$(\star\boldsymbol{\tau})^{\#}$ denotes the $k$-vector field naturally associated
with the Hodge dual of~$\boldsymbol{\tau}$, which induces an orientation
on the normal bundle to~$H$, and $(u-y)_{*}((\star\boldsymbol{\tau})^{\#})$
denotes its {\BBB push-forward through~$u-y$.} Then, we define
\begin{equation} \label{Stop}
 \S_y(u) := \sum_{(H, \, \boldsymbol{\tau})} \epsilon(H, \, \boldsymbol{\tau}, \, K, \, \boldsymbol{\sigma})
 \gamma(K, \, \boldsymbol{\sigma}) \llbracket H, \, \boldsymbol{\tau}\rrbracket,
\end{equation}
where the sum is taken over all oriented $(d-k)$-cells $(H, \, \boldsymbol{\tau})$ of~$(u-y)^{-1}(\X)$.
Then, $\S_y(u)$ is a smooth $(m-k)$-chain with coefficients in~$\pi_{k-1}(\NN)$.
Each term of the sum in~\eqref{Stop} is invariant under the changes of orientation $\boldsymbol{\tau}\mapsto-\boldsymbol{\tau}$
and $\boldsymbol{\sigma}\mapsto-\boldsymbol{\sigma}$. From now on, we will omit the $\boldsymbol{\tau}$'s
and~$\boldsymbol{\sigma}$'s in the notation.


\begin{lemma} \label{lemma:local_multiplicity}
 Let~$\Sigma$ be a smoothly embedded $k$-disk that intersects transversely a $(d-k)$-cell~$H$ of~$(u-y)^{-1}(\X)\cap\overline{\Omega}$
 at a point~$x\in H\setminus\partial H$, and suppose that $\Sigma\cap (u-y)^{-1}(\X) = \Sigma\cap H = \{x\}$.
 Let~$K$ be the $(m-k)$-cell of~$\X$ that contains~$(u-y)(H)$. Then, we have
 \[
  \epsilon(H, \, K) \gamma(K) = [\RR\circ(u-y)_{*}(\partial\Sigma)] .
 \]
\end{lemma}
\begin{proof}
 Assume, for simplicity of notation only, that $y=0$ and~$u(x) = 0$. 
 Let~$D^k_r := B^d_r(x)\cap \T_x\Sigma$, for~$0 < r < \dist(x, \, \partial\Sigma)$.
 The sphere~$\partial\Sigma$ is homotopic to~$\partial D^k_r$
 (one contracts~$\partial\Sigma$ towards~$x$, then project it on the tangent space).
 Moreover, if~$r$ is small enough, $u_{|\partial D^k_r}$ is homotopic to~$\d u_{x|\partial D^k_r}$
 because $\|u - \d u_x\|_{L^\infty(\partial D^k_r)}\to 0$ as~$r\to 0$. Therefore, we have
 \begin{equation} \label{local_multiplicity1}
  \left[(\RR\circ u)_{*}(\partial\Sigma)\right] = \left[(\RR\circ u)_*(\partial D^k_r(x))\right] 
  = \left[(\RR\circ\d u_x)_{*}(\partial D^k_r(x))\right] \! .
 \end{equation}
 Now, the transversality assumption yields $\d u_x (\T_x\Sigma) + \T_0 K = \R^m$
 and hence, by a dimension argument, $\d u_x$ restricts to an isomorphism of~$\T_x\Sigma$ onto its image.
 Thus, we have
 \[
  \left[(\RR\circ\d u_x)_{*}(\partial D^k_r))\right] 
  = \sign\det(\d u_{x|\T_x\Sigma}) \left[\RR_{*}(\d u_{x}(\partial D^k_r))\right]
  = \epsilon(H, \, K)\gamma(K).
 \]
 Combining this identity with~\eqref{local_multiplicity1}, the lemma follows.
\end{proof}

\begin{lemma} \label{lemma:Stop_cycle}
 $\S_y(u)$ is a relative cycle, that is, $\partial(\S_y(u))\mres\Omega = 0$.
\end{lemma}
\begin{proof}
 By construction, $\partial(\S_y(u))$ is supported on the $(d-k-1)$-skeleton 
 of~$(u-y)^{-1}(\X)$. Let~$H$ be a~$(d-k-1)$-cell of~$(u-y)^{-1}(\X)$ that is
 contained in~$\overline{\Omega}$, and let~$H_1, \, \ldots, \, H_q$ be the $(d-k)$-cells 
 of~$(u-y)^{-1}(\X)$ that are adjacent to~$H$. By composing with a diffeomorphism, 
 we can assume without loss of generality that $H$, $H_1, \, \ldots, \, H_q$ are affine polyhedra.
 Moreover, since~$\S_y(u)$ is independent on the choice of the orientations on the cells, 
 we can choose the orientation~$\boldsymbol{\tau}_j$ of~$H_j$ in such a way 
 that~${\boldsymbol{\tau}_j}_{|H}$ agrees with the orientation of~$H$.
 Take~$x\in H\setminus\partial H$ and a radius~$r>0$ so small that,
 setting~$D^{k+1}_r(x) := B^d_r(x)\cap H^\perp$, we have
 $D^{k+1}_r(x)\cap \partial H_j = D^{k+1}_r(x)\cap H = \{x\}$,
 for any~$j\in\{1, \, \ldots, \, q\}$. Then~$\partial D^{k+1}_r(x)$ intersects 
 each~$H_j$ at a single point, which we call~$x_j$,
 and~$\RR\circ(u-y)$ restricts to a continuous map
 $\partial D^{k+1}_r(x) \setminus\{x_j\}_{j=1}^q\to\NN$. 
 Take a smooth~$k$-disk~$\Sigma\subseteq\partial D^{k+j}_r(x)$ such
 that~$x_j\in\Sigma\setminus\partial\Sigma$ for any~$j$.
 We endow~$\Sigma$ with the orientation induced by~$\partial D^{k+1}_r(x)$. 
 Finally, for each~$j$ we take a small $k$-disk~$\Sigma_j\subseteq\Sigma$,
 in such a way that~$x_j\in\Sigma_j\setminus\partial\Sigma_j$ 
 and the~$\Sigma_j$'s are pairwise disjoint.
 By Lemma~\ref{lemma:local_multiplicity}, the multiplicity of~$\S_y(u)$ at~$H_j$
 is equal to~$[\RR\circ(u-y)_{*}(\partial\Sigma_j)]$.
 Therefore, with our choice of the orientation, we have
 \[
  \begin{split}
   \textrm{multiplicity of } \partial \S_y(u) \textrm{ at } H 
   &= \sum_{j=1}^k \left(\textrm{multiplicity of } \S_y(u) \textrm{ at } H_j\right) \\ 
   &= \sum_{j=1}^k \left[\RR\circ(u-y)_{*}(\partial\Sigma_j)\right] 
   = \left[\RR\circ(u-y)_{*}(\partial\Sigma)\right].
  \end{split}
 \]
 On the other hand, $\RR\circ(u-y)_{|\partial\Sigma}$ is null-homotopic,
 because~$\RR\circ(u-y)$ is continuous on the 
 set~$\partial D^{k+1}_r(x)\setminus\Sigma$, which is diffeomorphic to a $k$-disk.
 Thus, we have~$\partial(\S_y(u))\mres H = 0$.
\end{proof}

In the rest of this section, we check that $\S_y(u)$ satisfies \eqref{S_intersection}--\eqref{S_flat}
\emph{in case $u$ is smooth}. The extension to the Sobolev case is left to 
Section~\ref{subsect:Sobolev}.

\paragraph*{$\S_y(u)$ satisfies~\eqref{S_intersection}.}

{\BBB By applying the deformation theorem~\cite[Theorem~7.3]{Fleming}, we can write
\begin{equation} \label{intersection0}
 R = A + \partial B + P,
\end{equation}
where~$A$ is a~$k$-chain with finite mass, $B$ is a
$(k+1)$-chain with finite mass, and~$P$ is a polyhedral~$k$-chain:
\begin{equation*}
 P = \sum_{\alpha=1}^q \lambda_\alpha \llbracket K_\alpha\rrbracket,
\end{equation*}
for some~$\lambda_\alpha\in\Z$, and some affine closed
$k$-polyhedra $K_\alpha$. Since~$\S_y(u)$ is a relative cycle 
(Lemma~\ref{lemma:Stop_cycle}), and since we have assumed 
that~$\spt(\S_y(u))\cap\spt(\partial R)=\emptyset$, we can make sure that}
\begin{gather}
 \spt(\S_y(u))\cap\spt(\partial P) = \emptyset, \qquad \spt(\partial\S_y(u))\cap\spt(P) = \emptyset, \label{intersection1} \\
 \spt(\S_y(u))\cap\spt(A) = \emptyset, \qquad \spt(\partial\S_y(u))\cap\spt(B) = \emptyset. \label{intersection2}
\end{gather}
By the transversality theorem, we can assume without
loss of generality that 
$(u-y)^{-1}(\X^{m-k-1})\cap K_\alpha =\emptyset$,
$\spt(\S_y(u))\cap\partial K_\alpha = \emptyset$
and $K_\alpha$ is transverse to~$\spt(\S_y(u))$ for any~$\alpha$.
Thanks to Lemma~\ref{lemma:intersection_product} and~\eqref{intersection2}, 
we have $\I(\S_y(u), \, A) = \I(\S_y(u), \, \partial B) = 0$.
Then, by bilinearity of the intersection product, we obtain
\[
 \I(\S_y(u), \, R) = \sum_{\alpha=1}^q \lambda_\alpha \I(\S_y(u), \, \llbracket K_\alpha\rrbracket).
\]
Due to~\eqref{intersection2} and the assumption $(u-y)^{-1}(\X^{m-k-1})\cap K_\alpha =\emptyset$,
$\RR\circ(u-y)$ is well-defined and continuous in a neighbourhood of~$\spt A$.
Therefore, taking the homology classes in~\eqref{intersection0}, we deduce
\[
 \left[\RR\circ(u-y)_{*}(\partial R)\right] = 
 \sum_{\alpha=1}^q \lambda_\alpha \left[\RR\circ(u-y)_{*}\llbracket \partial  K_\alpha\rrbracket\right].
\]
Thus, it suffices to show that $\I(\S_y(u), \, \llbracket K_\alpha\rrbracket) = \left[\RR\circ(u-y)_{*}\llbracket \partial K_\alpha\rrbracket\right]$.
Because we assumed that~$K_\alpha$ is transverse to~$\spt(\S_y(u))$, their intersection is a finite set.
Using again additivity on both sides, we reduce to the case~$\#(\spt(\S_y(u))\cap K_\alpha) = 1$,
and then~\eqref{S_intersection} follows by Lemma~\ref{lemma:local_multiplicity}.
\qed

\paragraph*{$\S_y(u)$ satisfies~\eqref{S_spt}.}

{\BBB We claim that $\RR\circ(u-y)\in W^{1, k-1}(\Omega,\, \R^m)$ for a.e.~$y$.
Once this claim is proved, we can deduce that 
$\RR\circ(u-y)\in H^{1, k-1}_{\mathrm{loc}}(\Omega\setminus\spt(\S_y(u)), \, \NN)$ for a.e.~$y$
by a ``removal of the singularity'' technique, as in~\cite{Bethuel-Density}
or in~\cite[Theorem~II]{PakzadRiviere}. }
Recall that, by assumption, $\|u\|_{L^\infty(\Omega)}\leq\Lambda$. If
\begin{equation} \label{support-M}
 |y| > M := \Lambda + \sup_{z\in \X} |z|,
\end{equation}
then $(u - y)^{-1}(\X) = \emptyset$ and~$\RR\circ(u-y)$ is smooth on~$\overline{\Omega}$.
Thus, we only need to consider the case~$y\in B^m_M$.
We can now apply the arguments in~\cite[Lemma~2.3]{HKL} or~\cite[Lemma~6.2]{HangLin},
which we recall below for the convenience of the reader,
to show that~$\RR\circ(u-y)\in W^{1, k-1}(\Omega,\, \R^m)$ for~a.e.~$y$.
In fact, we will prove a slightly stronger statement,
because it will be useful later on. 

\begin{lemma}\label{lemma:projection_cont}
 For any~$v\in X := (L^\infty\cap W^{1, k-1})(\Omega, \, \R^m)$, 
 let $\Phi(v)\colon y\in\R^m\mapsto\RR\circ(v-y)$. Then, $\Phi$
 is a well-defined and continuous operator
 \[
  \Phi\colon X \to L^1_{\mathrm{loc}}(\R^m, \, W^{1, k-1}(\Omega, \, \NN)).
 \]
 Moreover, for any positive~$M$, $\Lambda$ and any~$v\in X$
 such that $\|v\|_{L^\infty(\Omega)}\leq\Lambda$, there holds
 \begin{equation} \label{HKL-estimate}
  \int_{B^m_M} \|\nabla (\RR\circ(v - y))\|_{L^{k-1}(\Omega)}^{k-1} \, \d y \leq 
  C_\Lambda \|\nabla v\|_{L^{k-1}(\Omega)}^{k-1},
 \end{equation}
 where~$C_\Lambda$ is a positive constant that only depends
 on~$M$, $\Lambda$, $k$, $\NN$ and $\RR$.
\end{lemma}
\begin{proof}
 We first remark the following useful fact, which is the essence 
 of the proof of~\cite[Lemma~2.3]{HKL}: for any positive numbers~$M$, $\Lambda$,
 any $v\in L^\infty(\Omega, \, \R^m)$ with~$\|v\|_{L^\infty(\Omega)}\leq\Lambda$,
 any measurable $w\colon\Omega\to[0, \, +\infty)$ and any Borel function
 $f\colon\R^m\to[0, \, +\infty)$, there holds
 \begin{equation} \label{trick}
  \int_{B^m_M} \left(\int_{\Omega} w(x) f(v(x) - y)\,\d x\right) \d y \leq 
  \int_{\Omega} w(x)\,\d x \ \int_{B^m_{M + \Lambda}}f(z) \, \d z.
 \end{equation}
 This follows by applying Fubini theorem, then making the change of 
 variable~$z := v(x) - y$ in the integral with respect to~$y$. 
 Another useful fact we will use in the proof is that
 \begin{equation} \label{RR_int}
  \nabla\RR\in L^{k-1}_{\mathrm{loc}}(\R^{m}, \, \R^{m\times m}).
 \end{equation}
 Indeed, $|\nabla\RR|\leq C \dist(\cdot, \, \X)^{-1}$
 by Lemma~\ref{lemma:X}, and~$\dist(\cdot, \, \X)^{-{k+1}}$ is locally
 integrable on~$\R^m$ 
 because~$\X$ is, up to a bounded change of metric in~$\R^m$,
 a finite union of simplices of codimension~$k$.
 
 Let us now check that~$\Phi$ is well-defined. For any~$v\in X$, the set
 \[
  N := \left\{(x, \, y)\in\Omega\times\R^m\colon v(x) - y\in \X\right\}
 \]
 is measurable and $\H^{d+m}(N)=0$, because each slice
 $N\cap(\{x\}\times\R^m) = v(x) - \X$ has dimension~$m-k$.
 By Fubini theorem, it follows that $\H^d(N\cap(\Omega\times\{y\})) = 0$
 for a.e.~$y\in\R^m$, so $\RR\circ(u - y)$ is well-defined for a.e.~$y\in\R^m$,
 and belongs to~$L^k(\Omega, \, \NN)$.
 By the chain rule, for a.e.~$(x, \, y)\in\Omega\times\R^m$ we have
 \[
  |\nabla(\RR\circ(v(x) - y))| = |(\nabla\RR)(v(x) - y)| |\nabla v(x)|
 \]
 and thus \eqref{HKL-estimate} follows by applying~\eqref{trick}
 with~$f = |\nabla\RR|^{k-1}$, $w = |\nabla v|^{k-1}$ and using~\eqref{RR_int}.
 
 It only remains to check the continuity of~$\Phi$.
 Let~$(v_j)_{j\in\N}$ be a sequence such that $v_j\to v$ in~$X$ as~$j\to+\infty$,
 and let~$\Lambda > 0$ be such that $\|v_j\|_{L^\infty(\Omega)}\leq\Lambda$
 for any~$j\in\N$. Up to extraction of a subsequence, 
 we assume that $v_j\to v$ a.e. Let~$M > 0$ be fixed.
 By Fubini theorem and a change of variable as in~\eqref{trick}, we obtain
 \[
  \int_{B^m_M} \|\RR\circ(v_j - y) - \RR\circ(v - y)\|^{k-1}_{L^{k-1}(\Omega)} \, \d y
  \leq\int_{\Omega} \int_{B^m_{M+\Lambda}} \left(
  \abs{\RR(z + \tilde{v}_j(x)) - \RR(z)}^{k-1} \, \d y \right) \d x,
 \]
 where~$\tilde{v}_j := v_j - v$. Since $\RR(z + \tilde{v}_j(x))\to\RR(z)$
 for any~$z\in\R^m\setminus\X$ and a.e.~$x\in\Omega$, Lebesgue's dominated 
 convergence theorem implies that the right hand side converges to zero as~$j\to+\infty$.
 Now, for fixed~$\varepsilon > 0$, let~$\varphi\in C^\infty(\R^m, \, \R^{m\times m})$ 
 be such that $\|\nabla\RR - \varphi\|_{L^{k-1}(B^m_{M+\Lambda})}^{k-1}\leq\varepsilon$.
 The chain rule implies
 \[
  \int_{B^m_M} \|\nabla(\RR\circ(v_j - y) - \RR\circ(v - y))\|^{k-1}_{L^{k-1}(\Omega)} \, \d y
  \leq 4^{k-2} (I_1 + I_2 + I_3 + I_4),
 \]
 where
 \begin{gather*}
  I_1 := \int_{B^m_M} \left(\int_\Omega \abs{(\nabla\RR)(v_j - y)}^{k-1}
  |\nabla v_j - \nabla v|^{k-1} \, \d\H^d \right) \d y, \\
  I_2 := \int_{B^m_M} \left(\int_\Omega |\nabla v|^{k-1} 
  \abs{(\nabla\RR)(v_j - y) - \varphi(v_j - y)}^{k-1} \, \d\H^d \right) \d y, \\
  I_3 := \int_{B^m_M} \left(\int_\Omega |\nabla v|^{k-1} 
  \abs{\varphi(v_j - y) - \varphi(v - y)}^{k-1} \, \d\H^d \right) \d y, \\
  I_4 := \int_{B^m_M} \left(\int_\Omega |\nabla v|^{k-1} 
  \abs{\varphi(v - y) - (\nabla\RR)(v - y)}^{k-1} \, \d\H^d \right) \d y.
 \end{gather*}
 We apply~\eqref{trick} to each of this integrals. For the first one, we obtain
 \[
  I_1 \leq \norm{\nabla v_j - \nabla v}^{k-1}_{L^{k-1}(\Omega)} 
  \int_{B^m_{M+\Lambda}}\abs{\nabla\RR(z)}^{k-1}\, \d z,
 \]
 and the integral with respect to~$z$ in the right hand side is finite,
 due to~\eqref{RR_int}. As for~$I_2$, we have
 \[
  I_2 \leq \norm{\nabla v}^{k-1}_{L^{k-1}(\Omega)} 
  \norm{\nabla\RR - \varphi}^{k-1}_{L^{k-1}(B^m_{M+\Lambda})} 
  \leq\varepsilon \norm{\nabla v}^{k-1}_{L^{k-1}(\Omega)} \! ,
 \]
 and the same holds for~$I_4$. Finally, for~$I_3$ we get
 \[
  I_3 \leq \int_{\Omega} |\nabla v(x)|^{k-1} \left(\int_{B^m_{M+\Lambda}}
  \abs{\varphi(z + \tilde{v}_j(x)) - \varphi(z)}^{k-1}\, \d z\right) \d x
 \]
 where~$\tilde{v}_j := v_j - v$, and again we can apply Lebesgue's
 dominated convergence theorem to show that the right hand side tends 
 to zero as~$j\to+\infty$. Putting all together, we deduce
 \[
  \limsup_{j\to+\infty} \int_{B^m_M} \|\RR\circ(v_j - y) 
  - \RR\circ(v - y)\|^{k-1}_{W^{1, k-1}(\Omega)} \, \d y \leq
  4^{k - 3/2} \varepsilon
  \norm{\nabla v}^{k-1}_{L^{k-1}(\Omega)}
 \]
 for arbitrary~$\varepsilon$, $M$, and hence the lemma follows.
\end{proof}

\paragraph*{$\S_y(u)$ satisfies~\eqref{S_mass}.}

Pick a positive constant~$C$ such that $|\gamma(K)|\leq C$ for any $(m-k)$-cell $K$ of~$\X$.
By the definition~\eqref{Stop} of~$\S_y(u)$, we have
\begin{equation*}
 \M(\S_y(u)\mres\overline{\Omega}) \leq C \sum_{K} \H^{d - k}\left((u-y)^{-1}(K)\cap\overline{\Omega}\right) \! ,
\end{equation*}
the sum being taken over all~$(m-k)$-cells of~$\X$, and~$\S_y(u)=0$ if~$|y|>M$ where~$M$ is defined in~\eqref{support-M}.
Since~$\X$ only contains a finite number of cells,
each of which is bilipschitz equivalent to an affine $(m-k)$-polyedron, it suffices to show
\begin{equation} \label{mass1}
 \int_{[-M, \, M]^m} \H^{d-k}\left((u-y)^{-1}(V)\cap\overline{\Omega}\right) \,\d y \leq C \norm{\nabla u}^k_{L^k(\Omega)},
\end{equation}
where~$V$ is an affine $(m-k)$-subspace of~$\R^m$. By composing with an isometry, we can assume without loss of generality that
$V = \left\{y\in\R^m\colon y_1=\ldots = y_k=0 \right\}$. We denote the variable $y = (z, \, z^\prime)\in V^\perp\times V$
and let~$p_\perp\colon\R^m\to V^\perp$ be the orthogonal projection onto~$V^\perp$.
Then, \eqref{mass1} can be rewritten as
\[
 \int_{[-M, \, M]^k\times[-M, \, M]^{m-k}} \H^{d-k}\left((p_\perp\circ u)^{-1}(z)\cap\overline{\Omega}\right) 
 \, \d z \, \d z^\prime \leq C \norm{\nabla u}^k_{L^k(\Omega)}
\]
and this inequality follows from the coarea formula, applied 
to the smooth function~$p_\perp\circ u\colon\overline{\Omega}\to V^\perp\simeq\R^k$.
This concludes the proof of~\eqref{S_mass}. 
\qed

\paragraph*{$\S_y(u)$ satisfies \eqref{S_boundary} and~\eqref{S_flat}.}
\label{subsect:continuity}

Properties~\eqref{S_boundary} and~\eqref{S_flat} follow at once from the result below.

\begin{prop} \label{prop:Stop_continuity}
 Let~$\pi\colon[0, \, 1]\times\R^d\to\R^d$ denote the canonical projection.
 Let~$u_0$, $u_1$ be two smooth maps~$\Omega^\prime\to\R^m$ and let
 $u\colon[0, \, 1]\times\Omega^\prime\to\R^m$ be given by~$u(t, \, x):=(1-t)u_0(x) + tu_1(x)$.
 For a.e.~$y\in\R^m$ there holds
 \begin{equation} \label{Stop_cobordism}
  \S_y(u_1) - \S_y(u_0) = \partial \left(\pi_{*}\S_y(u)\right)
  \qquad \textrm{in }  \Omega^\prime.
 \end{equation}
 If, in addition, $u_0 = u_1$ on~$\Omega^\prime\setminus\Omega$, then $\pi_{*}\S_y(u)$
 is supported in~$\overline{\Omega}$ for a.e.~$y\in\R^m$. Finally, for any positive~$\Lambda$ 
 there exists a constant~$C_\Lambda$ such that, if~$u_0$, $u_1$ 
 satisfy~$\|u_0\|_{L^\infty(\Omega^\prime)}\leq\Lambda$,
 $\|u_1\|_{L^\infty(\Omega^\prime)}\leq\Lambda$, then
 \begin{equation} \label{Stop_continuity}
  \int_{\R^k} \M(\pi_*\S_y(u)\mres\overline{\Omega})\d y \leq C_\Lambda
  \int_\Omega \abs{u_1 - u_0}\left(\abs{\nabla u_0}^{k-1} + \abs{\nabla u_1}^{k-1}\right)\!.
 \end{equation}
\end{prop}

Once the proposition is proved, in order to show~\eqref{S_boundary} we apply Proposition
\ref{prop:Stop_continuity} with~$u_0$ identically equal to~$0$, 
and notice that $\S_y(0)=0$ for any $y\in\R^m\setminus\X$.
{\BBB Moreover, \eqref{Stop_cobordism} and Lemma~\ref{lemma:relative_flat} imply that
$\F_\Omega(\S_y(u_1) - \S_y(u_0))\leq\M(\pi_*\S_y(u)\mres\Omega)$,
so~\eqref{Stop_continuity} becomes
\[
  \int_{\R^k} \F_\Omega(\S_y(u_1) - \S_y(u_0)) \d y \leq C_\Lambda
  \int_\Omega \abs{u_1 - u_0}\left(\abs{\nabla u_0}^{k-1} + \abs{\nabla u_1}^{k-1}\right)
\]
and \eqref{S_flat} is proved. The continuity of~$\S$ also follows. 
Indeed, if~$(u_j)_{j\in\N}$ is a sequence of smooth maps that converge to~$u$ in~$X$
then, by taking a subsequence such that the $|\nabla u_j|$'s 
are dominated and applying Lebesgue convergence
theorem, we conclude that  $\|\S(u_j) - \S(u)\|_Y\to 0$.}


The proof of Proposition~\ref{prop:Stop_continuity}
is in some sense a refinement of~\eqref{S_mass}.
It will be convenient to work in the setting of differential forms and currents.
We follow here the notation of~\cite[Section~7.4]{ABO1}. 
Given a smooth map $v\colon [0, \, 1]\times\Omega^\prime\to\R^k$, we define the
Jacobian~$\mathrm{J}v$ as the pull-back of the standard volume form on~$\R^k$ through~$v$,
i.e. $\mathrm{J}v := v^*(\d y^1\wedge\ldots\wedge\d y^k)$.
If we denote by~$(x^1, \, \ldots, \, x^d)$ the coordinates on~$\Omega^\prime$
and by~$x^0 = t$ the coordinate in~$[0, \, 1]$, then we can write
\begin{equation} \label{jacobian}
 \mathrm{J} v = \sum_{\alpha\in I(k, \, d)} \det(\partial_{\alpha} v) \, \d x^\alpha,
\end{equation}
where $I(k, d)$ is the set of multi-indices~$\alpha = (\alpha_1, \, \ldots, \, \alpha_k)\in\N^k$
such that~$0 \leq \alpha_1 < \alpha_2 < \ldots < \alpha_k \leq d$, and
\[
 \partial_\alpha v := \left(\partial_{\alpha_1}v, \, \ldots, \, \partial_{\alpha_k}v\right), \qquad
 \d x^\alpha := \d x^{\alpha_1}\wedge\ldots\wedge\d x^{\alpha_{k}}.
\]
For any regular value~$y\in\R^k$ of~$v$ and any~$x\in v^{-1}(y)$, the Hodge 
dual~$\star\mathrm{J}v(x)$ is a simple $(d-k+1)$-vector which spans~$\T_x v^{-1}(y)$.
By~\eqref{jacobian}, we have $|\star\mathrm{J}v|^2 = |\mathrm{J}v|^2 = 
\det((\nabla v)(\nabla v)^{\mathsf{T}})$.
We define~$N_y(v)$ as the rectifiable current 
(in the ambient space~$[0, \, 1]\times\Omega^\prime$) supported by~$v^{-1}(y)$,
with orientation given by $\star\mathrm{J}v/|\star\mathrm{J}v|$ and constant multiplicity~$1$.
$N_y(v)$ can be identified with an element 
of~$\M_{d-k+1}([0, \, 1]\times\overline{\Omega}^\prime; \, \Z)$. 
The mass of~$N_y(v)$, whether it be regarded as a current or as a flat chain 
with coefficients in~$\Z$, coincides with the Hausdorff measure of the set~$N_y(v)$,
because $N_y(v)$ is rectifiable.

\begin{lemma} \label{lemma:oriented_coarea}
 Let~$v_0$, $v_1$ be two smooth maps~$\Omega^\prime\to\R^k$,
 and let $v\colon[0, \, 1]\times\Omega^\prime\to\R^k$ be defined by 
 $v(t, \, x) := (1-t)v_0(x) + tv_1(x)$. Let~$K\csubset[0, \, 1]\times\Omega^\prime$
 be a Borel set. Then, there holds
 \[
  \int_{\R^k} \M\left(\pi_{*}(N_y(v)\mres K)\right) \,\d y \leq C
  \int_{\pi(K)} \abs{v_1 - v_0}\left(\abs{\nabla v_0}^{k-1} + \abs{\nabla v_1}^{k-1}\right)
 \]
\end{lemma}
\begin{proof}
 By definition of the mass of a current, we can write
 \begin{equation} \label{coarea1}
  \M\left(\pi_{*}(N_y(v)\mres K)\right)
  = \sup_\omega \langle N_y(v)\mres K, \, \pi^*\omega\rangle,
 \end{equation}
 where the sup is taken over all smooth $(d-k+1)$-forms~$\omega$ supported in~$\Omega^\prime$,
 such that the \emph{comass} norm~$\|\omega(x)\|\leq 1$ for any~$x\in \Omega^\prime$. 
 (This condition means $\langle\omega(x), \, \xi\rangle\leq 1$ for any unit, 
 \emph{simple} $(d-k+1)$-vector~$\xi$ and any~$x\in\Omega^\prime$).
 Any such form~$\omega$ can be written as
 \begin{equation*}
  \omega = \sum_{\beta\in I(d-k+1, \, d)\colon \beta_1 > 0} \omega_\beta \, \d x^\beta
 \end{equation*}
 for some functions $\omega_\beta\in C^\infty_{\mathrm{c}}(\Omega^\prime)$ which satisfy
 \begin{equation} \label{coarea2}
  \sum_{\beta} \omega^2_\beta(x) = |\omega(x)|^2 \leq C \|\omega(x)\|^2 
  \leq C \qquad \textrm{for any } x\in\Omega^\prime
 \end{equation}
 where~$C = C(d, \, k)$ is a positive constant. 
 Using the properties of~$\star$ and~\eqref{jacobian}, we compute
 \begin{equation*} 
  \begin{split}
   \langle N_y(v)\mres K, \, \pi^*\omega\rangle
   &= \int_{v^{-1}(y) \cap K} \langle \pi^*\omega, \, \frac{\star\mathrm{J} v}{|\star\mathrm{J} v|}  \rangle \,\d\H^{d-k+1} \\
   &= (-1)^{k(d-k+1)}\int_{v^{-1}(y) \cap K} \frac{\star\left(\pi^*\omega\wedge\mathrm{J} v \right)}{|\star\mathrm{J} v|} \,\d\H^{d-k+1} \\
   &\leq \sum_{\alpha\in I(k, \, d)\colon \alpha_1 = 0}\int_{v^{-1}(y)\cap K}
   |\omega_{\bar{\alpha}}| \frac{|\det(\partial_{\alpha}v)|}{|\star\mathrm{J} v|} \, \d\H^{d-k+1},
  \end{split}
 \end{equation*}
 where~$\bar{\alpha}$ denotes the unique element of~$I(d-k+1, \, d)$ that complements~$\alpha$.
 Recalling the definition of~$v$, for any~$\alpha\in I(k, \, d)$ such that~$\alpha_1 = 0$ we obtain
 \[
  |\det(\partial_\alpha v)| \leq \abs{\partial_t v} \abs{\nabla_x v}^{k-1}
  \leq |v_0 - v_1| \left(\abs{\nabla v_0} + \abs{\nabla v_1}\right)^{k-1}.
 \]
 Then, using~\eqref{coarea1} and~\eqref{coarea2} as well, we have
 \begin{equation*} 
  \M\left(\pi_{*}(N_y(v)\mres K)\right) \leq C
  \int_{v^{-1}(y)\cap K} \frac{|v_0 - v_1| \left(\abs{\nabla v_0}
  + \abs{\nabla v_1}\right)^{k-1}}{|\star\mathrm{J} v|} \,\d\H^{d-k+1}.
 \end{equation*}
 By integrating this inequality with respect to~$y\in\R^k$, and applying the coarea formula,
 we conclude that
 \[
  \int_{\R^k} \M\left(\pi_{*}(N_y(v)\mres K)\right) \d y \leq C
  \int_{K}  |v_0 - v_1| \left(\abs{\nabla v_0} + \abs{\nabla v_1}\right)^{k-1} \, \d\H^{d+1}
 \]
 whence the lemma follows.
\end{proof}

\begin{proof}[Proof of Proposition~\ref{prop:Stop_continuity}]
 We first prove~\eqref{Stop_cobordism}. Pick~$y\in\R^k$ such that~$u_0 - y$, $u_1 - y$ 
 and~$u - y$, together with their restrictions to~$\partial\Omega$,
 are transverse to all the cells of~$\X$. Then, up to subdivision,
 we can assume that all the cells of~$(u - y)^{-1}(\X)$
 are contained either in~$\{0, \, 1\}\times\Omega$ or in~$(0, \, 1)\times\Omega^\prime$.
 Then, \eqref{Stop_cobordism} follows by the same argument of Lemma~\ref{lemma:Stop_cycle}.
 In case~$u_0 = u_1$ out of~$\Omega$, we have $u(t, \, x) = u_0(x)$ 
 for any~$(t, \, x)\in [0, \, 1]\times(\Omega^\prime\setminus\overline\Omega)$, so
 \[
  \S_y(u)\mres([0, \, 1]\times(\Omega^\prime\setminus\overline\Omega)) 
  = \llbracket [0, \, 1] \rrbracket \times \S_y(u_0)\mres(\Omega^\prime\setminus\overline\Omega)
 \]
 and
 \[
  \pi_{*}\S_y(u)\mres(\Omega^\prime\setminus\overline{\Omega})
  = \pi_{*}\left(\S_y(u)\mres([0, \, 1]\times(\Omega^\prime\setminus\overline\Omega)) \right) 
  = \pi_{*}\llbracket [0, \, 1] \rrbracket \times \S_y(u_0)\mres(\Omega^\prime\setminus\overline\Omega) = 0.
 \]
 Thus, $\pi_{*}\S_y(u)$ is supported in $\overline\Omega$.
 
 We now prove~\eqref{Stop_continuity}. Fix~$\Lambda > 0$ such that 
 $\|u_0\|_{L^\infty(\Omega^\prime)}\leq\Lambda$ and
 $\|u_1\|_{L^\infty(\Omega^\prime)}\leq\Lambda$. Then, we have
 $\|u\|_{L^\infty([0, \, 1]\times\Omega^\prime)}\leq\Lambda$
 and so $\S_y(u_0) = \S_y(u_1) = 0$ whenever
 \begin{equation*} 
  |y| > M :=  \Lambda + \sup_{z\in \X} |z|.
 \end{equation*}
 If we choose a constant $C$ such that~$|\gamma(K)|\leq C$ for any $(m-k)$-cell $K$ of~$\X$, then
 the definition~\eqref{Stop} of~$\S_y(u)$ implies
 \begin{equation} \label{Stop_continuity2}
  \begin{split}
   \M\left(\pi_{*}\S_y(u)\mres\overline{\Omega}\right) &\leq C 
   \sum_{K} \M\left(\pi_{*}\llbracket (u-y)^{-1}(K)\rrbracket
   \mres\overline{\Omega}\right) \! .
  \end{split}
 \end{equation}
 Fix a $(m-k)$-cell $K$ of~$\X$. By composing with a diffeomorphism, we can assume
 without loss of generality that~$K$ is an affine polyhedron
 contained in the $(m-k)$-plane $V:=\{y\in\R^m\colon y_1=\ldots= y_k=0\}$.
 We denote the variable in~$\R^m$ by~$y = (z, \, z^\prime)\in V^\perp\times V$, and we
 let~$p$, $p_\perp$ be the orthogonal projections onto~$V$, $V^\perp$ respectively.
 For a suitable choice of the orientation of~$K$, we have
 \begin{equation*} 
  \llbracket (u-y)^{-1}(K)\rrbracket = 
  N_{p_\perp(y)}(p_\perp\circ u)\mres \left((p\circ u -p(y))^{-1}(K)\right).
 \end{equation*}
 Thus, for any~$z^\prime\in V$, by applying Lemma~\ref{lemma:oriented_coarea}
 to~$v := p_\perp\circ u$ and
 \[
  K_{z^\prime} :=(p\circ u - z^\prime)^{-1}(K) 
  \cap \left([0, \, 1]\times\overline{\Omega}\right) \! ,
 \]
 we obtain
 \begin{equation*}
  \begin{split}
   \int_{V^\perp}  \! \M\left(\pi_{*}\llbracket (u-(z, \, z^\prime))^{-1}(K)\rrbracket
   \mres\overline{\Omega}\right) \d z &=
   \int_{V^\perp}  \M\left(\pi_{*}(N_{z}(v)\mres K_{z^\prime})\right) \,\d z \\
   &\leq C \int_{\Omega} \abs{u_0 - u_1} 
   \left(\abs{\nabla u_0}^{k-1} + \abs{\nabla u_1}^{k-1}\right) \! .
  \end{split}
 \end{equation*}
 By integrating with respect to~$z^\prime\in V\cap B^m_M$, summing over~$K$,
 and using~\eqref{Stop_continuity2}, we obtain
 \begin{equation*} \label{Stop_continuity3}
  \int_{V^\perp\times (V\cap B^m_M)} \M\left(\pi_{*}\S_{y}(u)\mres\overline{\Omega}\right) \d y
  \leq C \int_{\Omega} \abs{u_0 - u_1} 
   \left(\abs{\nabla u_0}^{k-1} + \abs{\nabla u_1}^{k-1}\right)
 \end{equation*}
 for some constant~$C$ depending on~$M$ (hence on~$\Lambda$) and on~$\X$.
 Now, reminding that~$\S_y(u) = 0$ if~$|y|>M$, the proposition follows.
\end{proof}

\subsection{The case of Sobolev maps}
\label{subsect:Sobolev}

In the previous section, we have defined $\S_y(u)$ in case~$u$ is smooth;
we now have to extend the definition to the case $u$ belongs to a suitable Sobolev space,
and of course this is accomplished by a density argument. 
We will then provide the proof of the main theorem, Theorem~\ref{th:Stop}, 
and of Proposition~\ref{prop:boundary}.

Since~$\Omega$ is assumed to be bounded and smooth, there exist a larger
domain $\Omega^\prime\supset\!\supset\Omega$ and a linear, continuous operator
$E\colon X := (L^\infty\cap W^{1,k-1})(\Omega, \, \R^m)
\to (L^\infty\cap W^{1,k-1})(\Omega^\prime, \R^m)$ that satisfies $Eu_{|\Omega} = u$ and
\begin{equation} \label{extension_op}
   \int_{\Omega^\prime} \abs{Eu_1 - Eu_0}\left(\abs{\nabla (Eu_0)}^{k-1}
   + \abs{\nabla (Eu_1)}^{k-1}\right) \leq 
   C\int_{\Omega} \abs{u_1 - u_0}\left(\abs{\nabla u_0}^{k-1}
   + \abs{\nabla u_1}^{k-1}\right)
\end{equation}
for any~$u\in X$ and some constant~$C$ that only depends on~$\Omega$.
Such an operator can be constructed, e.g., by standard reflection about
the boundary~$\partial\Omega$.

Let $\Psi\colon\F_{d-k}(\Omega^\prime; \, \pi_{k-1}(\NN))\to\F_{d-k}(\Omega; \, \pi_{k-1}(\NN))$
be the restriction map given by Lem\-ma~\ref{lemma:flat_restriction}.
For any~$u\in E^{-1}C^\infty(\Omega^{\prime}, \, \R^m)$ and any~$y\in\R^m$, with a slight
abuse of notation, we let~$\S_y(u) := \Psi(\S_y(E u)) = \S_y(E u)\mres\Omega$.
By Proposition~\ref{prop:Stop_continuity} and~\eqref{extension_op}, this defines a uniformly
continuous operator $\S\colon E^{-1}C^\infty(\Omega^{\prime}, \, \R^m)\to Y$,
if $E^{-1}C^\infty(\Omega^{\prime}, \, \R^m)$ is given the topology of a subspace of~$X$.
Since $E^{-1}C^\infty(\Omega^{\prime}, \, \R^m)$ is dense in~$X$, we can extend~$\S$
to a continuous operator $X\to Y$, still denoted~$\S$, that safisfies~\eqref{S_flat}.
Now, before completing the proof of Theorem~\ref{th:Stop}, 
we state a useful lemma. 

\begin{lemma} \label{lemma:Stop-delta}
 Let~$\delta_0 := \dist(\NN,\, \X)$. For any smooth map
 $u\colon\Omega^\prime\to\R^m$ and a.e.~$y$, $y^\prime\in\R^m$ 
 with~$|y^\prime|<\delta_0$, there holds $\S_y(u) = \S_{y^\prime}(\RR\circ(u - y))$.
\end{lemma}

For the sake of convenience of exposition, we leave the proof of  
Lemma~\ref{lemma:Stop-delta} to Section~\ref{subsect:Sobolev-NN}.
By the continuity of~$\S$, and because $\RR\circ(u_j - y)\to\RR\circ(u - y)$
in~$W^{1, k-1}$ for a.e.~$y\in\R^m$ if~$u_j\to u$ in~$X$
(Lemma~\ref{lemma:projection_cont}), from Lemma~\ref{lemma:Stop-delta}
we derive

\begin{lemma} \label{lemma:Stop-delta-general}
 As above, let~$\delta_0 := \dist(\NN, \, \X)$. For any
 $u\in X$ and a.e.~$y$, $y^\prime\in\R^m$ 
 with~$|y^\prime|<\delta_0$, there holds $\S_y(u) = \S_{y^\prime}(\RR\circ(u - y))$.
\end{lemma}


\begin{proof}[Proof of Theorem~\ref{th:Stop}]
 We already know, by Proposition~\ref{prop:Stop_continuity} and~\eqref{extension_op},
 that $\S$ satisfies~\eqref{S_flat}; we need to check that it also satisfies
 \eqref{S_intersection}--\eqref{S_mass}. Properties~\eqref{S_boundary} and~\eqref{S_mass} 
 follow by a density argument, since we 
 have already established that they hold for smooth maps,
 using the $\F_\Omega$-lower semi-continuity of the mass,
 Lemma~\ref{lemma:relative_lsc_mass}. Property~\eqref{S_spt} can be proved by a 
 ``removal of the singularity'' technique, exactly as in~\cite[Theorem~II]{PakzadRiviere}.
 
 We now check~\eqref{S_intersection}. For fixed $u\in X$ and~$y\in\R^m$,
 take a chain~$R\in\N_{k}(\R^d; \, \Z)$ such that~$\spt R\subseteq\Omega$ and
 $\spt(\partial R)\cap\spt(\S_y(u)) =\emptyset$.
 Let~$U$ be an open neighbourhood of~$\spt(\partial R)$ such that 
 $U\cap\spt(\S_y(u))=\emptyset$.
 Taking a smaller~$U$ if necessary, we can assume that~$U$ retracts 
 by deformation over~$\spt(\partial R)$.
 Moreover, we can assume without loss of generality that~$\partial R$
 is polyhedral. Indeed, due to the Deformation Theorem~\cite[Theorem~7.3]{Fleming},
 there is a $k$-chain of finite mass~$\tilde{R}$, supported in~$U$, such that
 $\partial\tilde{R} - \partial R$ is polyhedral. By Lemma~\ref{lemma:intersection_product},
 and because $\spt\tilde{R}\subseteq U\subseteq\R^d\setminus\spt(\S_y(u))$,
 we have $\I(\S_y(u), \, \tilde{R})=0$, so we may redefine $R := R - \tilde{R}$.

 Under these conditions, we can apply~\eqref{S_spt} and
 Lemma~\ref{lemma:H1,p_loc} to deduce that $\RR\circ(u-y)\in H^{1,k-1}(U, \, \NN)$.
 As a consequence, we find {\BBB a sequence of open sets~$U^\prime_j$, with
 $\spt(\partial R)\subseteq U^\prime_j\csubset U$,
 and a sequence~$w_j\in C^\infty(\Omega^\prime, \, \R^m)$
 such that~$w_j(x)\in\NN$ for any~$x\in U^\prime_j$ and any~$j$,}
 and~$w_j\to\RR\circ(u-y)$ in~$X$. 
 Thus, for a.e.~$y^\prime\in\R^m$ with~$|y^\prime| < \dist(\NN, \, \X)$ we have
 $\spt(\S_{y^\prime}(w_j))\cap U^\prime_j = \emptyset$ and we can apply~\eqref{S_intersection} 
 to~$w_j$, because we have already proved~\eqref{S_intersection} for smooth maps.
 This gives
 \begin{equation} \label{S_inter1}
  \I(\S_{y^\prime}(w_j), \, R) = [\RR\circ(w_j - y^\prime)_*(\partial R)].
 \end{equation}
 Since~$w_j\to\RR\circ(u-y)$ in~$X$, using Lemma~\ref{lemma:homotopy}
 we see that 
 \begin{equation} \label{S_inter2}
  [\RR\circ(w_j - y^\prime)_*(\partial R)] = 
  [\RR\circ(\RR\circ(u - y) - y^\prime)_*(\partial R)] = [\RR\circ(u - y)_*(\partial R)],
 \end{equation}
 for~$j$ large enough and a.e.~$y$, $y^\prime$ with $|y^\prime|<\dist(\NN, \, \X)$.
 The latter identity holds because the map~$z\in\NN\mapsto\RR(z - y^\prime)$
 is homotopic to the identity on~$\NN$ (a homotopy is given
 by $(z, \, t)\in\NN\times[0, \, 1]\mapsto\RR(z - ty^\prime)$).
 As for the left-hand side of~\eqref{S_inter1}, we use again
 that~$w_j\to\RR(u-y)$ in~$X$, the continuity of~$\S$,
 Lemmas~\ref{lemma:intersection_product}
 and~\ref{lemma:Stop-delta-general} to obtain that
 \begin{equation} \label{S_inter3}
  \I(\S_{y^\prime}(w_j), \, R) = \I(\S_{y^\prime}(\RR(u-y)), \, R)
  = \I(\S_y(u), \, R)
 \end{equation}
 for~a.e. $y$, $y^\prime$, provided that $j$ is large enough
 and~$|y^\prime|$ is sufficiently small. Then, \eqref{S_intersection}
 follows from~\eqref{S_inter1}, \eqref{S_inter2} and~\eqref{S_inter3}.
 
 Finally, we prove the uniqueness part of the theorem.
 Let~$\S^\prime\colon X\to Y$ be a continuous operator that 
 satisfies~\eqref{S_intersection}, and let~$u\in C^\infty(\Omega^\prime, \, \R^m)$.
 Let~$y\in\R^m$ be such that~$u - y$ intersects transversly each cell of~$\X$, 
 and let~$B\csubset\Omega\setminus (u - y)^{-1}(\X)$ be a ball.
 Since~$\RR(u - y)$ is well-defined and smooth on~$B$,
 by~\eqref{S_intersection} we have
 \[
  \I(\S^\prime_y(u), \, R) = [\RR(u-y)_*(\partial R)] = 0
 \]
 for any $k$-disk~$R$ supported in~$B$ such that
 $\spt(\partial R)\cap\spt(\S^\prime_y(u))=\emptyset$. 
 By approximating~$\S^\prime_y(u)$ with polyhedral chains, 
 using the deformation theorem as stated in~\cite[Theorem~1.1]{White-Deformation}
 together with~\cite[Proposition~2.2]{White-Deformation}, we deduce that 
 $\S^\prime_y(u)\mres B = 0$, hence $\spt(\S^\prime_y(u))\subseteq (u - y)^{-1}(\X)$.
 However, using again~\eqref{S_intersection}, we see that the
 multiplicity of~$\S_y(u)$ and~$\S^\prime_y(u)$ must agree on~$(u - y)^{-1}(K)$,
 for any~$(m-k)$-open cell~$K$ of~$\X$. Thus, $\S_y(u) - \S^\prime_y(u)$ 
 must be supported on the $(d-k-1)$-skeleton of~$(u - y)^{-1}(\X)$, and
 thus $\S_y(u) = \S^\prime_y(u)$ because no non-trivial 
 $(d-k)$-chain can be supported on a $(d-k-1)$-dimensional set
 \cite[Theorem~3.1]{White-Deformation}. We have shown that
 $\S^\prime$ agrees with~$\S$ on smooth maps, and by continuity of~$\S^\prime$, 
 we must have $\S^\prime = \S$.
\end{proof} 

We now turn to the study of~$\S^{\mathrm{bd}}$.
Suppose that $d\geq k+1$, and let~$\Omega^\prime\supset\!\supset\Omega$ 
be an open set. For
$g\in X^{\mathrm{bd}} := (L^\infty\cap W^{1-1/k, k})(\partial\Omega, \, \R^m)$,
take a map~$u\in (L^\infty \cap W^{1, k})(\Omega^\prime, \, \R^m)$
that satisfies~$u_{|\partial\Omega} = g$ in the sense of traces.
Since~$\S_y(u)\in\F_{d-k}(\Omega^\prime; \, \pi_{k-1}(\NN))$ has 
finite mass for a.e.~$y$, due to~\eqref{S_mass}, 
the restriction $\S_y(u)\mres\Omega$ is well-defined, for a.e.~$y$.
Let~$\S^{\mathrm{bd}}_y(g) := \partial(\S_y(u)\mres\Omega)$.

\begin{proof}[Proof of Proposition~\ref{prop:boundary}]
 By construction, $\S^{\mathrm{bd}}_y(g)$ is supported
 in~$\overline{\Omega}$. On the other hand, 
 by noting that $\S_y(u)$ has no boundary inside~$\Omega^\prime$
 due to~\eqref{S_boundary}, we see that
 \begin{equation} \label{bd1}
  \S^{\mathrm{bd}}_y(g) = -\partial(\S_y(u) - \S_y(u)\mres\Omega) 
  = -\partial(\S_y(u)\mres(\R^d\setminus\Omega))
 \end{equation}
 is supported in~$\R^d\setminus\Omega$. Thus,
 $\S^{\mathrm{bd}}_y(g)\in \F_{d-k-1}(\partial\Omega; \,\, \pi_{k-1}(\NN))$
 for a.e.~$y$. In fact, the map $y\mapsto\S^{\mathrm{bd}}_y(g)$ belongs to
 $Y^{\mathrm{bd}} := L^1(\R^m, \, \F_{d-k-1}(\partial\Omega; \,\, \pi_{k-1}(\NN)))$,
 because $\F(\S^{\mathrm{bd}}_y(g)) \leq \M(\S_y(u))$ by~\eqref{flat-mass}
 and the integral of~$\M(\S_y(u))$ with respect to~$y$ is finite, due to~\eqref{S_mass}.
 We now claim that
 \begin{equation} \label{bd2}
  \S_y(u)\mres\partial\Omega = 0 \qquad \textrm{for a.e. } y\in\R^m.
 \end{equation}
 Indeed, for~$\rho\in (0, \, \dist(\Omega, \, \partial\Omega^\prime)$,
 let~$\Gamma_\rho := \{x\in\R^d\colon \dist(x, \, \partial\Omega)<\rho\}$.
 Thanks to~\eqref{S_mass} and the locality of~$\S$ 
 (Corollary~\ref{cor:locality}), we have
 \[
  \int_{\R^d} \M(\S_y(u)\mres\partial\Omega) \,\d y \leq
  \int_{\R^d} \M(\S_y(u)\mres\Gamma_\rho) \,\d y \leq C\norm{\nabla u}^k_{L^k(\Gamma_\rho)}
 \]
 and the right-hand side tends to zero as~$\rho\to 0$, so~\eqref{bd2} follows.
 As a consequence of~\eqref{bd2}, we have $\S_y(u)\mres\Omega = \S_y(u)\mres\overline{\Omega}$
 for a.e.~$y$.
 
 We check that~$\S^{\mathrm{bd}}_y(g)$ is independent of the choice of~$u$.
 Let~$u_1$, $u_2$ be two maps in $(L^\infty \cap W^{1, k})(\Omega^\prime, \, \R^m)$
 such that $u_1 = u_2 = g$ on~$\partial\Omega$ in the sense of traces.
 Define the map $u_*$ by
 \[
  u_* := \begin{cases}
          u_1 & \textrm{on } \Omega \\
          u_2 & \textrm{on } \R^d\setminus\Omega,
         \end{cases}
 \]
 and note that $u_*\in (L^\infty \cap W^{1, k})(\Omega^\prime, \, \R^m)$.
 By the locality of the operator~$\S$ (Corollary~\ref{cor:locality}), we have
 $\S_y(u_*)\mres\Omega = \S_y(u_1)\mres\Omega$, 
 $\S_y(u_*)\mres(\R^d\setminus\overline{\Omega}) = 
 \S_y(u_2)\mres(\R^d\setminus\overline{\Omega})$  and hence 
 \[
  \begin{split}
  \partial(\S_y(u_1)\mres\Omega) &= \partial(\S_y(u_*)\mres\Omega)
  \stackrel{\eqref{bd1}-\eqref{bd2}}{=} 
   -\partial(\S_y(u_*)\mres(\R^d\setminus\overline\Omega)) \\
  &= -\partial(\S_y(u_2)\mres(\R^d\setminus\overline\Omega))  
  \stackrel{\eqref{bd1}-\eqref{bd2}}{=} \partial(\S_y(u_2)\mres\Omega).
  \end{split}
 \]
 
 It only remains to prove the sequential continuity of~$\S^{\mathrm{bd}}$.
 Let~$(g_j)_{j\in\N}$ be a sequence that converges to~$g$ weakly
 in~$W^{1-1/k, k}(\partial\Omega, \, \R^m)$, and suppose that 
 $\Lambda := \sup_j\|g_j\|_{L^\infty(\partial\Omega)} <+\infty$. By
 Rellich-Kondrakov theorem, we know that $g_j\to g$
 strongly in~$L^k(\Omega, \, \R^m)$.
 We can find an open set $\Omega^\prime\supset\!\supset\Omega$
 and functions $u_j$, $u\in W^{1, k}(\Omega^\prime, \, \R^m)$
 such that $u_{j|\partial\Omega} = g_j$, $u_{|\partial\Omega} = g$
 in the sense of traces, and
 \begin{gather}
   \norm{u_j - u}_{L^k(\Omega^\prime)} 
      \leq C\norm{g_j - g}_{L^{k}(\partial\Omega)}, \label{bd3} \\
   \norm{\nabla u_j}_{L^{k}(\Omega^\prime)} 
      \leq C\norm{g_j}_{W^{1 - 1/k, k}(\partial\Omega)}, \qquad
   \norm{\nabla u}_{L^{k}(\Omega^\prime)} 
      \leq C\norm{g}_{W^{1 - 1/k, k}(\partial\Omega)}. \label{bd4}
 \end{gather}
 By a truncation argument, we can also assume that 
 $\sup_j\|u_j\|_{L^\infty(\Omega)} \leq\Lambda$, 
 $\|u\|_{L^\infty(\Omega)} \leq\Lambda$.
 For any $\rho\in (0, \, \dist(\Omega, \, \partial\Omega^\prime)$,
 let~$\Omega_\rho := \Omega\cup\Gamma_\rho = 
 \{x\in\R^d\colon \dist(x, \, \Omega)<\rho\}$.
 By applying~\eqref{flat-average} with $U = \Omega_\rho$, $H = \Omega$, 
 and using that $\F_{\Omega_\rho}\leq\F_{\Omega^\prime}$ (as a consequence
 of Lemma~\ref{lemma:relative_flat}), we obtain
 \[
  \begin{split}
   \F(\S^{\mathrm{bd}}_y(g_j) - \S^{\mathrm{bd}}_y(g))
   &\leq \F((\S_y(u_j) - \S_y(u))\mres\Omega) \\
   &\leq (1 + \rho^{-1})\F_{\Omega^\prime}(\S_y(u_j) - \S_y(u)) 
   + \M((\S_y(u_j) - \S_y(u))\mres\Gamma_\rho) \! .
  \end{split}
 \]
 We integrate with respect to~$y$ and
 apply~\eqref{Stop_weak_cont}, \eqref{S_mass} to deduce
 \[
  \begin{split}
   \int_{\R^m} \F(\S^{\mathrm{bd}}_y(g_j) - \S^{\mathrm{bd}}_y(g)) \,\d y
   & \leq C(1 + \rho^{-1})\norm{u_j - u}_{L^k(\Omega^\prime)}
   \left(\norm{\nabla u_j}^{k-1}_{L^{k}(\Omega^\prime)} 
   + \norm{\nabla u}^{k-1}_{L^{k}(\Omega^\prime)}\right) \\
   &\qquad\qquad + \norm{\nabla u_j}^{k}_{L^{k}(\Gamma_\rho)} 
   + \norm{\nabla u}^{k}_{L^{k}(\Gamma_\rho)} \\
   & \stackrel{\eqref{bd3}-\eqref{bd4}}{\leq} C(1 + \rho^{-1})\norm{g_j - g}_{L^k}
   \left(\norm{g_j}^{k-1}_{W^{1 - 1/k, k}} 
   + \norm{g}^{k-1}_{W^{1 - 1/k, k}}\right) \\
   &\qquad\qquad + \norm{\nabla u_j}^{k}_{L^{k}(\Gamma_\rho)} 
   + \norm{\nabla u}^{k}_{L^{k}(\Gamma_\rho)} \! .
  \end{split}
 \]
 By letting~$j\to+\infty$ first, and then~$\rho\to 0$,
 we deduce that $\S^{\mathrm{bd}}$ is sequentially continuous.
\end{proof}

\begin{proof}[Proof of Proposition~\ref{prop:cobordism}]
 We first prove~\eqref{cobordism1}.
 Let~$u\in (L^\infty\cap W^{1, k})(\Omega^\prime, \, \R^m)$
 be such that~$u=g$ on~$\partial\Omega$, in the sense of traces. 
 For~$j\in\{1, \, 2\}$, define
 \[
  \tilde{u}_j := \begin{cases}
                  u_j & \textrm{on } \overline{\Omega}, \\
                  u & \textrm{on } \Omega^\prime\setminus\overline{\Omega}.
                 \end{cases}
 \]
 Let~$\rho_\varepsilon$ be a standard mollifier supported in~$B^d_\varepsilon$,
 and let~$v_{j, \varepsilon} := \tilde{u}_j*\rho_\varepsilon$.  
 By taking a smaller~$\Omega^\prime$, we have that $v_{j, \varepsilon}$
 is well-defined and smooth on~$\Omega^\prime$, for any~$\varepsilon$ small enough.
 Setting~$\Omega_\varepsilon :=\{x\in\R^d\colon\dist(x, \, \Omega)<\varepsilon\}$,
 we have $v_{1,\varepsilon} = v_{2, \varepsilon}$ on~$\Omega^\prime\setminus\Omega_\varepsilon$.
 Therefore, by Proposition~\ref{prop:Stop_continuity},
 for a.e.~$y\in\R^m$ and any~$\varepsilon$ there exists a smooth 
 chain~$R_\varepsilon\in\M_{d-k+1}(\overline{\Omega}_\varepsilon; \, \pi_{k-1}(\NN))$
 such that 
 \begin{equation*}
  \S_{y}(v_{2, \varepsilon})\mres\overline{\Omega}_\varepsilon -
  \S_{y}(v_{1, \varepsilon})\mres\overline{\Omega}_\varepsilon 
  = \partial R_\varepsilon
 \end{equation*}
 and $\sup_\varepsilon\M(R_\varepsilon) <+\infty$.
 Up to extraction of a subsequence, we have
 $\F(R_\varepsilon - R)\to 0$ as~$\varepsilon\to 0$,
 for some~$R\in\M_{d-k+1}(\overline{\Omega}; \, \pi_{k-1}(\NN))$.
 Therefore, \eqref{cobordism1} follows if we show that
 $\S_{y}(v_{j, \varepsilon})\mres\overline{\Omega}_\varepsilon$
 $\F$-converges to~$\S_{y}(u_j)\mres\overline{\Omega}$, for~$j\in\{1, \, 2\}$ and a.e.~$y$.
 To this end, let us fix~$\varepsilon_0>0$ and take~$0 <\varepsilon<\varepsilon_0$.
 We apply~\eqref{flat-average} 
 and Lemma~\ref{lemma:relative_flat}, to obtain
 \[
  \begin{split}
   \F(\S_{y}(v_{j, \varepsilon})\mres\overline{\Omega}_\varepsilon - 
  \S_{y}(\tilde{u}_{2})\mres\overline{\Omega})
   \leq (1 + \varepsilon_0^{-1})\F_{\Omega^\prime}
     (\S_y(v_{j, \varepsilon}) - \S_y(v_2)) 
   + \M(\S_y(v_{j, \varepsilon}))
     \mres(\overline{\Omega}_{\varepsilon_0}\setminus\overline{\Omega}) )
  \end{split}
 \] 
 for~$j\in\{1, \, 2\}$. For a.e.~$y$, the first term in the right-hand side 
 converges to zero as~$\varepsilon\to 0$, because 
 $v_{j,\varepsilon}\to v_j$ in~$(L^\infty\cap W^{1,k-1})(\Omega^\prime, \, \R^m)$
 and because of~\eqref{S_flat}. As for the right-hand side, we have
 \[
  \begin{split}
   \sup_{0 <\varepsilon<\varepsilon_0} \int_{\R^d}\M(\S_y(v_{j, \varepsilon})) 
      \mres(\overline{\Omega}_{\varepsilon_0}\setminus\overline{\Omega}) ) \, \d y
   \leq C \sup_{0 <\varepsilon<\varepsilon_0}  
      \norm{\nabla v_{j, \varepsilon}}^k_{L^k(\Omega_{\varepsilon_0}\setminus\Omega)}
   \leq C
      \norm{\nabla v_j}^k_{L^k(\Gamma_{2\varepsilon_0})}
  \end{split}
 \]
 where $\Gamma_{2\varepsilon_0}:=\{x\in\R^d\colon\dist(x, \, \partial\Omega)<2\varepsilon_0\}$.
 (We have applied here Young's inequality for the convolution.) Since the right-hand side 
 converges to zero as~$\varepsilon_0\to 0$, we conclude the proof of~\eqref{cobordism1}. 
 
 We turn now to the proof of~\eqref{cobordism2}. In view of~\eqref{cobordism1},
 we can assume without loss of generality that~$u_1 = u_2$. Let~$0<\theta<1$ be fixed.
 Let~$y_1\in\R^m$ with $|y_1|\leq\theta\delta_0 = \theta\dist(\NN, \, \X)$.
 Let~$\xi\in C^{\infty}_{\mathrm{c}}(\R^m)$ be a cut-off function such that $0\leq\xi\leq 1$,
 $\xi = 0$ in a neighbourhood of~$\NN$ and~$\xi = 1$ in a neighbourhood of~$\X + y_1$, 
 and let~$\phi\colon\R^m\to\R^m$ be given by~$\phi(z) := z - y_0\,\xi(z)$
 for a fixed~$y_0\in\R^m$.
 There exists~$\delta > 0$ such that, for $|y_0|\leq\delta$, the map~$\phi$
 is a diffeomorphism. In fact, since the $C^1$-norm of~$\xi$
 can be bounded in terms of~$\theta$, $\delta_0$, we can
 choose~$\delta = \delta(\theta, \, \delta_0)$ uniformly with 
 respect to~$y_1\in B^m_{\theta\delta_0}$.
 Then, for~$|y_0|\leq\delta$ and a.e.~$y$ in a neighbourhood of~$y_1$, there holds
 \[
  \S_{y}(\phi(u_1)) = \S_{y + y_0}(u_1).
 \]
 This equality is readily checked in case~$u_1$ is smooth, and remains
 true in general by the continuity of~$\S$. Since~$\phi(u_1)$ has trace~$g$
 on~$\partial\Omega$ (because~$g$ is $\NN$-valued), we can apply~\eqref{cobordism1}
 and deduce that \eqref{cobordism2} holds, provided that $|y_1|\leq\theta\delta_0$
 and~$|y_1 - y_2|\leq\delta$. Since $B_{\theta\delta_0}^m$ 
 can be covered by finitely many balls of diameter~$\delta$, \eqref{cobordism2}
 remains true when $|y_1|\leq\theta\delta_0$, $|y_2|\leq\theta\delta_0$,
 and the proposition follows by letting~$\theta\nearrow 1$.
\end{proof}

\subsection{The topological singular set of $\NN$-valued maps}
\label{subsect:Sobolev-NN}

In this section, we study the special case of $\NN$-valued Sobolev maps.
We show that Pakzad and Rivi\`ere's construction~\cite{PakzadRiviere}
of a topological singular set
\[
 \S^{\mathrm{PR}}\colon W^{1,k-1}(\Omega, \, \NN) \to 
 \F_{d-k}(\overline{\Omega}, \, \pi_{k-1}(\NN))
\]
is essentially equivalent to~$\S$, that is, one can reconstruct the operator~$\S$ 
given~$\S^{\mathrm{PR}}$, and conversely.
As a consequence, we prove Theorem~\ref{th:PR},
which extends the results in~\cite{PakzadRiviere}.

We first recall the definition of~$\S^{\mathrm{PR}}$.
Let~$R^\infty_p(\Omega, \, \NN)$ (resp. $R^0_p(\Omega, \, \NN)$) 
be the class of maps $u\in W^{1, p}(\Omega, \, \NN)$ 
that are smooth (resp., continuous) on~$\overline{\Omega}$ 
away from the skeleton of a
polyhedral~$(d-\lfloor p \rfloor - 1)$-complex.
The set $R^\infty_p(\Omega, \, \NN)$ is dense in $W^{1,p}(\Omega, \, \NN)$
\cite[Theorem~2]{Bethuel-Density} 
(and so is, a fortiori, $R^0_p(\Omega, \, \NN)$).
Let~$u\in R^0_p(\Omega, \, \NN)$ and let~$Z$ 
be a polyhedral $(d-\lfloor p \rfloor - 1)$-complex such that~$u\in C^0(\Omega\setminus Z)$.
For each $(d-\lfloor p \rfloor - 1)$-cell~$H$ of~$Z$, we take a
$(\lfloor p \rfloor+1)$-disk~$B_H$ that intersect transversely~$H$
at a unique point~$x_H$, and do not intersect any other cell of~$Z$.
We orient~$H$ and~$B_H$ in such a way that $\T_{x_H}H \oplus \T_{x_H} B_H$
induces the standard orientation on~$\R^d$. We set
\begin{equation} \label{Stop_PR}
 \S^{\mathrm{PR}}(u) := \sum_{H} [u_*(\partial B_H)] \llbracket H\rrbracket,
\end{equation}
the sum being taken over all $(d-\lfloor p \rfloor - 1)$-cells~$H$ of~$Z$.
Pakzad and Rivi\`ere~\cite[Theorem~II]{PakzadRiviere} showed that,
in case~$\Omega = B^d$ and~$p\in[1, \, 2)\cup [d-1, \, d)$, 
the map~$\S^{\mathrm{PR}}$ can be extended continuously to~$W^{1, k-1}(B^d,\, \NN)$.

\begin{lemma} \label{lemma:Stop_PR}
 Let~$\delta_0 := \dist(\NN, \, \X)$. For any~$u\in R^0_{k-1}(\Omega, \, \NN)$
 and a.e.~$y\in\R^m$ with~$|y|<\delta_0$, there holds $\S_y(u) = \S^{\mathrm{PR}}(u)$.
\end{lemma}
\begin{proof} 
 Choose a number $0 < \delta < \delta_0$,
 and pick a function $u\in R^0_{k-1}(\Omega, \, \NN)$.
 By reflection (see e.g. \cite[Lemma~8.1]{ABO2}), we can extend~$u$ to a new map
 defined on a slightly larger domain~$\Omega^\prime\supset\!\supset\Omega$
 that retracts onto~$\overline{\Omega}$, in such a way
 that~$u\in W^{1, k-1}(\Omega^\prime, \, \NN)$.
 Let~$\rho_\varepsilon$ be a standard mollifier supported in~$B_\varepsilon^d$,
 and let~$u_\varepsilon := u*\rho_\varepsilon$. For any~$0<\varepsilon<\dist(\Omega,
 \, \partial\Omega^\prime)$, $u_\varepsilon$ is a well-defined map 
 in~$C^\infty(\overline{\Omega}, \, \R^m)$.
 Let~$Z$ be a polyhedral $(d - k)$-complex such that
 $u\in C^0(\Omega\setminus Z)$, and for any~$\eta >0$,
 let~$V_\eta$ be the closed~$\eta$-neighbourhood of~$Z$.
 Since~$u$ is $\NN$-valued and uniformly continuous on~$\overline{\Omega}\setminus V_\eta$, for
 $\varepsilon$ small enough and any~$x\in\overline{\Omega}\setminus V_\eta$ 
 we have~$\dist(u_\varepsilon(x), \, \NN) < \dist(\NN, \, \X) - \delta$.
 Thus, $\S_y(u_\varepsilon)\mres(\overline{\Omega}\setminus V_\eta) = 0$
 for any~$y$ such that~$|y|\leq \delta$. 
 Taking the limit as~$\varepsilon\to 0$ with the help of~\eqref{S_flat},
 and using that the flat-convergence preserves the support, we conclude that
 \begin{equation*} \label{PR-spt}
  \spt(\S_y(u)) \subseteq \bigcap_{\eta > 0} V_\eta = Z \qquad 
  \textrm{for any } y \textrm{ with } |y|\leq\delta.
 \end{equation*}
 Moreover, $\S_y(u)$ is a cycle relative to~$\Omega$, being the flat limit of
 the relative cycles~$\S_y(u_\varepsilon)$. Therefore, the constancy 
 theorem~\cite[Theorem~7.1]{DePauwHardt2} implies that, for any open~$(d-k)$-cell~$H$ of~$Z$,
 there exists $\alpha(H)\in\pi_{k-1}(\NN)$ such that 
 $\S_y(u)\mres H = \alpha(H)\llbracket H \rrbracket$. In fact, we also have
 \[
  \S_y(u) = \sum_H \alpha(H)\llbracket H \rrbracket,
 \]
 because no non-trivial $(d-k)$-chain can be supported on the $(d-k-1)$-skeleton of~$Z$
 \cite[Theorem~3.1]{White-Deformation}. Finally, let~$B_H$ be a closed $k$-disk
 that intersects transversely~$H$ at a single point, and does not intersect any other cell of~$Z$.
 Arguing as above, we see that~$\spt(\S_y(u_\varepsilon))\cap\partial B_H=\emptyset$ 
 for any~$y$ such that~$|y|\leq\delta$ and for
 $\varepsilon$ small enough. Therefore, using the stability of~$\I$ with respect to flat
 convergence (Lemma~\ref{lemma:intersection_product}) and \eqref{S_intersection},
 we conclude that
 \[
  \alpha(H) = \I(\S_y(u), \, \llbracket B_H\rrbracket)
  = \I(\S_y(u_\varepsilon), \, \llbracket B_H\rrbracket) 
  = [\RR\circ(u - y)_*(\partial B_H)]. 
 \]
 Now, when $|y|\leq\delta<\dist(\NN, \, \X)$,
 the map~$z\in\NN\mapsto\RR(z - y)$ 
 is homotopic to the identity on~$\NN$; a homotopy is given by 
 $(t, \, z)\in[0, \, 1]\times\NN\mapsto\RR(z - ty)$. Therefore, 
 we have $[\RR\circ(u - y)_*(\partial B_H)] = [u_*(\partial B_H)]$ and hence
 $\S_y(u) = \S^{\mathrm{PR}}(u)$ for a.e.-$y$ with~$|y|\leq\delta$. 
 By letting~$\delta\nearrow\delta_0$, the lemma follows.
\end{proof}

\begin{remark} \label{remark:PR_Th}
 Note that, in the proof of Lemma~\ref{lemma:Stop_PR}, we only need to apply 
 Property~\eqref{S_intersection} to smooth maps, so Lemma~\ref{lemma:Stop_PR}
 only relies on the results in Section~\ref{subsect:smooth} and the continuity of~$\S$.
\end{remark}

\begin{proof}[Proof of Lemma~\ref{lemma:Stop-delta}]
 If~$u\colon\Omega^\prime\to\R^m$ is smooth then, for
 a.e.~$y\in\R^m$, there holds $\RR\circ(u - y)\in R^{1,k-1}(\Omega, \, \NN)$. Thus, 
 Lemma~\ref{lemma:Stop_PR} (see also Remark~\ref{remark:PR_Th})
 and the very definition of~$\S_y(u)$ imply 
 \[
  \S_{y^\prime}(\RR\circ(u - y)) = \S^{\mathrm{PR}}(\RR\circ(u - y)) = \S_y(u) 
 \]
 for a.e.~$y^\prime$ with~$|y^\prime|<\delta_0$.
\end{proof}

\begin{proof}[Proof of Proposition~\ref{prop:N-valued}]
 In case~$u\in W^{1, 1-k}(\Omega, \, \NN)$, the statement follows immediately 
 from Lemma~\ref{lemma:Stop_PR}, combined with a density argument.
 For~$g\in X^{\mathrm{bd}}$, the statement follows by taking the boundary of both sides
 of~\eqref{cobordism2}, and using Proposition~\ref{prop:boundary}. Finally, an arbitrary map
 $u\in W^{1, k}(\Omega, \, \NN)$ can be approximated (in the~$W^{1, k}$-norm)
 by maps~$\tilde{u}\colon\Omega\to\NN$ that are smooth away from the skeleton of a
 smooth complex of dimension~$d - k - 1$. By Lemma~\ref{lemma:Stop_PR}, for a.e.~$y$ 
 with $|y|<\delta_0$ we have $\S_{y}(\tilde{u})=\S^{\mathrm{PR}}(\tilde{u})$, 
 and the latter must be zero because no non-trivial, smooth $(d-k)$-chain can be supported 
 on a~$(d-k-1)$-dimensional set. The proposition follows by a density argument.
\end{proof}

We conclude this section by giving the proof of Theorem~\ref{th:PR}.

\begin{proof}[Proof of Theorem~\ref{th:PR}]
 For any~$u\in R^{1,p}(B^d, \, \NN)$, $\S^{\mathrm{PR}}(u)$ is defined by~\eqref{Stop_PR}, 
 as in~\cite{PakzadRiviere}. 
 For any two maps~$u_0, \, u_1\in R^{1,p}(B^d, \, \NN)$,
 Lemma~\ref{lemma:Stop_PR} and~\eqref{S_flat}
 (with the choice~$k = \lfloor p\rfloor + 1$) imply that
 \[
  \F\left(\S^{\mathrm{PR}}(u_1) - \S^{\mathrm{PR}}(u_0)\right) 
  \leq C \int_{B^d} \abs{u_1 - u_0} \left(\abs{\nabla u_0}^{\lfloor p \rfloor} 
  + \abs{\nabla u_1}^{\lfloor p \rfloor} + 1\right)
 \]
 for some constant~$C = C(\NN, \, \X, \, p)$. Then, by applying Lebesgue dominated theorem
 to the right-hand side of this inequality,
 we deduce that~$\S^{\mathrm{PR}}$ maps Cauchy sequences in~$R^{1,p}(B^d, \, \NN)$ 
 into Cauchy sequences 
 in $\F_{d-\lfloor p\rfloor-1}(\overline{B}^d; \, \pi_{\lfloor p\rfloor}(\NN))$.
 Thus, $\S^{\mathrm{PR}}$ admits a continuous extension to~$W^{1,p}(B^d, \, \NN)$.
 Now, the theorem follows by the same arguments of~\cite[Theorem~II]{PakzadRiviere}.
\end{proof}

\section{Applications to $\NN$-valued BV spaces}
\label{sect:BV}

\subsection{Density of smooth, $\NN$-valued maps in BV}
\label{subsect:BV-density}

In this section, we consider the space~$\BV(\Omega, \, \R^m)$, consisting of
functions~$u\in L^1(\Omega, \, \R^m)$ whose distributional derivative~$\D u$ is a
finite Radon measure, endowed with the norm
$\|u\|_{\BV(\Omega)} := \|u\|_{L^1(\Omega)} + |\D u|(\Omega)$.
We also consider the semi-norm $|u|_{\BV(\Omega)} := |\D u|(\Omega)$.
The distributional derivative of a $\BV$-function has the following representation:
\[
 \D u = \nabla u \, \L^d + \D^{\mathrm{c}} u + \D^{\mathrm{j}} u,
\]
where~$\nabla u$ is called the approximate gradient of~$u$, $\D^c u$ and~$\D^j u$ are,
respectively, the Cantor and the jump part.
The latter is supported on a~$(d-1)$-rectifiable set $\J_u$, called the jump set,
and we have
\[
 \D^{\mathrm{j}}u = (u^+ - u^-)\otimes\nnu_{u} \H^{d-1} \mres \J_u
\]
where~$\nnu_u$ is the approximate unit normal to~$\J_u$ and~$u^+$, $u^-$ are the approximate
traces of~$u$ from either side of~$\J_u$. We define~$\SBV(\Omega, \R^m)$ as the set
of all functions~$u\in\BV(\Omega, \, \R^m)$ such that~$\D^c u = 0$.
We refer the reader, e.g., to~\cite{AmbrosioFuscoPallara} for more details and notation.
We define $\BV(\Omega, \, \NN)$ (resp., $\SBV(\Omega, \, \NN)$)
as the set of maps~$u\in\BV(\Omega, \, \R^m)$ (resp., $u\in\SBV(\Omega, \, \R^m)$)
such that~$u(x)\in\NN$ for a.e.~$x\in\Omega$.
We say that a sequence~$u_j$ of~$\BV$-functions converges weakly to~$u$ if and only if~$u_j\to u$
strongly in~$L^1$ and~$\D u_j \rightharpoonup^* \D u$ weakly$^*$ as elements of the
dual~$C_0(\Omega, \, \R^m)^\prime$. 

\begin{proof}[Proof of Theorem~\ref{th:BV_density}]
 Let~$u_j\in C^\infty(B^d, \, \R^m)$ be a
 sequence of smooth maps that converges to~$u$ weakly in~$\BV$ and a.e.,
 with $\|\nabla u_j\|_{L^1(B^d)} \leq C |\D u|(B^d)$
 (see e.g.~\cite[Theorem~3.5]{AmbrosioFuscoPallara}). Since~$\NN$ is compact,
 by a truncation argument we can make sure that $\|u\|_{L^\infty(B^d)}\leq\Lambda$, 
 for some constant~$\Lambda$ that only depends on the embedding of~$\NN$ in~$\R^m$.
 Since~$\NN$ is connected and~$\pi_1(\NN)$ is abelian, we can apply Theorem~\ref{th:Stop}
 to~$u_j$ with~$k=2$. In particular, by~\eqref{S_boundary}, for any~$j\in\N$  
 and a.e.~$y\in B^m_{\delta_0}$ (where $\delta_0 := \dist(\NN, \, \X)$)
 there exists a $(d-1)$-chain~$R_y^j$
 such that $(\partial R_y^j - \S_y(u_j))\mres B^d = 0$ and
 \[
  \int_{B^m_{\delta_0}} \M(R^j_y) \, \d y \leq 
  C\norm{\nabla u_j}_{L^1(B^d)} \leq C|\D u|(B^d) . 
 \]
 Moreover, by Lemma~\ref{lemma:projection_cont} we have
 \[
  \int_{B^m_{\delta_0}}\left(\int_{B^d} 
  \abs{\nabla(\RR\circ(u_j - y))(x)} \,\d x\right)\,\d y 
  \leq C\norm{\nabla u_j}_{L^1(B^d)} \leq C|\D u|(B^d) .
 \]
 By an average argument we deduce that, 
 for each~$j\in\N$ and~$\delta\in(0, \, \delta_0)$,
 there exists~$y(j)\in B^m_{\delta}$ such that
 \begin{equation*} \label{BV_control}
  \norm{\nabla(\RR\circ(u_j - y(j)))}_{L^1(B^d)} \leq C |\D u|(\Omega), \qquad 
  \M(R^j_{y(j)}) \leq C |\D u|(B^d),
 \end{equation*}
 for some constant~$C$ that depends on~$\delta$.
 By choosing~$\delta$ small enough, 
 we can make sure that the map $\RR_y\colon z\in\NN \mapsto\RR(z - y)$ has a smooth
 inverse~$\RR_y^{-1}\colon\NN\to\NN$ for any~$y\in B^m_\delta$.
 
 We set~$w_j := (\RR_{y(j)}^{-1}\circ\RR)(u_j - y(j))$. Then, $w_j\in R^{1,1}(B^d, \, \NN)$
 and the $L^1$-norm of~$\nabla w_j$ is bounded by the total variation of~$\D u$.
 By applying the ``removal of the singularity'' 
 technique in~\cite[Proposition~5.1]{PakzadRiviere}
 we find a map~$v_j\in C^\infty(B^d, \, \NN)$ such that
 \begin{gather} 
  \norm{v_j - w_j}_{L^1(B^d)}\leq j^{-1}, \label{w-v1} \\
  \norm{\nabla v_j}_{L^1(B^d)}\leq \norm{\nabla w_j}_{L^1(B^d)} 
       + C\M(R^j_{y(j)}) + Cj^{-1} \leq C \abs{\D u}(B^d). \label{v-w2}
 \end{gather}
 Now, $(w_j)_{j\in\N}$ is bounded in the~$\BV$-norm and hence, modulo extraction of a
 subsequence, $w_j$ converges weakly in $\BV$ and a.e. to some limit~$w\in\BV(B^d, \, \R^m)$.
 On the other hand, it can be easily checked that, up to subsequences, $v_j$ converges
 to~$u$ a.e. out of the $\H^d$-negligible set $\cup_j \spt R^j_{y(j)}$,
 and hence by~\eqref{w-v1} we have~$w=u$.
\end{proof}

\subsection{Lifting results in BV}
\label{subsect:BV-lifting}

In this section, we consider the lifting problem in~$\BV$. Let~$\pi\colon\EE\to\NN$
be the universal covering of~$\NN$. We endow~$\EE$ with the pull back metric~$\pi^*(h_{\NN})$,
$h_{\NN}$ being the metric of~$\NN$, so that~$\pi$ is a local isometry.
We also identify~$\EE$ with an isometrically embedded submanifold of
some Euclidean space~$\R^\ell$, and we define~$\BV(\Omega, \, \EE)$ as the set
of functions~$u\in\BV(\Omega, \, \R^{\ell})$ such that~$u(x)\in\EE$ for a.e.~$x\in\Omega$.
We say that~$v\in\BV(\Omega, \, \EE)$ is a \emph{lifting}
for $u\in\BV(\Omega,\, \NN)$ if $u = \pi\circ v$ a.e. on~$\Omega$.
When the domain is a ball, the existence of a lifting in~$\BV$ could be deduced by 
a density argument, based on Theorem~\ref{th:BV_density},
but we give below a different proof which works on more general domains.

\begin{proof}[Proof of Theorem~\ref{th:BV_lifting}]
 We choose a norm~$|\cdot|$ on~$\pi_1(\NN)$ that, in addition to~\eqref{group_norm}, satisfies
 \begin{equation} \label{group_norm_length}
  \inf \left\{\int_{\SS^1} \abs{\gamma^\prime(s)}\d s\colon 
  \gamma\in g\cap W^{1,1}(\SS^1, \, \NN)\right\} \leq C \abs{g}
 \end{equation}
 for any~$g\in\pi_1(\NN)$ and some $g$-independent constant~$C$. Such a norm exists. 
 Indeed, the left-hand side itself
 of~\eqref{group_norm_length} defines a norm on~$\pi_1(\NN)$ that satisfies~\eqref{group_norm}
 up to a multiplicative factor, as any loop whose length is less than the injectivity radius
 of~$\NN$ is contained in a contractible geodesic ball.
 
 We also need to fix some notation. 
 Given a smooth chain~$R\in\M_{d-1}(\R^d; \, \pi_1(\NN))$, we can always
 assign an orientation to each~$(d-1)$-cell of~$R$. We can then write
 $R = \sum_i g_i \llbracket H_i\rrbracket$, where the~$H_i$ are oriented, smooth 
 $(d-1)$-polyhedra with pairwise disjoint interiores and $g_i\in\pi_1(\NN)$.
 We denote the local multiplicity of~$R$ at a point~$x\in H_i\setminus\partial H_i$ 
 by~$\mathfrak{g}[R](x) := g_i$. Note that~$\mathfrak{g}[R]$ depends on the choice of 
 the orientation on~$H$, and $\mathfrak{g}[R](x)$ changes into~$-\mathfrak{g}[R](x)$
 when the orientation of~$H$ is flipped.
 
 \setcounter{step}{0}
 \begin{step}[Construction of an approximating sequence]
  Let~$\Omega^\prime$ be an open cube (i.e., $\Omega^\prime=(-L, \, L)^d$
  for some $L>0$) that contains~$\overline{\Omega}$.
  Let $u_\Omega := \H^d(\Omega)^{-1}\int_\Omega u\in\R^m$ be the average of~$u$ over~$\Omega$.
  Thanks to \cite[Proposition~3.21]{AmbrosioFuscoPallara} and the
  BV-Poincar\'e inequality \cite[Theorem~3.44]{AmbrosioFuscoPallara},
  we can extend~$u - u_\Omega$
  to a map~$u_*\in(L^\infty\cap\BV)(\Omega^\prime, \, \R^m)$ 
  that satisfies $|u_*|_{\BV(\Omega^\prime)}\leq C|u|_{\BV(\Omega)}$
  for some constant~$C = C(\Omega)$.   
  Thus, redefining $u := u_* + u_\Omega$, we have constructed an extension of the
  given map~$u$ that belong to~$(L^\infty\cap\BV)(\Omega^\prime, \, \R^m)$ 
  and satisfies $|u|_{\BV(\Omega^\prime)}\leq C|u|_{\BV(\Omega)}$.
  
  Take a sequence of smooth functions $u_j\in C^\infty(\Omega^\prime, \, \R^m)$
  that converges to~$u$ $\BV$-weakly and a.e., and 
  is uniformly bounded in~$L^\infty$. 
  By applying Theorem~\ref{th:Stop} to~$u_j$, with the choice~$k=2$,
  and using an average argument, 
  for any~$j\in\N$ we find~$y(j)\in\R^m$ and a smooth $(d-1)$-chain
  $R_j := R^j_{y(j)}\in\M_{d-1}(\overline{\Omega^\prime};\, \pi_1(\NN))$ 
  such that the following properties are satisfied.
  We set $S_j := \S_{y(j)}(u_j)$, $Z_j := (u_j - y(j))^{-1}(\X^{m-3})$
  and~$w_j := (\RR_{y(j)}^{-1}\circ\RR)(u_j - y(j))$. Then 
  $Z_j$ is a union of submanifolds of dimension $\leq d-3$ and there holds
  \begin{gather}
   w_j\in W^{1,1}(\Omega^\prime, \, \NN) \cap 
      C^\infty(\Omega^\prime\setminus(\spt S_j\cup Z_j), \, \NN) \label{w_j} \\
   w_j \to u \quad \textrm{a.e. on } \Omega \label{ae_w_j} \\
   \norm{\nabla w_j}_{L^1(\Omega^\prime)} + \M(R_j)
      \leq C|u|_{\BV(\Omega)} \label{BV_norm} \\
  (\partial R_j - S_j)\mres\Omega^\prime = 0. \label{boundary_R}
  \end{gather}
 \end{step}
 
 \begin{step}[Construction of a lifting for~$w_j$]
  For each~$j\in\N$, we will construct a lifting~$v_j$ of~$w_j$ such that
  $v_j\in C^\infty(\Omega^\prime\setminus(\spt R_j\cup Z_j), \, \EE)$ and
  \begin{equation} \label{jump_lifting}
   v_j^+(x) = (-1)^{d}\,\mathfrak{g}[R_j](x) \cdot v^-_j(x) 
   \qquad \textrm{for } \H^{d-1} \textrm{-a.e. } x\in \spt R_j. 
  \end{equation}
  To this end, we adapt a well-known topological construction 
  (see e.g. \cite[Proposition~1.33]{Hatcher}). We choose base
  points~$x_0\in\Omega^\prime\setminus\cup_j(\spt R_j \cup Z_j)$, $n_j := w_j(x_0)$
  and~$e_j\in\pi^{-1}(n_0)$. For any~$x\in\Omega^\prime\setminus(\spt R_j\cup Z_j)$,
  we take a smooth path~$\gamma\colon [0, \, 1]\to\Omega^\prime\setminus(\spt S_j\cup Z_j)$
  from~$x_0$ to~$x$. (Such a path exists, by transversality reasons.)
  We suppose that~$\gamma$ crosses transversely each cell of~$R_j$, 
  which is generically the case, by Thom's transversality theorem. In particular,
  there exists finitely many~$t_i\in (0, \, 1)$ such that~$\gamma(t_i)\in\spt R_j$;
  moreover, each~$\gamma(t_i)$ lies in the interior of a~$(d-1)$-cell.
  We define~$g_i\in\pi_{1}(\NN)$ by
  \begin{equation} \label{multiplicities}
   g_i := \begin{cases}
           (-1)^d\,\mathfrak{g}[R_j](\gamma(t_i))  & 
              \textrm{if } \gamma^\prime(t_i) \textrm{ agrees with the orientation of } R_j \\
           (-1)^{d-1}\,\mathfrak{g}[R_j](\gamma(t_i)) & \textrm{otherwise. } 
          \end{cases}
  \end{equation}
  We define a path~$\alpha\colon [0, \, 1]\to\EE$ in the following way:
  on the interval~$[0, \, t_1)$, $\alpha$ is the lifting of~$w_j\circ\gamma_{x|[0, t_1]}$
  starting from the point~$e_j$; on~$[t_1, \, t_2)$, $\alpha$ is the lifting 
  of~$w_j\circ\gamma_{x|[t_1, t_2]}$ starting from~$g_1\cdot \alpha(t_1^{-})$, and so on.
  Note that~$\alpha$ is uniquely defined by~$\gamma$. Then, we set~$v_j(x) := \alpha(1)\in\EE$.
  
  We need to check that~$v_j$ is well-defined. Let $\gamma$, $\eta$ be two paths from~$x_0$
  to~$x$ as above, and let~$\alpha$, $\beta$ be the corresponding paths in~$\EE$
  obtained via the previous construction. Let~$g_1, \, \ldots, \, g_p$, resp. 
  $h_1, \, \ldots , \, h_q$, be the elements of~$\pi_1(\NN)$ 
  associated with~$\gamma$, resp.~$\eta$, via~\eqref{multiplicities}. We denote 
  by~$\gamma*\bar{\eta}$ 
  the loop obtained by first travelling along~$\gamma$ 
  then along~$\eta$, 
  the opposite way from~$x$ to~$x_0$. Since
  $\Omega^\prime$ is a cube, hence a simply connected set, $\gamma*\bar{\eta}$ can be seen 
  as the boundary of a smooth chain~$T\in\M_2(\Omega^\prime; \, \Z)$. By definition 
  of the~$g_i$'s and $h_k$'s and by Lemma~\ref{lemma:intersection_product}, we have
  \begin{equation} \label{gh}
   - \sum_{i=1}^p g_i + \sum_{k=1}^q h_k = (-1)^{d-1}\I(R_j, \, \partial T)
   = \I (\partial R_j, \, T) \stackrel{\eqref{boundary_R}}{=} 
   \I (S_j, \, T) \stackrel{\eqref{S_intersection}}{=} [w_{j,*}(\partial T)].
  \end{equation}
  Let~$\sigma\colon[0, \, 1]\to\EE$, resp. $\tau\colon[0, \, 1]\to\EE$, be liftings for
  $w_j\circ\gamma$, resp. $w_j\circ\eta$, with~$\sigma(0)=\tau(0)=e_j$.
  Then, by construction of~$\alpha$, $\beta$, we have
  \[
   \alpha(1) = \sum_{i=1}^p g_i \cdot\sigma(1), \qquad \beta(1) = \sum_{k=1}^q h_k \cdot\tau(1)
  \]
  and $\sigma(1) = [w_{j,*}(\partial T)]\cdot\tau(1)$. From these identities and~\eqref{gh},
  it follows that $\alpha(1) = \beta(1)$, so $v_j(x)$ is well-defined.
  Now, arguing exactly as in
  \cite[Proposition~1.33]{Hatcher}, one sees that~$v_j$ is smooth
  on~$\Omega^\prime\setminus(\spt R_j\cup Z_j)$, 
  and~\eqref{jump_lifting} is satisfied by construction.
 \end{step}
 
 \begin{step}[Passage to the limit]
  Since~$\pi$ is a local isometry and~$w_j = \pi\circ v_j$, we have that
  $|\nabla v_j| = |\nabla w_j|$ on~$\Omega^\prime\setminus(\spt R_j\cup Z_j)$; moreover, 
  for any~$y\in\EE$ and~$g\in\pi_{1}(\NN)$ there holds
  \[
   |y - g \cdot y| \leq \dist_{\EE}(y, \, g\cdot y) = \inf_{\gamma\in g} 
   \int_{\S^1}\abs{\gamma^\prime(s)}\d s \stackrel{\eqref{group_norm_length}}{\leq} C\abs{g} \!,
  \]
  where~$\dist_{\EE}$ denotes the geodesic distance in~$\EE$. Together
  with~\eqref{jump_lifting}, this yields $|\D^{\mathrm{j}}v_j|(\Omega^\prime)\leq\M(R_j)$
  and hence, by \eqref{BV_norm},
  \begin{equation} \label{BV_seminorm_bdd}
   \abs{v_j}_{\BV(\Omega^\prime)} \leq C\abs{u}_{\BV(\Omega)}.
  \end{equation}
  Now, thanks to the BV-Poincar\'e-type inequality \cite[Lemma~6, Eq.~(16)]{Chiron-trace}, 
  for each~$j$ we find~$\xi_j\in\EE$ such that
  \begin{equation} \label{Poincare}
   \int_{\Omega^\prime} \dist_{\EE}(v_j(x), \, \xi_j) \, \d x \leq
   C\abs{v_j}_{\BV(\Omega^\prime)}.
  \end{equation}
  Since the group~$\pi_1(\NN)$ acts isometrically 
  on~$\EE$, and since~$\EE$ admits a cover of the form~$\{g\cdot U\}_{g\in\pi_1(\NN)}$
  where~$U\subseteq\EE$ is bounded, by multiplying each~$v_j$ by a suitable element 
  of~$\pi_1(\NN)$ we can assume without loss of generality that the~$\xi_j$'s are uniformly bounded.
  Then, \eqref{BV_seminorm_bdd} and~\eqref{Poincare}
  imply that~$({v_j}_{|\Omega})_{j\in\N}$ is bounded in~$\BV$. 
  We extract a subsequence that converges
  $\BV$-weakly and a.e. to a limit~$v\in\BV(\Omega, \, \EE)$; by~\eqref{ae_w_j}
  and~\eqref{BV_seminorm_bdd}, $v$ is a lifting of~$u$ with the desired properties.
 \end{step}
 
 \begin{step}[The case $u\in\SBV$]
  Let~$\iota$ be the canonical embedding~$\R^m\to\R^{m}\times\R^{2\ell}$.
  We first construct a smooth immersion~$\tilde{\pi}\colon \R^{\ell}\to\R^{m+2\ell}$ that 
  restricts to~$\iota\circ\pi$ on~$\EE\subseteq\R^{\ell}$.
  We consider a tubular neighbourhood~$U$ of~$\EE$ together with the nearest-point
  projection~$\tau\colon U\to\EE$, which is well-defined and smooth. 
  We take smooth cut-off functions~$\xi_0$, $\xi_1$ such that~$\xi_0 = 0$ and~$\xi_1 = 1$
  in a neighbourhood of~$\EE$, $\spt(\xi_1)\subseteq U$ and $\spt(1 - \xi_0)$ is contained 
  in the interior of~$\xi_1^{-1}(1)$ (so that, for any~$x\in\R^{\ell}$, either~$\xi_0$ 
  or~$\xi_1$ is equal to~$1$ in a neighbourhood of~$x$). We set
  \[
   \tilde{\pi} (x) := \left(\xi_1(x)\pi(\tau(x)), \, \xi_1(x)(x - \tau(x)), \, \xi_0(x) x\right)
   \qquad \textrm{for } x\in\R^\ell.
  \]
  Using the fact that~$\pi\colon\EE\to\NN$ is a local isometry, and in particular an immersion, 
  it can be checked that~$\tilde{\pi}$ has injective differential at any point; 
  moreover, $\tilde{\pi}_{|\EE} = \iota\circ\pi$. Take now a map~$u\in\SBV(\Omega, \, \NN)$
  and a lifting~$v\in\BV(\Omega, \, \EE)$.
  Then~$\tilde{\pi}\circ v = \iota\circ u\in\SBV(\Omega, \, \R^{m+2\ell})$
  and hence the chain rule for $\BV$-functions \cite[Theorem~3.96]{AmbrosioFuscoPallara} implies
  $\nabla\tilde{\pi}(\bar{v})\D^{\mathrm{c}}v = \D^{\mathrm{c}}(\iota\circ u) = 0$,
  where~$\bar{v}$ is the precise representative of~$v$
  (see, e.g., \cite[Corollary~3.80]{AmbrosioFuscoPallara}). Since~$\nabla\tilde{\pi}(y)$ 
  is injective for any~$y\in\R^{m+2\ell}$, we conclude that~$\D^{\mathrm{c}}v =0$,
  that is, $v\in\SBV(\Omega, \, \EE)$. \qedhere
 \end{step}
\end{proof}

\section*{Acknowledgements}
 The authors are grateful to the anonimous referees for their useful comments.
 G.C.'s  research  was supported 
 by the European Research Council under the
 European Union's Seventh Framework Programme (FP7/2007-2013)/ERC grant agreement n° 291053;
 by the Basque Government through the BERC 2018-2021 program; 
 by the Spanish Ministry of Science, Innovation and Universities:
 BCAM Severo Ochoa accreditation SEV-2017-0718; and by the
 Spanish Ministry of Economy and Competitiveness: MTM2017-82184-R.
 G. O. was partially supported by GNAMPA-INdAM.


\bibliographystyle{plain}
\bibliography{singular_set}

\end{document}